\documentclass[12pt,letterpaper]{amsart}

\usepackage{graphicx}
\usepackage[utf8]{inputenc}
\usepackage{amssymb}
\usepackage{epstopdf}
\usepackage{amsmath}
\usepackage{mathtools}
\usepackage{mathrsfs}
\usepackage{empheq}

\newtheorem{theorem}{Theorem}[section]
\newtheorem{lemma}[theorem]{Lemma}
\newtheorem{prop}[theorem]{Proposition}

\newtheorem{example}[theorem]{Example}

\numberwithin{equation}{section}
\def\sech{\mathrm{sech}}
\def\ifl{\iffalse}

\def\R{{\mathbb R}}

\def\bc{\begin{center}}

\def\ec{\end{center}}
\def\be{\begin{equation}}
\def\ee{\end{equation}}
\def\ba{\begin{array}}
\def\ea{\end{array}}
\def\bea{\begin{eqnarray}}
\def\eea{\end{eqnarray}}
\def\beaa{\begin{eqnarray*}}
\def\eeaa{\end{eqnarray*}}

\def\lb{\label}

\def\x#1{(\ref{#1})}

\begin{document}
\title[Large-Amplitude steady solitary water waves ]
{Large-Amplitude steady solitary water waves with general vorticity}

\author[]{Jifeng Chu$^{1}$ \quad Zihao Wang$^2$ \quad Yong Zhang$^{*3}$}

\address{$^1$  School of Mathematical, Hangzhou Normal University, Hangzhou 311121, China}
\address{$^2$  Department of Mathematics, Shanghai Normal University, Shanghai 200234, China}
\address{$^3$  School of Mathematical Sciences, Jiangsu University, Zhenjiang 212013, China}

\email{jifengchu@126.com (J. Chu)}
\email{wongtzuhao@163.com (Z. Wang)}
\email{18842629891@163.com (Y. Zhang)}

\thanks{Corresponding author: Yong Zhang$^{*3}$, E-mail address: 18842629891@163.com. \\
Jifeng Chu was supported by the National Natural Science Foundation of China (Grant
No. 12571168), the Science and Technology Innovation Plan of Shanghai (Grant No. 23JC1403200) and
Zhejiang Provincial Natural Science Foundation of China (No. LZ26A010006). Yong Zhang was supported by National Natural Science Foundation of China (No. 12301133), the Postdoctoral Science Foundation of China (No. 2023M741441, No. 2024T170353) and Jiangsu Education Department (No.
23KJB110007).}


\begin{abstract}
We study two-dimensional steady solitary gravity water waves with general vorticity, allowing for overhanging free-surface profiles. The main challenges arise from the free boundary, the unbounded fluid domain, and the inherent complexity of the general vorticity setting. To address these, we introduce a conformal reformulation that reduces the problem to an equivalent system on a fixed strip, consisting of an overdetermined elliptic problem coupled with an elliptic boundary-value problem. This framework enables a local analysis without imposing restrictive assumptions on the vorticity. {Using the center-manifold construction of Chen, Walsh and Wheeler \cite{chennonlinearity},} we first establish the existence of small-amplitude solitary waves for such non-trivial vorticity distributions. Subsequently, via a global analytic bifurcation argument, we prove the existence of continuous branches of large-amplitude solitary waves for analytic vorticity functions. Along these global solution curves, the free surfaces may develop overhanging profiles and need not remain graphs over the horizontal coordinate. Our results provide a new constructive framework for large-amplitude solitary waves with general vorticity, extending the existing theory beyond the classical settings of constant vorticity \cite{susannaarma} and non-overhanging profiles \cite{milesjma}.
\end{abstract}

\maketitle

\section{Introduction}

The study of water waves has a long history and the rigorous mathematical analysis can be traced back
to the 18th century. In particular, Stokes \cite{sto} studied the water wave problem in infinite
depth and conjectured that there is a branch of nontrivial periodic waves bifurcating from a half-space and
limiting to an extrme. Gerstner \cite{ger} found an explicit solution in the Lagrangian setting for deep water gravity waves.
Although the exact solutions obtained by Gerstner are applicable to rotational water waves, most results at that time are
related to irrotational settings, which correspond to that the curl of the velocity field vanishes. For
 irrotational water waves, a big advantage is that we can reformulate the problem into a problem for a harmonic
 function in a fixed domain. In \cite{friedrichs}, Friedrichs and Hyers proved the existence of solitary waves
 in exponentially weighted spaces, based on an application of inverse function theorem. Mielke \cite{Mielke}
 developed a spatial dynamical approach, which involves looking for a center manifold for quasilinear partial differential
 equations, to the irrotational steady waves. For large-amplitude irrotational
 water waves, the first rigorous result was proved by Amick and Toland in \cite{amickarma}, in which they used a
 conformal mapping approach and constructed a connected set of solitary waves by taking the limit of a sequence of
 approximate problems with better compactness properties. The limiting wave is singular and has a stagnation point
 with $120^{\rm o}$  angle at the crest \cite{amickacta}, which thus solved the famous Stokes conjecture.

Compared with irrotational water waves, rotational waves are more commonly observed in nature. Rotational water
waves are the waves with vorticity, which is the hallmark of underlying non-uniform currents. The study of rotational
waves is relatively later than that of irrotational settings. In \cite{ter1961}, Ter-Krikorov used the methods
of \cite{friedrichs} to prove the first existence result of small-amplitude waves. Groves and Wahl\'{e}n used a
centre-manifold technique in \cite{gw} to reduce the system of solitary water waves with an arbitrary distribution
of vorticity to a locally equivalent Hamiltonian system with one degree of freedom and then obtained the existence
result. Hur proved the symmetry property for rotational solitary waves in \cite{hurmrl}. For large-amplitude
rotational waves, the first rigorous mathematical result was provided by Constantin and Strauss \cite{constantincpam04},
in which they used an appropriate hodograph change of variable devised by Dubeil-Jacotin \cite{dj} to transform the problem into an equivalent form of a quasilinear
elliptic equation in a fixed rectangular domain, and then used the method of bifurcation theory \cite{cr}
to construct a global connected set of periodic solutions. Wheeler in \cite{milesjma} provided the first construction of exact solitary waves of large amplitude with an arbitrary distribution of vorticity, and the results can be regarded as a generalization of the classical result of Amick and Toland. We refer to the monograph \cite{c-book} and the papers \cite{bbmm, dz, hcpde} for more results on rotational periodic water waves.  Results on rotational solitary water waves in the absence of critical layers can be found in \cite{chenpoincare, hur08, dan, wjfm, warma}.
Note that all results mentioned above exclude the possibility of critical layers and overhanging wave profiles.

During the last decade, one of the more interesting and more challenging problems in the theory of steady water waves is to prove the existence of waves with critical layers where the horizontal fluid velocity in the moving frame vanishes, or the existence of overhanging waves which means that the free surface is not the graph of a function. Such waves occur even in the case of non-zero constant vorticity. In \cite{wahlenjde09}, Wahl\'{e}n constructed for the first time small-amplitude periodic waves with critical layers for the flows with non-constant vorticity,  but the waves are not overhanging. A breakthrough work was done by Constantin, Strauss and Varvaruca \cite{constantinacta}, in which
they constructed large-amplitude periodic waves with non-zero constant vorticity and the waves can be overturning. Hur and Wheeler obtained in \cite{hw} the overturning periodic waves for the flows with constant vorticity and weak gravity, by perturbing a family of explicit solutions with zero gravity. However, the study of solitary overhanging waves is more difficult than that of the periodic counterparts, because the unboundedness of the surface leads to compactness issues. Based on the center manifold reduction method and analytic global bifurcation theory, Haziot and Wheeler in \cite{susannaarma} constructed continuous curves of large-amplitude solutions for flows with constant vorticity, which allows for overhanging waves profiles. The proof relies on a novel reformulation of the problem as an elliptic system for two scalar functions in a fixed domain, one describing the conformal map of the fluid region and the other the flow beneath the wave. The first construction of overhanging solitary gravity water waves having the approximate form of a disk joined to a strip by a thin neck was proved in \cite{dpmw} very recently by assuming that the vorticity is constant and the dimensionless gravitational constant $g>0$ is sufficiently small. Note that the method developed in \cite{dpmw} can be applied to other overdetermined elliptic problems.
See \cite{av, cvw, csv, constantinarma11,  hcmp, kkl, kl, Var, wahlenduke, ww} for more results on rotational water waves in the presence of critical layers or overhanging waves.

In this paper, we continue this topic and aim to construct the existence of large-amplitude solitary gravity waves which can be overhanging. Different from \cite{susannaarma}, we consider the settings allowing for general vorticity distributions.
To achieve the goal, we need to overcome two principal difficulties: the lack of compactness stemming from the unbounded domain, and the added complexity arising from a general vorticity distribution. To resolve these issues, we adopt a new functional formulation and apply a singular bifurcation argument introduced by Chen, Walsh and Wheeler in \cite{chenpoincare,chennonlinearity}.
The rest part of this paper is organized as follows. In section 2, we present the governing equations of our problem and reformulate the problem into an equivalent system in term of conformal variables. Then we further reformulate the problem into an operator form which is strongly related to an elliptic system. In section 3, we present the linearization of the operator along  the trivial solution and prove the functional properties of the linearized operator. In section 4, we construct the small-amplitude solitary water waves by a recently derived center manifold reduction theorem. In section 5, we give the velocity in the conformal variables, and prove that the solutions satisfy a monotonicity property which is referred to as the nodal property. In section 6, we construct an invariant known as the flow force, and apply its invariance to show the non-existence of bore solution, which implies the compactness properties.  Finally in section 7, we construct large-amplitude solution and prove our main result.

To finish this section, we introduce some notations used in this paper. Let $\Omega$ be a connected, open, possibly unbounded subset of $\R^n$. We say that $\varphi\in C_c^\infty(\bar\Omega)$ if $\varphi\in C^\infty$ and the support of $\varphi$ is a compact subset of $\bar\Omega$. For ${\beta\in(0,1)}$ and $k\in\mathbb{N}$ we denote by $C^{k+\beta}(\bar\Omega)$ the space of functions whose partial derivatives up to order $k$ are H\"older continuous in $\bar\Omega$ with exponent $\beta$. We say that $u_n\to u$ in $C_{\mathrm{loc}}^{k+\beta}(\bar\Omega)$ if $\|\varphi(u_n-u)\|_{C^{k+\beta}(\Omega)}\to0$ for all $\varphi\in C_{c}^\infty(\bar\Omega)$.
Let $C_b^{k+\beta}(\bar\Omega)$ be the Banach space of functions $u\in C^{k+\beta}(\bar\Omega)$ such that $\|u\|_{C^{k+\beta}(\Omega)}<\infty$.
When $\Omega$ is unbounded, we denote $C_0^k(\bar\Omega)\subset C_b^k(\bar\Omega)$ the closed subspace of functions whose partial derivatives up to order $k$ vanish uniformly at infinity, that is
\begin{equation*}
    C_0^k(\bar\Omega):=\bigg\{u\in C_b^k(\bar\Omega):\lim_{r\to\infty}\sup_{|x|=r}|D^j u(x)|=0\quad\mathrm{for~}0\leq j\leq k\bigg\}.
\end{equation*}
We also define the weighted H\"older spaces allowing for exponential growth in the $x_1$ direction. Let $C_{\tau}^{k+\beta}(\bar\Omega)$ be the space of $u\in C^{k+\beta}(\bar\Omega)$ with $\|u\|_{C_{\tau}^{k+\beta}(\Omega)}<\infty$, where
\begin{equation}\label{wd}
    \|u\|_{C_{\tau}^{k+\beta}}:=\sum_{|\alpha|\leq k}\|\sech(\tau x_1)\partial^\alpha u\|_{C^0(\Omega)}+\sum_{|\alpha|=k}\|\sech(\tau x_1)|\partial^\alpha u|_\beta\|_{C^0(\Omega)},
\end{equation}
with $\tau>0$ and $|u|_\beta$ being a local H\"older semi-norm
$|u|_{\beta}(x):=\sup_{|y|<1}\frac{|u(x+y)-u(x)|}{|y|^\beta}.$

\section{Governing equations and reformulations}

In this section, we present the governing equations of our problem and give the equivalent reformulations in term of conformal mapping.
\subsection{Governing equations}\par\noindent{Let} $\Omega$ {denote the unbounded fluid region in the $(X,Y)$-plane. Its lower boundary is the flat bed $Y=0$, and its upper boundary is the free surface $\mathcal S$.} We assume that the fluid  domain $\Omega$ approaches to a horizontal asymptotic depth $d$  as $|X|\to \infty$, and
$\Omega$ is the image under a conformal mapping $X+iY=\xi(x,y)+i\eta(x,y)$ of the strip
\begin{equation*}
    \mathcal{R}:=\{(x,y)\in\R^2:0<y<d\}.
\end{equation*}
The upper and lower boundaries of $\mathcal{R}$ are denoted by
\begin{equation*}
    \Gamma:=\{(x,y)\in\R^2:y=d\}\quad\mathrm{and}\quad B:=\{(x,y)\in\R^2:y=0\}.
\end{equation*}
In this frame, we are working with the following two-dimensional Eulerian governing equations in terms of the velocity $(U,V)$ and pressure $P$
\begin{subequations}
	\label{euler}
	\begin{align}
		U_X+V_Y&=0 &&\mathrm{in~}\Omega\label{euler1},\\
       UU_X+VU_Y&=-P_X&&\mathrm{in~}\Omega \label{euler2},\\
      UV_X+VV_Y&=-P_Y-g&&\mathrm{in~}\Omega,\label{euler3}
		\end{align}
\end{subequations}
coupled with the kinematic boundary conditions
\begin{align}
  U\eta_x-V\xi_x&=0\quad&&\mathrm{on~}\mathcal{S}\tag{\ref{euler}{d}}\label{euler4},\\
V&=0 &&\mathrm{on~}Y=0\tag{\ref{euler}{e}}\label{euler5},
 \end{align}
and the dynamic boundary
\begin{equation}
    P=P_{\mathrm{atm}}\quad\mathrm{on~}\mathcal{S},\tag{\ref{euler}{f}}\label{euler6}
\end{equation}
here $P_{\mathrm{atm}}$ is the  constant atmospheric pressure and $g>0$ is the gravity acceleration.
We focus on the solitary waves, which are solutions to \eqref{euler} satisfying the asymptotic conditions
 \begin{equation}
 \eta(x,d)\to d,\quad\xi(x,d)\to\pm\infty\quad\mathrm{as~}x\to\pm\infty,\tag{\ref{euler}{g}}\label{euler7}
 \end{equation}
and
\begin{equation}
    V(X,Y)\to 0,\quad U(X,Y)\to U(Y)\quad\mathrm{as~}X\to\pm\infty\tag{\ref{euler}{h}}\label{euler8}
\end{equation}
uniformly in $Y$, where $U(Y)$ is the shear flow at $X=\pm\infty$.
For the convenience of our analysis, we introduce a one-parameter family of shear flows
\begin{equation*}
    U(Y)=FU^*(Y),
\end{equation*}
where $F>0$ is the Froude number to denote a dimensionless wave speed, given by
\begin{equation*}
F^2=\frac{U^2(d)}{gd},
\end{equation*}
and $U^*$ is the fixed function, normalized so that
\begin{equation*}
    U^*(d)=\sqrt{gd}.
\end{equation*}
%Thus, we have
%\begin{equation}\label{e2.2}
 %   U(d)=FU^*(d)=F\sqrt{gd}>0.
%\end{equation}
Due to the condition \eqref{euler1}, we can introduce the \emph{stream function} defined as
\begin{equation*}
  \Psi_X=-V,\quad\Psi_Y=U.
\end{equation*}
The kinematic boundary conditions  \eqref{euler4}-\eqref{euler5} imply that $\Psi$ is constant on the free surface and on the flat bed. Thus we can normalize $\Psi$ such that
\begin{equation*}
  \Psi=m \quad \mathrm{on} \quad \mathcal{S}\quad\mathrm{and}\quad\Psi=0 \quad \mathrm{on} \quad Y=0,
\end{equation*}
where $m$ is some constant to denote the relative mass flux.

Let us define the vorticity function as
$\mathrm{curl}(U,V)=U_Y-V_X$.
Taking the curl of \eqref{euler2} and \eqref{euler3} shows that the vorticity is constant along streamlines. We prescribe the vorticity law $U_Y-V_X=-\gamma(\Psi)$, so that
\begin{equation*}
 \Delta\Psi=-\gamma(\Psi).
\end{equation*}
Moreover, it follows from the Bernoulli's law  that
\begin{equation*}
  P+\frac{1}{2}|\nabla\Psi|^2+gY-G(\Psi)=\mathrm{const}\quad\mathrm{in~}\Omega,
\end{equation*}
where $G(s)=-\int_0^s\gamma(t)dt.$  To sum up, in the $(X,Y)-$plane, $\Psi$ satisfies the following overdetermined problem
\begin{subequations}
	\label{streamori}
	\begin{align}
		\Delta\Psi&=-\gamma(\Psi) &&\mathrm{in~}\Omega\label{streamori1},\\
        \Psi&=m &&\mathrm{on~}\mathcal{S}\label{streamori2},\\
        \Psi&=0 &&\mathrm{on~}Y=0\label{streamori3},\\
        |\nabla\Psi|^2+2g(Y-d)&=Q &&\mathrm{on~}\mathcal{S},
		\end{align}
\end{subequations}
where $Q$ is the total head.
A solitary wave solution to \eqref{streamori} must satisfy the asymptotic conditions
\begin{equation*}
    \Psi_X\to 0,\quad \Psi_Y\to FU^*(Y)\quad\mathrm{as~}X\to\pm\infty.
\end{equation*}
The total head $Q$ and the mass flux $m$ can be given in terms of the Froude number by
\begin{equation*}
    Q=F^2gd\quad\mathrm{and}\quad m=F\int_0^d U^*(Y)dY.
    \end{equation*}
The vorticity function $\gamma$ is given in  terms of $U^*$ and $F$ by
\begin{equation*}
    \gamma(s)=-FU^*_Y(Y)\quad\mathrm{with} \quad s=F\int_0^Y U^*(t)dt.
\end{equation*}
In the following, we require that the general vorticity function $\gamma$ satisfies
\begin{equation}\label{unidirectional}
    {\min_{0\leq Y\leq d}U^*(Y)>0.}
\end{equation}
By introducing the following dimensionless variables,
\begin{equation*}
    (\tilde{X},\tilde{Y})=\frac{1}{d}(X,Y),\quad \tilde{\Psi}(\tilde{X},\tilde{Y})=\frac{1}{m}\Psi(X,Y),\quad\tilde{\gamma}(\tilde{\Psi})=\frac{d^2}{m}\gamma(\Psi),
\end{equation*}
we obtain the final equations of the problem (for convenience, we drop the tilde)
\begin{subequations}\label{streamfin}
\begin{align}
\Delta\Psi&=-\gamma(\Psi) &&\mathrm{in~}\Omega\label{streamfin1},\\
\Psi&=1 &&\mathrm{on~}\mathcal{S}\label{streamfin2},\\
\Psi&=0 &&\mathrm{on~}Y=0\label{streamfin3},\\
|\nabla\Psi|^2+2\alpha(Y-1)&=\mu &&\mathrm{on~}\mathcal{S},
\end{align}
\end{subequations}
where $\alpha$ is the rescaled gravity and $\mu$ is the rescaled Bernoulli constant, given as
\begin{equation}\label{alphasca}
\alpha=\frac{gd^3}{m^2}=\frac{gd^3}{F^2}\bigg(\int_0^d U^*dY\bigg)^{-2},\quad \mu=\frac{Qd^2}{m^2}=F^2\alpha.
\end{equation}
The asymptotic conditions become
\begin{equation*}
\Psi_{X}\to 0,\quad\Psi_{ Y}\to\frac{dU^*( Yd)}{\int_0^d U^*( Y) dY}\quad \mathrm{as~}X\to\pm\infty,
\end{equation*}

\subsection{The laminar solution}
Let us first consider the laminar solutions of \eqref{streamfin}, which depend merely on the $Y$-variable with the flat surface $Y=1$.
Obviously the laminar solutions solve the following boundary value problem
\begin{subequations}\label{laminar}
\begin{align}
\psi_{yy}&=-\gamma(\psi),\label{laminar1}\\
	\psi(1)&=1,\\
	\psi(0)&=0.
	\end{align}
\end{subequations}
Note that problem \x{laminar} does not admit an explicit solution for general vorticity functions.  Let $\psi(y):=\psi_{\mathrm{triv}}(y)$
be the unique solution of the following initial problem
\begin{equation} \label{trivial}
\begin{cases}
		\psi_{yy}=-\gamma(\psi),\\
	\psi(1)=1,\\
	\psi_y(1)=\frac{dU^*(d)}{\int_0^dU^*(y)dy}.
	\end{cases}
\end{equation}
Moreover, it is easy to verify that
\begin{subequations}\label{lameq}
    \begin{align}
    &\psi_{\mathrm{triv},y}(y)=\frac{dU^*(dy)}{\int_0^dU^*(y)dy},\label{lameq1}\\
    &\psi_{\mathrm{triv}}(0)=0,\quad \psi_{\mathrm{triv}}(1)=1,\label{lameq2}\\
    &\psi^2_{\mathrm{triv},y}(1)=\mu. \label{lameq3}
    \end{align}
\end{subequations}
From \eqref{unidirectional} it follows that
\begin{equation} \label{nocritical}
{\min_{0\leq y\leq1}\psi_{\mathrm{triv},y}(y)>0,\qquad
\psi_{\mathrm{triv},y}(y)=\frac{dU^*(dy)}{\int_0^dU^*(t)dt},}
\end{equation}
which implies that the laminar flow is unidirectional.

\subsection{Conformal mapping}
From now on, we use the dimensionless notations but the same labels as before to denote the domain
and its boundaries, that is,
\[\mathcal{R}:=\{(x,y)\in\R^2:0<y<1\},\quad \Gamma:=\{(x,y)\in\R^2:y=1\},\quad B:=\{(x,y)\in\R^2:y=0\}.\]
Now we define a conformal map $H$ from the domain $\mathcal{R}$ to $\Omega$ as
\[H(x+iy)=\xi(x,y)+i\eta(x,y),\]  and we normalize it by requiring that
\begin{equation*}
    \xi_x+i\eta_x\to 1\quad\mathrm{as~}x\to\infty.
\end{equation*}
The free surface now can be parameterized as
\begin{equation}\label{freesurface}
    \mathcal{S}=\{(\xi(x,1),\eta(x,1):x\in\R\}.
\end{equation}
We rewrite the stream function in conformal variables as
\begin{equation}
    \label{stripstream}
    \psi(x,y)=\Psi(\xi(x,y),\eta(x,y)).
\end{equation}
By the chain principle and the Cauchy-Riemann equation, equations \eqref{streamfin} can be rewritten as the following forms in conformal variables
 \begin{subequations}\label{stripstreameq}
\begin{align}
\Delta\psi&=-\gamma(\psi)|\nabla\eta|^2\quad&&\mathrm{in~}\mathcal{R},\label{stripstreameq1}\\
\psi&=1\quad&&\mathrm{on~}\Gamma,\label{stripstreameq2}\\
\psi&=0&&\mathrm{on~}B,\label{stripstreameq3}\\
|\psi_y|^2&=(\mu-2\alpha(\eta-1))|\nabla\eta|^2&&\mathrm{on~}\Gamma,\label{stripstreameq4}\\
\Delta\eta &=0\quad&&\mathrm{in~}\mathcal{R}\tag{\ref{stripstreameq}{e}}\label{harmoniceta},\\
\eta &=0 \quad&&\mathrm{on~}B\tag{\ref{stripstreameq}{f}}\label{etalow},
\end{align} 	
 \end{subequations}
with the asymptotical conditions
\begin{equation}
\lim_{x\to\pm\infty}\psi(x,y)=\psi_{\mathrm{triv}}(y)\quad\mathrm{and}\quad\lim_{x\to\pm\infty}\eta(x,y):=\eta_{\mathrm{triv}}(y)=y.\tag{\ref{stripstreameq}{g}}\label{streamasy}
\end{equation}
Without loss of generality, we can introduce a new function ${w \in C_b^{3+\beta}(\mathbb{R})}$ depending only the variable $x$, such that
\begin{equation}
   \eta=1+w(x) \quad \mathrm{on~}\Gamma \tag{\ref{stripstreameq}{h}}\label{zy1}.
\end{equation}

Up to now, we have reformulated the original problem (\ref{streamfin}) to the system \x{stripstreameq}, which consists of an overdetermined elliptic problem (\ref{stripstreameq1})-(\ref{stripstreameq4}) for $\psi$, an elliptic boundary value problem (\ref{harmoniceta})-(\ref{etalow}) with (\ref{zy1}) for $\eta$, and the asymptotic condition (\ref{streamasy}). To make further analysis and from the physical viewpoint, we assume that
the regularity condition
\begin{equation}
    \eta,\psi\in C_b^{3+\beta}(\overline{\mathcal{R}}),\tag{\ref{stripstreameq}{i}}
\end{equation}
and the symmetry condition
\begin{equation}
 \eta\mathrm{~and~}\psi\mathrm{~are~even~in~}x.\tag{\ref{stripstreameq}{j}}
\end{equation}
Moreover, we assume that
 \begin{equation}\label{stripk}
 \inf_{\mathcal{R}} (\mu-2\alpha(\eta-1))|\nabla\eta|^2>0,\tag{\ref{stripstreameq}{k}}
\end{equation}
which has been introduced and explained in \cite{susannaarma}. That is, the first factor not vanishing implies that we cannot have a wave of greatest
height, and the second one being nonzero ensures that $\eta$ defines the imaginary
part of a conformal mapping. Moreover, one can check that (\ref{stripstreameq}{k}) holds whenever there are no stagnation points on the
free surface.

\subsection{Equivalent systems}
Because we are considering the solitary waves, we cannot adopt the reformulation presented in \cite{susannaarma} and \cite{wahlenduke}.

{
Now we develop a formulation in which the surface trace has the same regularity as its harmonic extension. Put
\begin{equation*}
\begin{aligned}
\mathcal X_b&=\{(\phi,w)\in C_b^{3+\beta}(\overline{\mathcal R})
                 \times C_b^{3+\beta}(\Gamma):\phi|_{\Gamma\cup B}=0\},\\
\mathcal Y_b&=\{(f,g)\in C_b^{3+\beta}(\overline{\mathcal R})
                 \times C_b^{2+\beta}(\Gamma):f|_{\Gamma\cup B}=0\},\\
\mathcal X&=\mathcal X_b\cap
          \big(C_0^2(\overline{\mathcal R})\times C_0^2(\Gamma)\big),\\
\mathcal Y&=\mathcal Y_b\cap
          \big(C_0^2(\overline{\mathcal R})\times C_0^1(\Gamma)\big).
\end{aligned}
\end{equation*}
All these spaces carry the indicated product H\"older norms. The subscript
$\mathrm e$ will denote the closed subspace of functions even in $x$.
The full spaces are used for linear estimates and the center-manifold calculation;
local invertibility at a nontrivial symmetric wave and global continuation are
understood on $\mathcal X_{\mathrm e}$ and $\mathcal Y_{\mathrm e}$.
Define
\begin{equation}\label{zy2}
\begin{split}
\mathcal U&=\{(\phi,w,\alpha)\in\mathcal X\times\R:
                    0<\alpha<\alpha_{\mathrm{cr}},\ \sigma(w,\alpha)>0\},\\
\sigma(w,\alpha)&=\inf_{\overline{\mathcal R}}
\big(\mu-2\alpha(y-1+\zeta_{[w]})\big)|e_y+\nabla\zeta_{[w]}|^2,
\qquad e_y=(0,1).
\end{split}
\end{equation}
Here $\alpha_{\mathrm{cr}}$ is defined in \eqref{criticalvalue} and
$\zeta_{[w]}$ is the harmonic extension in \eqref{zy5} below.
The corresponding even open set is $\mathcal U_{\mathrm e}
=\mathcal U\cap(\mathcal X_{\mathrm e}\times\R)$.
}
Set
 \begin{equation*}
\psi(x,y)=\psi_{\mathrm {triv}}(y)+\phi(x,y),\quad\eta(x,y)=\eta_{\mathrm {triv}}(y)+\zeta(x,y).
\end{equation*}
Then the governing equations (\ref{harmoniceta})-(\ref{zy1}) yields the following system for $\zeta(x, y)$:
\begin{equation}\label{zy5}\begin{cases}
        \Delta\zeta=0\quad&\mathrm{in~}\mathcal{R},\\
        \zeta=w &\mathrm{on~}\Gamma,\\
        \zeta=0 &\mathrm{on~}B.
\end{cases}\end{equation}
The asymptotic condition \eqref{streamasy}, together with the boundary
conditions in (\ref{zy5}), implies that the harmonic extension satisfies
\[
\zeta(x,y)\rightarrow0
\]
uniformly for \(y\in[0,1]\). Hence, by the uniqueness of the
Dirichlet problem in the decaying class, there exists a unique
solution \(\zeta=\zeta[w]\) to (\ref{zy5}).
Moreover, this decay at infinity allows us to follow \cite{wahlenduke} and introduce an abstract operator $\mathcal{A}$ that depends on $(\phi, w, \alpha)$, so that $\mathcal{A}(\phi, w, \alpha)$ is the unique solution to
\begin{equation*}\begin{cases}
\Delta\mathcal{A}=-\gamma(\phi+\psi_{\mathrm{triv}})|(0,1)+\nabla\zeta_{[w]}|^2+\gamma(\psi_{\mathrm{triv}})\quad &\mathrm{in~} \mathcal{R},\\
    \mathcal{A}=0 &\mathrm{on~}\Gamma\cup B.
\end{cases}\end{equation*}
Obviously $\phi=\mathcal{A}(\phi,w,\alpha)$ is equivalent to the statement that
\[{\psi=\phi+\psi_\mathrm{triv}}\]
solves \eqref{stripstreameq1}-\eqref{stripstreameq3}.
The equation \eqref{stripstreameq4}  is transformed into
\begin{equation}\label{toplinear}
    \frac{\Big|(\partial_y\mathcal{A}(\phi,w,\alpha)+\partial_y\psi_{\mathrm{triv}})\big|_{\Gamma}\Big|^2}{2}=\left(\frac{\mu}{2}-\alpha w\right)\big|(0,1)+\nabla\zeta_{[w]}\big|_{\Gamma}\big|^2.
\end{equation}
To summarize up, we have rewritten the problem \eqref{stripstreameq} into its equivalent operator form $\mathcal{F}(\phi,w,\alpha)=0$, where $\mathcal{F}$ is given as
\begin{equation} \label{maine}
\mathcal{F}(\phi,w,\alpha)=(\mathcal{F}_1(\phi,w,\alpha),\mathcal{F}_2(\phi,w,\alpha)),
\end{equation}
where
\begin{equation}\label{maine1}\begin{cases}
\mathcal{F}_1(\phi,w,\alpha)=\phi-\mathcal{A}(\phi,w,\alpha),\\
\mathcal{F}_2(\phi,w,\alpha)=\frac{\Big|(\partial_y\mathcal{A}(\phi,w,\alpha)+\partial_y\psi_{\mathrm{triv}})\big|_{\Gamma}\Big|^2}{2}-\left(\frac{\mu}{2}-\alpha w\right)\big|(0,1)+\nabla\zeta_{[w]}\big|_{\Gamma}\big|^2.
\end{cases}\end{equation}

{
For $\gamma\in C_{\mathrm{loc}}^{3+\beta}$, harmonic extension is
bounded $C_b^{3+\beta}(\Gamma)\to C_b^{3+\beta}(\overline{\mathcal R})$,
and the bounded inverse maps $C_b^{1+\beta}$ to
$C_b^{3+\beta}$ with zero boundary trace. The right-hand side defining
$\mathcal A$ belongs to $C_b^{1+\beta}$ on every bounded set of
$(\phi,w)\in\mathcal X_b$. It follows that
$\mathcal F_1\in C_b^{3+\beta}$ has zero boundary trace and
$\mathcal F_2\in C_b^{2+\beta}(\Gamma)$. Of course, these mapping statements hold
throughout $\mathcal U$ as well. For decaying data,
uniqueness of the bounded Dirichlet problem and local estimates show
that $\zeta_{[w]}$ and $\mathcal A$ belong to $C_0^2$. At the same time, the 
formulation for $\mathcal F_2$ then gives its $C_0^1$ decay.
Thus $\mathcal F:\mathcal U\to\mathcal Y$ is well defined.
The same argument applies to the even spaces. Differentiation of these
bounded linear solution operators and the pointwise nonlinearities gives
the derivatives below. In Section \ref{smallsec}, we assume $\gamma\in C_{\mathrm{loc}}^{7+\beta}$
and verifies the higher regularity needed for the center manifold.
For real-analytic $\gamma$, pointwise composition is real analytic on
bounded attained ranges; hence $\mathcal F$ is real analytic on
$\mathcal U$ in the stated norms.
}
\section{Linearization and Functional Analytic Formulation}\label{linearization}

It follows from \eqref{maine}-\eqref{maine1} that the Fr\'echet derivative \[\partial_{(\phi,w)}\mathcal{F}(\phi,w,\alpha):=(\partial_{(\phi,w)}\mathcal{F}_1(\phi,w,\alpha),\partial_{(\phi,w)}\mathcal{F}_2(\phi,w,\alpha)):\mathcal{X} \to \mathcal{Y}\] is given by
\begin{subequations}\label{frechet1}\begin{align}
        \partial_{(\phi,w)}\mathcal{F}_{1}(\phi,w,\alpha)(\dot\phi,\dot w)=&\dot \phi-\mathcal{A}_\phi(\phi,w,\alpha)\dot\phi-\mathcal{A}_w(\phi,w,\alpha)\dot w,\label{frechet11}\\
        \partial_{(\phi,w)}\mathcal{F}_{2}(\phi,w,\alpha)(\dot\phi,\dot w)=&\bigg(\partial_y\mathcal{A}(\phi,w,\alpha)+\psi_{\mathrm{triv},y}(1)\bigg)\bigg(\mathrm{\partial}_y(\mathcal{A}_\phi(\phi,w,\alpha)\dot \phi+\mathcal{A}_w(\phi,w,\alpha)\dot w)\bigg)\label{frechet12}
        \\&+{\alpha|e_y+\nabla\zeta_{[w]}|^2\dot w}-(\mu-2\alpha w)\left\langle e_y+\nabla\zeta_{[w]},\nabla\zeta_{[\dot w]}\right\rangle\bigg|_\Gamma\notag,
\end{align}\end{subequations}
where $\langle\cdot,\cdot\rangle$ in \eqref{frechet12} is the standard inner product in $\R^2$ and $e_y=(0,1)$. Besides, we use the norm in  $\mathcal{X}_b$ as
\[\|(\phi,w)\|_{\mathcal{X}_b}=\|\phi\|_{C_b^{3+\beta}(\overline{\mathcal{R}})}+{\|w\|_{C_{b}^{3+\beta}(\Gamma)}}.\]
It is easy to verify that $\mathcal{A}_w(\phi,w,\alpha)\dot w$ and $\mathcal{A}_\phi(\phi,w,\alpha)\dot \phi$ in \eqref{frechet1} are the unique solutions of

{
\begin{equation}\label{awphi}
\begin{cases}
\Delta(\mathcal A_w\dot w)=-2\gamma(\psi_{\mathrm{triv}}+\phi)
           (e_y+\nabla\zeta_{[w]})\cdot\nabla\zeta_{[\dot w]}
                    &\mathrm{in}\ \mathcal R,\\
\mathcal A_w\dot w=0&\mathrm{on}\ \Gamma\cup B,
\end{cases}
\end{equation}
and
\begin{equation}\label{aphiw}
\begin{cases}
\Delta(\mathcal A_\phi\dot\phi)=-\gamma'(\psi_{\mathrm{triv}}+\phi)
                    |e_y+\nabla\zeta_{[w]}|^2\dot\phi
                    &\mathrm{in}\ \mathcal R,\\
\mathcal A_\phi\dot\phi=0&\mathrm{on}\ \Gamma\cup B.
\end{cases}
\end{equation}
The coefficients in these interior equations are evaluated throughout the strip,
not restricted to $\Gamma$.
}
In particular, for $\phi=0$ and $w=0$, it follows from (\ref{awphi})-(\ref{aphiw}) that $\mathcal{A}_w^0\dot w$ and $\mathcal{A}_\phi^0\dot\phi$ are the unique solutions of
\begin{equation*}
\begin{cases}\Delta\mathcal{A}_w^0\dot w=-2\gamma(\psi_{\mathrm{triv}})\partial_y\zeta_{[\dot w]}\quad&\mathrm{in~}\mathcal{R},\\
        \mathcal{A}_w^0\dot w=0 &\mathrm{on~} \Gamma\cup B,
    \end{cases}\end{equation*}
and
\begin{equation*}
\begin{cases}\Delta\mathcal{A}_\phi^0\dot\phi=-\gamma'(\psi_{\mathrm{triv}})\dot\phi\quad&\mathrm{in~}\mathcal{R},\\
        \mathcal{A}_\phi^0\dot\phi=0 &\mathrm{on~} \Gamma\cup B.
    \end{cases}\end{equation*}
Obviously, it follows from \eqref{frechet1} that
\begin{equation}\label{frechet00}\begin{aligned}
 \partial_{(\phi,w)}\mathcal{F}_{1}(0,0,\alpha)(\dot\phi,\dot w)=&\dot \phi-\mathcal{A}_\phi^0\dot\phi-\mathcal{A}_w^0\dot w,\\
        \partial_{(\phi,w)}\mathcal{F}_{2}(0,0,\alpha)(\dot\phi,\dot w)=&\psi_{\mathrm{triv},y}(1)\left(\mathrm{\partial}_y(\mathcal{A}_\phi^0\dot \phi+\mathcal{A}_w^0\dot w)\big|_{\Gamma}\right)
        +\alpha\dot w-\mu\partial_y\zeta_{[\dot w]}\big|_{\Gamma}.
\end{aligned}\end{equation}

Before studying the functional properties of the linearized operator $\mathcal{F}_{(\phi,w)}$, let us introduce an isomorphism to facilitate the subsequent analysis. Such an isomorphism is the so-called $\mathcal{T}$-isomorphism. See \cite{alinhac, matssiam, lannesjams}.

{
For the good unknown, introduce
\begin{equation*}
\begin{aligned}
\widetilde{\mathcal X}_b&=
\{\theta\in C_b^{3+\beta}(\overline{\mathcal R}):\theta|_B=0\},\\
\widetilde{\mathcal X}&=\widetilde{\mathcal X}_b\cap
                                  C_0^2(\overline{\mathcal R}),
\qquad
\|\theta\|_{\widetilde{\mathcal X}_b}
=\|\theta\|_{C_b^{3+\beta}(\overline{\mathcal R})}.
\end{aligned}
\end{equation*}
The trace $\theta|_\Gamma$ automatically belongs to $C_b^{3+\beta}(\Gamma)$.
The harmonic extension is bounded from $C_b^{3+\beta}(\Gamma)$ to
$C_b^{3+\beta}(\overline{\mathcal R})$ and preserves the stated decay.
Consequently, \eqref{tiso} and \eqref{tinverse} define bounded inverse maps
on the bounded, decaying and even spaces. In particular,
\[
\|\mathcal T\theta\|_{\mathcal X_b}\leq C\|\theta\|_{\widetilde{\mathcal X}_b},
\qquad
\|\mathcal T^{-1}(\dot\phi,\dot w)\|_{\widetilde{\mathcal X}_b}
\leq C\big(\|\dot\phi\|_{C_b^{3+\beta}}
                        +\|\dot w\|_{C_b^{3+\beta}}\big).
\]
}
Define the linear isomorphism   $\mathcal{T}:\tilde{\mathcal{X}} \to\mathcal{X}$ by
\begin{equation}\label{tiso}
\mathcal{T}\theta=\left(\theta-\frac{\psi_{\mathrm{triv},y}}{\psi_{\mathrm{triv},y}(1)}
\zeta_{[\theta|_\Gamma]},-\frac{\theta|_\Gamma}{\psi_{\mathrm{triv},y}(1)}\right),
\end{equation}
which has the inverse given by
\begin{equation}\label{tinverse}[\mathcal{T}]^{-1}(\phi,w)=\phi-\psi_{\mathrm{triv},y}\zeta_{[w]}.\end{equation}

\begin{lemma}\label{identitylemma}For all $\theta\in\tilde{\mathcal{X}},$ it holds that
\[\mathcal{A}^0_w(\theta|_\Gamma)+\mathcal{A}^0_\phi(\psi_{\mathrm{triv},y}\zeta_{[\theta|_\Gamma]})
=(\psi_{\mathrm{triv},y}-\psi_{\mathrm{triv},y}(1))\zeta_{[\theta|_\Gamma]}.\]
\end{lemma}
\begin{proof}
Let $f:=\mathcal{A}^0_w(\theta|_\Gamma)+\mathcal{A}^0_\phi(\psi_{\mathrm{triv},y}\zeta_{[\theta|_\Gamma]})-(\psi_{\mathrm{triv},y}-\psi_{\mathrm{triv},y}(1))\zeta_{[\theta|_\Gamma]}$. Then
\begin{equation*}\begin{aligned}
    \Delta f&=-2\gamma(\psi_{\mathrm{triv}})\partial_{y}\zeta_{[\theta|_\Gamma]}-\gamma'(\psi_{\mathrm{triv}})\psi_{\mathrm{triv},y}\zeta_{[\theta|_\Gamma]}-\psi_{\mathrm{triv},yyy}\zeta_{[\theta|_\Gamma]}-2
    \psi_{\mathrm{triv},yy}\partial_{y}\zeta_{[\theta|_\Gamma]}\\
    &=-2\gamma(\psi_{\mathrm{triv}})\partial_{y}\zeta_{[\theta|_\Gamma]}-\gamma'(\psi_{\mathrm{triv}})\psi_{\mathrm{triv},y}\zeta_{[\theta|_\Gamma]}+\gamma'(\psi_{\mathrm{triv}})\psi_{\mathrm{triv},y}\zeta_{[\theta|_\Gamma]}
    +2\gamma(\psi_{\mathrm{triv}})\partial_{y}\zeta_{[\theta|_\Gamma]}\\
    &=0
\end{aligned}\end{equation*}
  in $\mathcal{R}$ and vanishes at the top and bottom. By Theorem \ref{strong}, we can deduce that $f\equiv0$.
\end{proof}
Now we introduce the operators $\mathcal{L}(\phi,w,\alpha)$ and $\mathcal{L}^0$, which are defined by
\[\mathcal{L}(\phi,w,\alpha)=\mathcal{F}_{(\phi,w)}(\phi,w,\alpha)\circ \mathcal{T},\]
and
\[\mathcal{L}^0:=\mathcal{L}(0,0,\alpha)=\mathcal{F}_{(\phi,w)}(0
,0,\alpha)\circ\mathcal{T}=(\mathcal{L}_1^0,\mathcal{L}_2^0),\]
It follows from (\ref{frechet00}), (\ref{tiso}) and Lemma \ref{identitylemma} that
\begin{align*}
    &\mathcal{L}^0_1\theta=\theta-(\mathcal{A}_\phi^0\theta+\zeta_{[\theta|_\Gamma]}),\\
    &\mathcal{L}^0_2\theta=\left(\psi_{\mathrm{triv},y}(1)\partial_y\mathcal{A}_\phi^0\theta-\left(\frac{\alpha}{\psi_{\mathrm{triv},y}(1)}-\gamma(1)\right)\theta+\frac{\mu}{\psi_{\mathrm{triv},y}(1)}\partial_y\zeta_{[\theta|_\Gamma]}\right)\bigg|_{\Gamma}.
    \end{align*}
Note that $\mathcal{L}_1^0\theta$ is the unique solution of the problem
\begin{equation}\label{L1}\begin{cases}
\Delta[\mathcal{L}_1^0\theta]=\Delta\theta+\gamma'(\psi_{\mathrm{triv}})\theta\quad&\mathrm{in~}\mathcal{R},\\
        \mathcal{L}_1^0\theta=0 &\mathrm{on~}\Gamma\cup B,
\end{cases}\end{equation}
and $\mathcal{L}_2^0\theta$ can be rewritten as
\begin{equation}\label{L2}
    \mathcal{L}_2^0\theta=\left(\psi_{\mathrm{triv},y}(1)\partial_y(\theta-\mathcal{L}_1^0\theta-\zeta_{[\theta|_\Gamma]})-
    \left(\frac{\alpha}{\psi_{\mathrm{triv},y}(1)}-\gamma(1)\right)\theta+\frac{\mu}{\psi_{\mathrm{triv},y}(1)}\partial_y\zeta_{[\theta|_\Gamma]}\right)\bigg|_{\Gamma}.
\end{equation}
{To study the kernel of $\mathcal F_{(\phi,w)}(0,0,\alpha)$, it suffices to study the kernel of $\mathcal L^0$.}
Besides, $\mathcal{L}^0\theta=0$ if and only if $\theta$ satisfies
\begin{equation*}\begin{cases}
        \Delta\theta+\gamma'(\psi_{\mathrm{triv}})\theta=0 \quad\mathrm{in~}\mathcal{R},\\
       \left( \psi_{\mathrm{triv},y}(1)\partial_y\theta-\left(\frac{\alpha}{\psi_{\mathrm{triv},y}(1)}-\gamma(1)\right)\theta+
       \left(\frac{\mu}{\psi_{\mathrm{triv},y}(1)}-\psi_{\mathrm{triv},y}(1)\right)\partial_y\zeta_{[\theta|_\Gamma]}\right)\bigg|_{\Gamma}=0.
\end{cases}\end{equation*}
It follows from (\ref{lameq3}) that
\begin{equation*}\begin{aligned}
		 \frac{\mu}{\psi_{\mathrm{triv},y}(1)}-\psi_{\mathrm{triv},y}(1)&=\frac{\mu-\psi_{\mathrm{triv},y}^2(1)}{\psi_{\mathrm{triv},y}(1)}=0.
\end{aligned}\end{equation*}
Thus we have deduced the following boundary value problem
\begin{equation}\label{kernelpde}\begin{cases}
        \Delta\theta+\gamma'(\psi_{\mathrm{triv}})\theta=0 \quad&\mathrm{in~}\mathcal{R},\\
      \partial_y\theta-\tilde\alpha\theta=0,&\mathrm{on~}\Gamma,\\
        \theta=0,&\mathrm{on~} B,
\end{cases}\end{equation}
where \[\tilde\alpha=-\frac{\gamma(1)}{\psi_{\mathrm{triv},y}(1)}+\frac{\alpha}{\psi_{\mathrm{triv},y}^2(1)}.\]

{
\begin{lemma}\label{coupledestimate}
For $(\phi,w,\alpha)\in\mathcal U$, the linearized operators satisfy
\begin{equation}\label{schauderphiw}
\|(\dot\phi,\dot w)\|_{\mathcal X_b}
\leq C\big(\|\mathcal F_{(\phi,w)}(\phi,w,\alpha)
                    (\dot\phi,\dot w)\|_{\mathcal Y_b}
             +\|\dot\phi\|_{C_b^0(\overline{\mathcal R})}
              +\|\dot w\|_{C_b^0(\Gamma)}\big)
\end{equation}
and
\begin{equation}\label{schauder}
\|\theta\|_{\widetilde{\mathcal X}_b}
\leq C\big(\|\mathcal L(\phi,w,\alpha)\theta\|_{\mathcal Y_b}
                       +\|\theta\|_{C_b^0(\overline{\mathcal R})}\big).
\end{equation}
The constant depends on bounded sets of coefficients, $\alpha$, the vorticity function $\gamma$, and the lower bound for $\sigma(w,\alpha)$.
\end{lemma}
\begin{proof}
Consider the non-homogeneous problem in operator form
\[
(f,g)=\mathcal F_{(\phi,w)}(\phi,w,\alpha)(\dot\phi,\dot w),
\]
and write
\[
\psi=\psi_{\mathrm{triv}}+\phi,\qquad
\eta=y+\zeta_{[w]},\qquad\dot\zeta=\zeta_{[\dot w]}.
\]
Since $\mathcal A_\phi\dot\phi+\mathcal A_w\dot w=\dot\phi-f$,
equations \eqref{awphi}--\eqref{aphiw} and \eqref{frechet12} give
\begin{equation}\label{schauder1}
\begin{cases}
\Delta\dot\phi+\gamma'(\psi)|\nabla\eta|^2\dot\phi
             +2\gamma(\psi)\nabla\eta\cdot\nabla\dot\zeta=\Delta f
                                      &\mathrm{in}\ \mathcal R,\\
\Delta\dot\zeta=0&\mathrm{in}\ \mathcal R,\\
\dot\phi=0&\mathrm{on}\ \Gamma,\\
\begin{aligned}
 &(\mathcal A_y+\psi_{\mathrm{triv},y}(1))\dot\phi_y
       -(\mu-2\alpha w)(\eta_x\dot\zeta_x+\eta_y\dot\zeta_y)\\
 &\qquad+\alpha|\nabla\eta|^2\dot\zeta
       =g+(\mathcal A_y+\psi_{\mathrm{triv},y}(1))f_y
\end{aligned}&\mathrm{on}\ \Gamma,\\
\dot\phi=\dot\zeta=0&\mathrm{on}\ B.
\end{cases}
\end{equation}
In the derivative-first convention of Lemma \ref{schauderlemma2}, the principal
boundary coefficients are
\[
\begin{aligned}
c_1&=1,\qquad c_2=a_{11}=0,\qquad
a_{21}=(\mathcal A_y+\psi_{\mathrm{triv},y}(1))|_\Gamma,\\
a_{12}&=-(\mu-2\alpha w)\eta_x,\qquad
a_{22}=-(\mu-2\alpha w)\eta_y.
\end{aligned}
\]
The complementing determinant in \eqref{lopatinskii} is therefore
\[
(\mu-2\alpha w)^2|\nabla\eta|^2
 \geq\frac{\sigma(w,\alpha)^2}
              {\|\nabla\eta\|_{C_b^0(\Gamma)}^2}>0.
\]
Applying Lemma \ref{schauderlemma2} to \eqref{schauder1} with $k=3$
and estimating its interior and boundary data gives
\begin{equation}\label{schauder3}
\begin{split}
\|\dot\phi\|_{C_b^{3+\beta}}+\|\dot\zeta\|_{C_b^{3+\beta}}
\leq C\big(&\|f\|_{C_b^{3+\beta}}+\|g\|_{C_b^{2+\beta}(\Gamma)}\\
             &+\|\dot\phi\|_{C_b^0}+\|\dot\zeta\|_{C_b^0}\big).
\end{split}
\end{equation}
The trace estimate and the maximum principle for harmonic extension implies
$\|\dot\zeta\|_{C_b^0}\leq\|\dot w\|_{C_b^0}$, which proves
\eqref{schauderphiw}. Finally, the bounded inverse maps
\eqref{tiso}--\eqref{tinverse} of $\mathcal{T}$ will give \eqref{schauder}.
\end{proof}
}
\begin{lemma}\label{trivialkernel}
Let $\alpha_{\mathrm{cr}}$ be given in {\rm\x{criticalvalue}}. Then the operator $\mathcal{L}^0:\tilde{\mathcal{X}}_b\to \mathcal{Y}_b$ has a trivial kernel for any $\alpha<\alpha_{\mathrm{cr}}$.
\end{lemma}
\begin{proof}
Let $\theta(x,y)=u(x)v(y)$ be the solution of problem \eqref{kernelpde}. Then we obtain the following self-adjoint Sturm-Liouville problem
\begin{subequations}\label{ev}\begin{align}
   &-v_{yy}- \gamma'(\psi_{\mathrm{triv}})v=\nu v\label{ev1},\\
   &v_y(1)=\tilde\alpha v(1),\label{ev2}\\
   &v(0)=0.\label{ev3}
    \end{align}\end{subequations}
It follows from the standard spectrum theory that problem \x{ev} admits a countable
number of simple eigenvalues $\nu_0<\nu_1<\cdots$ with $\nu_n\to\infty$.

First we show that $\nu=0$ is an eigenvalue of \eqref{ev} if and only if \begin{equation*}\tilde\alpha=\tilde\alpha_{\mathrm{cr}}=-\frac{\gamma(1)}{\psi_{\mathrm{triv},y}(1)}+\frac{1}{\psi_{\mathrm{triv},y}^2(1)\int_0^1\frac{dy}{\psi_{\mathrm{triv},y}^2(y)}}.\end{equation*}
Indeed, for $\nu=0$, differentiating \eqref{trivial} with respect to $y$, we have that $v_1=\psi_{\mathrm{triv},y}(y)$ is a non-trivial solution. By the Liouville's formulation, the general solution of \eqref{ev1} with \eqref{ev3} is
\be\label{generalsol}v(y)=C\psi_{\mathrm{triv},y}(y)\int_0^y\frac{dt}{\psi_{\mathrm{triv},y}^2(t)}\ee
with $C$ being an arbitrary constant.
It is easy to see that \eqref{ev2} is  satisfied if and only if $\tilde{\alpha}=\tilde{\alpha}_{\mathrm{cr}}$. Besides, the critical value
has the following relation
\begin{equation}\label{criticalvalue}
\alpha_{\mathrm{cr}}={\tilde\alpha_{\mathrm{cr}}}\psi_{\mathrm{triv},y}^2(1)+\gamma(1)\psi_{\mathrm{triv},y}(1)=\frac{1}{\int_0^1\frac{dy}{\psi_{\mathrm{triv},y}^2(y)}}.
\end{equation}

{
At the critical parameter, choose the normalization
\begin{equation}\label{kernel1}
\varphi_0(y)=\frac{\psi_{\mathrm{triv},y}(y)}
 {\psi_{\mathrm{triv},y}(1)\displaystyle\int_0^1\psi_{\mathrm{triv},y}^{-2}(t)\,dt}
\int_0^y\psi_{\mathrm{triv},y}^{-2}(t)\,dt,
\qquad \varphi_0(1)=1.
\end{equation}
Then $\varphi_0>0$ on $(0,1]$. The kernel in the bounded strip space
is $\operatorname{span}\{\varphi_0\}$. In an exponentially growing
space with $0<\tau<\sqrt{\nu_1}$, it is
$\operatorname{span}\{\varphi_0,x\varphi_0\}$. Indeed, projection on
each transverse eigenfunction gives $a_n''=\nu_n a_n$. At criticality
$\nu_0=0$, while $\nu_n\geq\nu_1>0$ for $n\geq1$; the weight excludes
both exponential modes for every $n\geq1$.
}
Next, we know that the eigenvalue of \eqref{ev} can be characterized by the following minimization problem
\begin{equation}\label{minivar}
	\nu_0(\tilde\alpha)=\inf_{v\in H^{1}(0,1),v(0)=0,v\not\equiv0}\frac{\int_0^1v_y^2-\gamma'(\psi_{\mathrm{triv}})v^2dy-\tilde\alpha v^2(1)}{\int_0^1v^2dy}.
\end{equation}
{
Differentiating the eigenvalue equation and its Robin condition gives
\[
\frac{d\nu_0}{d\widetilde\alpha}
<0,
\]
 hence
$\nu_0(\widetilde\alpha)>0$ for
$\widetilde\alpha<\widetilde\alpha_{\mathrm{cr}}$.
}

{
To prove triviality of the bounded kernel, no integrability in $x$ is assumed.
Let $v_n$ be an orthonormal eigenbasis for \eqref{ev} and set
$a_n(x)=\int_0^1\theta(x,y)v_n(y)\,dy$. Integration by parts in $y$
gives $a_n''=\nu_n a_n$. Since $\nu_n\geq\nu_0>0$ and $a_n$ is
bounded on the whole line, $a_n=0$. Completeness of the eigenfunctions
then implies $\theta=0$.
}
\end{proof}

\begin{example}\label{constcr}When the vorticity function is constant $\gamma(\psi)=-\gamma$, the laminar solution becomes
\[\psi_{\mathrm{triv}}(y)=\frac{\gamma}{2}y^2+\bigg(1-\frac{\gamma}{2}\bigg)y,\quad \psi_{\mathrm{triv},y}(1)=1+\frac{\gamma}{2},\]
and the critical value can be computed as
 \[{\tilde\alpha_{\mathrm{cr}}=1},\quad \alpha_{\mathrm{cr}}=1-\gamma^2/4.\]
This critical value $\alpha_{\mathrm{cr}}$ is consistent with the definition  in \cite{milesjma}.
Moreover, using the non-dimensional variables, we can compute that
\[{m=\frac{Fg^{1/2}d^{3/2}}{1+\gamma/2},\qquad
\alpha=\frac{(1+\gamma/2)^2}{F^2}=\frac\mu{F^2}},\qquad
\mu=\big(1+\frac{\gamma}{2}\big)^2.\]
Note that such waves should satisfy $\gamma\in(-2,2)$ and of course $\psi_{\mathrm{triv},y}(1)>0$. In particular, for the irrotational flows $\gamma=0$, the critical value is $\alpha_{\mathrm{cr}}=1$, which corresponds to $F_{\mathrm{cr}}=1$.
\end{example}
In the following example, we consider the affine vorticity.
\begin{example}
    When the vorticity function is linear, that is $\gamma(\psi)=\lambda \psi$ with $\lambda\neq0$, the laminar solution is given by
    \begin{equation*}
        \psi_{\mathrm{triv}}(y)=\begin{cases}
            \frac{\sin(\sqrt{\lambda}y)}{\sin \sqrt{\lambda}},\quad&\lambda>0,\\
 \frac{\sinh (\sqrt{-\lambda}y)}{\sinh \sqrt{-\lambda}},&\lambda<0.
        \end{cases}
    \end{equation*}
The critical value can be computed as
    \begin{equation*}
        {\tilde{\alpha}_{\mathrm{cr}}}=\begin{cases}
\sqrt{\lambda}\cot{\sqrt{\lambda}},\quad&\lambda>0\\
\sqrt{-\lambda}\coth{\sqrt{-\lambda}},&\lambda<0,
        \end{cases}
    \end{equation*}
    and
    \begin{equation*}
        \alpha_{\mathrm{cr}}=\begin{cases}
{\lambda^{3/2}\big(\cot^3(\sqrt{\lambda})+\cot(\sqrt{\lambda})\big)},\quad&\lambda>0\\
{(-\lambda)^{3/2}\big(\coth^3(\sqrt{-\lambda})-\coth(\sqrt{-\lambda})\big)},&\lambda<0.
        \end{cases}
    \end{equation*}
    It is easy to verify that when $\lambda<\frac{\pi^2}{4}$, the critical value $\alpha_{\mathrm{cr}}$ is well-defined and is strictly positive. Moreover, the eigenfunction $\varphi_0$ is given by
    \begin{equation*}
        \varphi_0(y)=\begin{cases}
            {\dfrac{\sin(\sqrt{\lambda}y)}{\sin\sqrt{\lambda}}}\quad&\lambda>0,\\
            {\dfrac{\sinh(\sqrt{-\lambda}y)}{\sinh\sqrt{-\lambda}}}&\lambda<0.
        \end{cases}
    \end{equation*}
\end{example}
Now we try to prove the local properness. The limiting operator of $\mathcal{L}(\phi,w,\alpha)$ is otained by taking its limit as $x\to\pm\infty$.
By the asymptotic conditions, it is easy to see that the limiting operator is  nothing other than $\mathcal{L}^0$.

\begin{lemma}\label{localproper}
	For $(\phi,w,\alpha)\in\mathcal{U}$, the operator $\mathcal{L}(\phi,w,\alpha)$ is locally proper both as a map $\tilde{\mathcal{X}}_b\to\mathcal{Y}_b$ and as a map  $\tilde{\mathcal{X}}\to\mathcal{Y}$. Moreover, the operator $\mathcal{F}_{(\phi,w)}(\phi,w,\alpha)$  is locally proper both as a map $\mathcal{X}_b\to\mathcal{Y}_b$ and as a map $\mathcal{X}\to\mathcal{Y}$.
\end{lemma}
\begin{proof}
	By the translation argument and Lemma \ref{coupledestimate}, one can show that $\mathcal{L}(\phi,w,\alpha):\tilde{\mathcal{X}}_b\to\mathcal{Y}_b$   is locally proper if and only if the limiting operator $\mathcal{
		L}^0:\tilde{\mathcal{X}}_b\to\mathcal{Y}_b$  has a trivial kernel \cite{vol}. By Lemma \ref{trivialkernel}, we can deduce that $\mathcal{L}(\phi,w,\alpha):\tilde{\mathcal{X}}_b\to\mathcal{Y}_b$ is locally proper.
	
	Let $\{\dot{\theta}_n\}\subset\tilde{\mathcal{X}}$ be a bounded sequence  such that $\{\mathcal{L}(\phi,w,\alpha)\dot{\theta}_n\}$ is a convergent sequence in $\mathcal{Y}$. Since $\mathcal{L}(\phi,w,\alpha):\tilde{\mathcal{X}}_b\to\mathcal{Y}_b$ is locally proper, we can extract a subsequence so that $\dot\theta_n\to\dot\theta\in\tilde{\mathcal{X}}_b$. Since $\tilde{\mathcal{X}}$ is a closed subspace, we must have $\dot\theta\in\tilde{\mathcal{X}}$. Thus $\mathcal{L}(\phi,w,\alpha):\tilde{\mathcal{X}}\to\mathcal{Y}$ is also locally proper.
	Finally, since the operator $\mathcal{T}$ is an isomorphism, the operator $\mathcal{F}_{(\phi,w)}(\phi,w,\alpha)$ is hence locally proper both as a map from $\mathcal{X}_b$ to $\mathcal{Y}_b$ and as a map from $\mathcal{X}$ to $\mathcal{Y}$.
\end{proof}

{
\begin{lemma}\label{invertiblelemma}
For $0<\alpha<\alpha_{\mathrm{cr}}$, the operator $\mathcal L^0$ is an
isomorphism both from $\widetilde{\mathcal X}_b$ to $\mathcal Y_b$ and from
$\widetilde{\mathcal X}$ to $\mathcal Y$. Consequently,
$\mathcal F_{(\phi,w)}(0,0,\alpha)$ is an isomorphism on the corresponding
product spaces. The same assertions hold on the even subspaces.
\end{lemma}
\begin{proof}
For $(f,g)\in\mathcal Y_b$, together with
$\mu=\psi_{\mathrm{triv},y}^2(1)$ and \eqref{L1}--\eqref{L2}, show that
$\mathcal L^0\theta=(f,g)$ is equivalent to
\begin{equation}\label{nonhomo}
\begin{cases}
\Delta\theta+\gamma'(\psi_{\mathrm{triv}})\theta=\Delta f
                                      &\mathrm{in}\ \mathcal R,\\
\theta_y-\widetilde\alpha\theta
       =f_y+\dfrac{g}{\psi_{\mathrm{triv},y}(1)}&\mathrm{on}\ \Gamma,\\
\theta=0&\mathrm{on}\ B.
\end{cases}
\end{equation}
In recovering the first component from this differential equation, we use
the zero trace of $f$ and uniqueness of the bounded Dirichlet problem.

Since $\psi_{\mathrm{triv},y}>0$ on $[0,1]$ and
$\psi_{\mathrm{triv},yyy}+\gamma'(\psi_{\mathrm{triv}})
\psi_{\mathrm{triv},y}=0$, equation \eqref{nonhomo} can be written as
\begin{equation}\label{coercivestrip}
\begin{cases}
\displaystyle
\nabla\cdot\left(\psi_{\mathrm{triv},y}^2
        \nabla\frac{\theta}{\psi_{\mathrm{triv},y}}\right)
       =\psi_{\mathrm{triv},y}\Delta f&\mathrm{in}\ \mathcal R,\\[2mm]
\displaystyle
\psi_{\mathrm{triv},y}^2(1)
       \partial_y\frac{\theta}{\psi_{\mathrm{triv},y}}
       -\frac{\alpha\theta}{\psi_{\mathrm{triv},y}(1)}
       =\psi_{\mathrm{triv},y}(1)f_y+g&\mathrm{on}\ \Gamma,\\
\theta=0&\mathrm{on}\ B.
\end{cases}
\end{equation}
For a test function vanishing on $B$, Cauchy--Schwarz gives
\begin{equation*}
\begin{aligned}
	\left|\frac{\theta}{\psi_{\mathrm{triv},y}}(x,1)\right|^2&=\bigg|\int_0^1\frac{1}{\psi_{\mathrm{triv},y}}\bigg(\psi_{\mathrm{triv},y}\partial_y\bigg(\frac{\theta}{\psi_{\mathrm{triv},y}}\bigg)(x,t)\bigg) dt\bigg|^2\\
	&\leq \left( \int_0^1 \frac{1}{\psi_{\mathrm{triv},y}^2}  dy \right)
	\left( \int_0^1 \psi_{\mathrm{triv},y}^2 \bigg|\partial_y\bigg(\frac{\theta}{\psi_{\mathrm{triv},y}}\bigg)\bigg|^2 dy \right)\\&=\frac{1}{\alpha_{\mathrm{cr}}}	\left( \int_0^1 \psi_{\mathrm{triv},y}^2\bigg|\partial_y\bigg(\frac{\theta}{\psi_{\mathrm{triv},y}}\bigg)\bigg|^2 dy \right).
\end{aligned}
\end{equation*}
Consequently, on $\mathcal R_M:=(-M,M)\times(0,1)$ with zero lateral data, the quadratic
form of \eqref{coercivestrip} satisfies
\begin{equation}\label{co}
\begin{split}
&B(\theta,\theta)=\int_{-M}^M\!\int_0^1\psi_{\mathrm{triv},y}^2
       \left|\nabla\frac{\theta}{\psi_{\mathrm{triv},y}}\right|^2dy\,dx
 -\frac{\alpha}{\psi_{\mathrm{triv},y}^2(1)}
       \int_{-M}^M\theta(x,1)^2dx\\
&\qquad\geq\left(1-\frac\alpha{\alpha_{\mathrm{cr}}}\right)
 \int_{-M}^M\!\int_0^1\psi_{\mathrm{triv},y}^2
       \left|\nabla\frac{\theta}{\psi_{\mathrm{triv},y}}\right|^2dydx.
\end{split}
\end{equation}
 Lax--Milgram therefore implies a weak solution on
each rectangle $\mathcal{R}_M$.

To obtain a bound independent of $M$, choose a function $v=v(y)$ satisfying
\[
-(\psi_{\mathrm{triv},y}^2v')'=1,\qquad v(0)=1,\qquad
\psi_{\mathrm{triv},y}^2(1)v'(1)-\alpha v(1)=1.
\]
A trivial integration gives $\psi_{\mathrm{triv},y}^2v'=C-y$.  Consider the function \[w = \theta/\psi_{\mathrm{triv},y} - Kv.\] 
Setting $w^+=\max(w,0)$ and plugging $w^+$ into $B(\cdot,\cdot)$  in \eqref{co}, we have $C\|w^+\|_{H_1}^2\leq B(w^+,w^+)\leq0$,  for some constant $C>0$. Thus we can deduce that $|\theta/\psi_{\mathrm{triv},y}|\leq K\|v\|_{C^0}$.

Schauder estimate \cite{agmon} gives 
\begin{equation*}
	\left\|\theta_M\right\|_{C^{3+\beta}((j,j+1)\times[0,1])} \leq C\left( \left\|f\right\|_{C_b^{3+\beta}} + \left\|g\right\|_{C_b^{2+\beta}} \right),
\end{equation*}
whenever \(j \in \mathbb{R}\) and \((j-1,j+2) \subset (-M,M)\), with \(C\) independent of \(j\) and \(M\). Letting \(M \to \infty\), a diagonal compactness argument produces \(\theta \in C_b^{3+\beta}(\overline{\mathbb{R}})\), zero on \(B\), which solves \(\mathcal{L}^0 \theta = (f,g)\). The local estimate passes to the limit, so \(\theta \in \tilde{\mathcal{X}}_b\). Thus \(\mathcal{L}^0\) is surject from $\tilde {\mathcal {X}}_b$ to $\mathcal Y_b$. Injectivity from Lemma 3.3 makes it an isomorphism and gives the uniform inverse bound.

If $(f,g) \in \mathcal{Y}$, their derivatives occurring in the right-hand side vanish at infinity. Any sequence of translates of $\theta$ going to either end has a subsequence converging locally in $C^2$ to a bounded solution of the homogeneous problem. Lemma \ref{trivialkernel} makes every such limit zero. Thus $\theta \in C_0^2(\overline{\mathcal{R}})$, which implies $\theta \in \tilde{\mathcal{X}}$. Reflection and uniqueness give the assertion on the even subspaces, and applying the isomorphism $\mathcal{T}$ yields the corresponding assertions for $\mathcal{F}_{(\phi,w)}$.

\end{proof}
}

{
\section{Existence of Small-amplitude Waves}\label{smallsec}

Throughout this section, assume $\gamma\in C_{\mathrm{loc}}^{7+\beta}(\mathbb R)$
and \eqref{nocritical}. Write $\alpha=\alpha_{\mathrm{cr}}-\varepsilon$ and
use the normalization $\varphi_0(1)=1$ in \eqref{kernel1}.
We use the weighted spaces
\begin{equation*}
\begin{aligned}
\mathcal X_\tau&=\{(\phi,w)\in C_\tau^{3+\beta}(\overline{\mathcal R})
                   \times C_\tau^{3+\beta}(\Gamma):\phi|_{\Gamma\cup B}=0\},\\
\mathcal Y_\tau&=\{(f,g)\in C_\tau^{3+\beta}(\overline{\mathcal R})
                   \times C_\tau^{2+\beta}(\Gamma):f|_{\Gamma\cup B}=0\},\\
\widetilde{\mathcal X}_\tau&=\{\theta\in C_\tau^{3+\beta}(\overline{\mathcal R}):
                                                   \theta|_B=0\}.
\end{aligned}
\end{equation*}
Here $0<\tau<\min\{\pi,\sqrt{\nu_1/2}\}$, with $\nu_1$ evaluated at
$\alpha=\alpha_{\mathrm{cr}}$. The harmonic extension and $\mathcal T$
are bounded on these spaces. At the critical $\alpha$,
\[
\mathcal L^{\mathrm{cr}}:=\mathcal L(0,0,\alpha_{\mathrm{cr}}),\qquad
\ker\mathcal L^{\mathrm{cr}}
 =\operatorname{span}\{\varphi_0,x\varphi_0\}.
\]
The weighted spaces here are full spaces; evenness is imposed only when
centering a solitary wave or asserting invertibility at that wave.

The local conformal equations are
\begin{equation}\label{localcm}
\begin{cases}
\Delta\phi+\gamma(\psi_{\mathrm{triv}}+\phi)|e_y+\nabla\zeta|^2
                         -\gamma(\psi_{\mathrm{triv}})=0&\mathrm{in}\ \mathcal R,\\
\Delta\zeta=0&\mathrm{in}\ \mathcal R,\\
\phi=0&\mathrm{on}\ \Gamma,\\
(\psi_{\mathrm{triv},y}(1)+\phi_y)^2
 -(\mu-2\alpha\zeta)|e_y+\nabla\zeta|^2=0&\mathrm{on}\ \Gamma,\\
\phi=\zeta=0&\mathrm{on}\ B.
\end{cases}
\end{equation}

\begin{lemma}\label{centerhypotheses}
The local system \eqref{localcm} admits a reversible center-manifold
reduction with a coordinate map
$\Upsilon\in C^3(\mathbb R^3;\widetilde{\mathcal X}_\tau)$ and reduced
vector field $f\in C^3(\mathbb R^3)$. More precisely, in sufficiently small
neighborhoods of the origin the following statements hold.

\noindent $(i)$ Every uniformly small bounded solution corresponds to
$q(x)=\theta(x,1)$ satisfying $q''=f(q,q',\varepsilon)$, where
$\theta=\phi-\psi_{\mathrm{triv},y}\zeta$. Conversely, every solution of
this ODE with $(q,q',\varepsilon)$ uniformly small reconstructs a uniformly
small bounded solution of \eqref{localcm}. The reconstruction is
\begin{equation}\label{cmreconstruction}
\theta(x+t,y)=q(x)\varphi_0(y)+q'(x)t\varphi_0(y)
             +\Upsilon(q(x),q'(x),\varepsilon)(t,y).
\end{equation}

\noindent $(ii)$ The coordinate map satisfies
\[
\begin{gathered}
\Upsilon(0,0,\varepsilon)=0,\qquad
\Upsilon_A(0,0,0)=\Upsilon_B(0,0,0)=0,\\
\Upsilon(A,B,\varepsilon)(0,1)
 =\Upsilon_x(A,B,\varepsilon)(0,1)=0,\\
f(A,B,\varepsilon)=\Upsilon_{xx}(A,B,\varepsilon)(0,1),
\qquad f(A,-B,\varepsilon)=f(A,B,\varepsilon).
\end{gathered}
\]

\noindent $(iii)$ If $\dot\theta$ is a bounded solution of the linearized
equation at one of these solutions, then $\dot q=\dot\theta|_\Gamma$
satisfies
\[
\dot q''=f_A(q,q',\varepsilon)\dot q
              +f_B(q,q',\varepsilon)\dot q'.
\]
The correspondence is injective: $\dot q\equiv0$ implies
$\dot\theta\equiv0$.
\end{lemma}
\begin{proof}
We verify the fixed-point hypotheses of \cite[Appendix A]{chennonlinearity}
for the scalar equation obtained from \eqref{localcm}. For
$(\phi,w)=\mathcal T\theta$ and $\zeta=\zeta_{[w]}$, equations
\eqref{tiso} and \eqref{tinverse} give
\[
\phi=\theta-\frac{\psi_{\mathrm{triv},y}}
 {\psi_{\mathrm{triv},y}(1)}\zeta_{[\theta|_\Gamma]},\qquad
\zeta=-\frac{\zeta_{[\theta|_\Gamma]}}{\psi_{\mathrm{triv},y}(1)}.
\]
The harmonic equation and all Dirichlet conditions in \eqref{localcm}
are then automatic. Collecting the linear terms gives the equivalent
scalar problem
\begin{equation}\label{cmscalar}
\begin{cases}
\Delta\theta+\gamma'(\psi_{\mathrm{triv}})\theta
                =\mathcal N_1(\theta,\varepsilon)&\mathrm{in}\ \mathcal R,\\
-\theta_y+\widetilde\alpha_{\mathrm{cr}}\theta
                =\mathcal N_2(\theta,\varepsilon)&\mathrm{on}\ \Gamma,\\
\theta=0&\mathrm{on}\ B,
\end{cases}
\end{equation}
where the two nonlinear remainders are explicitly
\begin{equation}\label{cmnonlinear}
\begin{aligned}
\mathcal N_1={}&-\gamma(\psi_{\mathrm{triv}}+\phi)
       +\gamma(\psi_{\mathrm{triv}})+\gamma'(\psi_{\mathrm{triv}})\phi\\
 &-2\big(\gamma(\psi_{\mathrm{triv}}+\phi)
                       -\gamma(\psi_{\mathrm{triv}})\big)\zeta_y
       -\gamma(\psi_{\mathrm{triv}}+\phi)|\nabla\zeta|^2,\\
\mathcal N_2={}&\left[
 \frac{\phi_y^2-\mu|\nabla\zeta|^2}{2\psi_{\mathrm{triv},y}(1)}
 +\frac{(\alpha_{\mathrm{cr}}-\varepsilon)\zeta}
             {\psi_{\mathrm{triv},y}(1)}
                      (2\zeta_y+|\nabla\zeta|^2)
 -\frac{\varepsilon\zeta}{\psi_{\mathrm{triv},y}(1)}\right]_\Gamma.
\end{aligned}
\end{equation}
In particular, each pointwise nonlinearity vanishes together with its
first derivatives at the origin.

For the fixed-point construction, we first work in $C^{2+\beta}$ space. Thus the right hand side in \eqref{cmscalar} belong to
$C^{\beta}(\overline{\mathcal R})\times C^{1+\beta}(\Gamma)$. The linear part of \eqref{cmscalar} is
self-adjoint and invariant under reflection in $x$. Its mixed Dirichlet
and Robin boundary conditions are covered by
\cite[Section 3.2]{chennonlinearity}, and Lemma \ref{trivialkernel}
provides the simple zero eigenvalue and the spectral gap.

The extra nonlocal term is controlled directly. For $0<y<1$, the strip
Poisson formula reads
\[
\zeta_{[w]}(x,y)=\frac12\int_{\mathbb R}
 \frac{\sin(\pi y)}{\cosh(\pi(x-t))+\cos(\pi y)}w(t)\,dt.
\]
The exponential decay of this kernel, followed by local Dirichlet
estimates, gives
\begin{equation}\label{cmharmonicbound}
\|\zeta_{[w]}\|_{C_\tau^{k+\beta}(\overline{\mathcal R})}
       \leq C_\tau\|w\|_{C_\tau^{k+\beta}(\Gamma)},
\qquad 0\leq\tau<\pi,\quad k=2,3.
\end{equation}
The constants are uniform on compact subintervals of $[0,\pi)$.
The extension commutes with translations and reflections in $x$.

For clarity, write out the center and complementary equations. Decompose
\[
\theta=U_1(x)\varphi_0(y)+R(x,y),\qquad
\int_0^1R(x,y)\varphi_0(y)dy=0,\qquad U_1'=\delta U_2,
\]
where $\delta>0$ is the auxiliary scaling parameter of the fixed-point
construction. Let $\mathcal K(f,g)$ denote the unique weighted solution
$R$, orthogonal to $\varphi_0$ on every cross section, of
\begin{equation}\label{cmgreen}
\begin{cases}
\displaystyle
\Delta R+\gamma'(\psi_{\mathrm{triv}})R
 =f-\frac{\varphi_0(y)}{\int_0^1\varphi_0^2dt}
            \left(\int_0^1 f(x,t)\varphi_0(t)dt+g(x)\right)
                                      &\mathrm{in}\ \mathcal R,\\
-R_y+\widetilde\alpha_{\mathrm{cr}}R=g&\mathrm{on}\ \Gamma,\\
R=0&\mathrm{on}\ B.
\end{cases}
\end{equation}
The boundary term $g$ in the first equation is required by Green's identity.
The partial Green estimate in \cite[Proposition 2.2]{chennonlinearity},
with the mixed boundary conditions just described, yields
\[
\|\mathcal K(f,g)\|_{C_\tau^{2+\beta}}
 \leq C_\tau\big(\|f\|_{C_\tau^\beta}
                        +\|g\|_{C_\tau^{1+\beta}(\Gamma)}\big).
\]
This also holds at weight zero, uniformly on compact subintervals below
$\sqrt{\nu_1/2}$. The operator $\mathcal K$ commutes with translations
and reflections.
Consider the  product space 
\[
C_\tau^{2+\beta}(\mathbb R)\times
C_\tau^{1+\beta}(\mathbb R)\times
C_\tau^{2+\beta}(\overline{\mathcal R}),
\] projecting \eqref{cmscalar} against $\varphi_0$ now
gives the fixed-point equations
\begin{equation}\label{cmfixedpoint}
\begin{aligned}
U_1(x)&=U_1(0)+\delta\int_0^xU_2(t)dt,\\
U_2(x)&=U_2(0)+\frac1{\delta\int_0^1\varphi_0^2dy}
 \int_0^x\left(\int_0^1\mathcal N_1(t,y)\varphi_0(y)dy
                                 +\mathcal N_2(t)\right)dt,\\
R&=\mathcal K(\mathcal N_1,\mathcal N_2),
\end{aligned}
\end{equation}
where the two constants $U_1(0)$ and $U_2(0)$ are prescribed integration constants.
It is easy to verify that the nonlinear hypotheses of
\cite[(A.2)--(A.8)]{chennonlinearity} hold in the product spaces.

Apply smooth scalar truncations to the independent arguments in
\eqref{cmnonlinear}, after the preceding linear cancellations. Choose
the truncations to equal the identity near zero and to commute with sign
changes; reflection invariance is then preserved. No interchange of a
cutoff and harmonic extension is used. The boundedness at weight zero
shows that the truncated and original equations, and their derivatives
in bounded directions, coincide in a sufficiently small unweighted ball.
For higher derivatives, apply \cite[Theorem A.1]{chennonlinearity} with
$k=2$, $M=3$ and an auxiliary weight $\tau_0<\tau/10$. All weights
needed in its differentiability argument, up to $8\tau_0$, then remain
within the ranges of \eqref{cmharmonicbound} and \eqref{cmgreen}.
This gives a $C^4$ family in the weighted $C^{2+\beta}$ spaces.

By flatness, the map from the prescribed pair $(U_1(0),U_2(0))$ to
$(\theta(0,1),\theta_x(0,1))$ has derivative
$\operatorname{diag}(1,\delta)$ at the origin. The inverse function
theorem therefore permits the coordinates $(A,B)=(q(0),q'(0))$.
Translation invariance and uniqueness of the fixed point give
\eqref{cmreconstruction} and $q''=f(q,q',\varepsilon)$, as in
\cite[Section 2.5 and Remark 1.8]{chennonlinearity}. Reflection gives
$f(A,-B,\varepsilon)=f(A,B,\varepsilon)$.

We finally recover the manuscript's $C^{3+\beta}$ regularity. For the
parameterized solution $\theta=(A+Bx)\varphi_0+\Upsilon$, differentiation
of the translation identity at zero gives
\[
\theta_x=B\,\partial_A\theta
                  +f(A,B,\varepsilon)\,\partial_B\theta.
\]
Hence $\theta_x$ is $C^3$ in the parameters with values in weighted
$C^{2+\beta}$. The first equation of \eqref{cmscalar}, with its truncated
right-hand side, gives
\[
\theta_{yyy}=-\theta_{xxy}
 -\partial_y\big(\gamma'(\psi_{\mathrm{triv}})\theta\big)
 +\partial_y\mathcal N_1.
\]
The last term and its first three parameter derivatives belong to
$C^{\beta}_{4\tau_0}$: each is a sum of products of at most four
weighted factors, with bounded derivatives of the scalar truncated
functions. Together with the preceding identity for $\theta_x$, this
proves that $\Upsilon$ is $C^3$ into
$\widetilde{\mathcal X}_\tau$. It also proves bounded $C^{3+\beta}$
regularity for the reconstructed small bounded solutions. Conversely,
every small $C_b^{3+\beta}$ solution is a small $C_b^{2+\beta}$ solution
to which the fixed-point correspondence applies. Smallness throughout
this correspondence is measured in unweighted norms.

For the linearized correspondence, augment \eqref{cmscalar} by its
linearization in $\dot\theta$. The principal linear operator is diagonal
with two copies of the same scalar operator. The nonlinear terms are
\[
\big(\mathcal N_i(\theta,\varepsilon),
       D_\theta\mathcal N_i(\theta,\varepsilon)\dot\theta\big),
\qquad i=1,2.
\]
They have the same form just verified, with the additional bounded linear
entries $(\dot\phi,\dot\phi_y,\dot\zeta,\nabla\dot\zeta)$.
The assumed regularity of $\gamma$ also supplies the derivatives required
for these terms. Thus the augmented fixed-point argument in
\cite[Sections 3.4--3.5, Theorem 1.9]{chennonlinearity} applies with
the same extension and partial Green estimates. Rescaling any bounded
$\dot\theta$ by a small constant places the pair in the augmented
neighborhood. Uniqueness of its fixed point shows that $\dot\theta$ is
determined by $(\dot q(0),\dot q'(0))$. Differentiating the reconstruction
then gives the stated variational ODE and its injective correspondence.
\end{proof}

Define
\begin{equation}\label{thetadef}
\begin{aligned}
\Theta(y)={}&-\frac12\gamma''(\psi_{\mathrm{triv}})
 \left(\varphi_0(y)-\frac{y\psi_{\mathrm{triv},y}(y)}
                              {\psi_{\mathrm{triv},y}(1)}\right)^2\\
&+\frac{2\gamma'(\psi_{\mathrm{triv}})}{\psi_{\mathrm{triv},y}(1)}
 \left(\varphi_0(y)-\frac{y\psi_{\mathrm{triv},y}(y)}
                              {\psi_{\mathrm{triv},y}(1)}\right)
 -\frac{\gamma(\psi_{\mathrm{triv}})}{\psi_{\mathrm{triv},y}^2(1)},\\
C_0={}&-\frac12\left(\varphi_0'(1)-1
                     +\frac{\gamma(1)}{\psi_{\mathrm{triv},y}(1)}\right)^2
       +\frac{\mu-4\alpha_{\mathrm{cr}}}{2\psi_{\mathrm{triv},y}^2(1)}\\
={}&-\frac{\alpha_{\mathrm{cr}}}{\psi_{\mathrm{triv},y}^2(1)}
       -\frac{\alpha_{\mathrm{cr}}^2}{2\psi_{\mathrm{triv},y}^4(1)},
\end{aligned}
\end{equation}
where the last equality follows from the critical Robin condition and
$\mu=\psi_{\mathrm{triv},y}^2(1)$. Set
\begin{equation}\label{M0}
M_0=\frac{C_0}{\psi_{\mathrm{triv},y}(1)}
                  -\int_0^1\Theta(y)\varphi_0(y)dy.
\end{equation}

\begin{theorem}\label{coordinate}
Under the assumptions of this section, the reduced equation in Lemma
\ref{centerhypotheses} is
\begin{equation}\label{reducedode1}
q''=f(q,q',\varepsilon),
\end{equation}
where
\begin{equation}\label{taylor}
\begin{aligned}
f(A,B,\varepsilon)={}&f_{101}\varepsilon A+f_{200}A^2
 +O\big(B^2+|A|^3+|\varepsilon|A^2+\varepsilon^2|A|\big),\\
f_{101}={}&\frac1{\psi_{\mathrm{triv},y}^2(1)
                            \displaystyle\int_0^1\varphi_0^2(y)dy},
\qquad
f_{200}=-\frac{M_0}{\displaystyle\int_0^1\varphi_0^2(y)dy}.
\end{aligned}
\end{equation}
Moreover,
\begin{equation}\label{M0sign}
M_0=-\frac{3\alpha_{\mathrm{cr}}^3}{2\psi_{\mathrm{triv},y}^3(1)}
                   \int_0^1\frac{dy}{\psi_{\mathrm{triv},y}^4(y)}<0.
\end{equation}
In particular, $f_{101}>0$ and $f_{200}>0$.
\end{theorem}
\begin{proof}
Lemma \ref{centerhypotheses} supplies the coordinate map. We use ordinary
Taylor coefficients, so that
$\Upsilon_{200}=\frac12\Upsilon_{AA}(0,0,0)$ and
$\Upsilon_{101}=\Upsilon_{A\varepsilon}(0,0,0)$.
To compute these two coefficients, set $B=0$ and compare the coefficients
of $A^2$ and $\varepsilon A$. Equivalently, differentiation of
$\mathcal F(\mathcal T\theta,\alpha_{\mathrm{cr}}-\varepsilon)=0$
gives
\begin{equation}\label{iteration}
\begin{aligned}
\mathcal L^{\mathrm{cr}}\Upsilon_{200}
 &=-\frac12\partial_{(\phi,w)}^2\mathcal F(0,0,\alpha_{\mathrm{cr}})
                   [\mathcal T\varphi_0,\mathcal T\varphi_0],\\
\mathcal L^{\mathrm{cr}}\Upsilon_{101}
 &=\partial_\alpha\partial_{(\phi,w)}
       \mathcal F(0,0,\alpha_{\mathrm{cr}})\mathcal T\varphi_0.
\end{aligned}
\end{equation}

For $\theta_0=(A+Bx)\varphi_0(y)$, the first-order reconstruction is
\begin{equation}\label{w0}
\begin{aligned}
\mathcal T\theta_0&=(\phi_0,w_0),\\
\phi_0(x,y)&=\left(\varphi_0(y)
       -\frac{y\psi_{\mathrm{triv},y}(y)}{\psi_{\mathrm{triv},y}(1)}\right)
                                                    (A+Bx),\\
w_0(x)&=-\frac{A+Bx}{\psi_{\mathrm{triv},y}(1)}.
\end{aligned}
\end{equation}
In particular, both boundary traces of $\phi_0$ vanish, and
\begin{equation}\label{zeta0}
\zeta_{[w_0]}=-\frac{y(A+Bx)}{\psi_{\mathrm{triv},y}(1)}.
\end{equation}
Substituting \eqref{w0}--\eqref{zeta0} with $B=0$ into
\eqref{cmnonlinear} gives
\begin{equation*}
\begin{aligned}
\mathcal N_1(A\varphi_0,0)&=\Theta A^2+O(A^3),\\
\mathcal N_2(A\varphi_0,\varepsilon)
 &=\frac{\varepsilon A}{\psi_{\mathrm{triv},y}^2(1)}
   -\frac{C_0A^2}{\psi_{\mathrm{triv},y}(1)}
   +O\big(|A|^3+|\varepsilon|A^2\big).
\end{aligned}
\end{equation*}
 Moreover, we can compute that
 \begin{equation*}
 	\partial_{(\phi,w)}^2\mathcal{F}_{1}(0,0,\alpha^0)[(\dot
 	\phi,\dot{w}),(\dot
 	\phi,\dot{w})]=-\mathcal{A}_{\phi\phi}^0[\dot
 	\phi,\dot{\phi}]-2\mathcal{A}_{\phi w}^0[\dot
 	\phi,\dot{w}]-\mathcal{A}_{ww}^0[\dot
 	w,\dot{w}],
 \end{equation*}
 \begin{equation*}
 	\partial_\varepsilon\partial_{(\phi,w)}\mathcal{F}_{2}(0,0,\alpha^0)(\dot\phi,\dot w)=-\dot w,
 \end{equation*}
 where $\mathcal{A}_{\phi\phi}^0[\dot
 \phi,\dot{\phi}]$, $\mathcal{A}_{\phi w}^0[\dot
 \phi,\dot{w}]$, $\mathcal{A}_{ww}^0[\dot
 w,\dot{w}]$ are the unique solutions of
 \begin{equation*}
 	\begin{cases}\Delta\mathcal{A}_{ww}^0[\dot w,\dot w]=-2\gamma(\psi_{\mathrm{triv}})|\nabla\zeta_{[\dot w]}|^2\quad&\mathrm{in~}\mathcal{R},\\
 		\mathcal{A}_{ww}^0[\dot w,\dot w]=0 &\mathrm{on~} \Gamma\cup B,
 	\end{cases}
 \end{equation*}
 \begin{equation*}
 	\begin{cases}\Delta\mathcal{A}_{\phi\phi}^0[\dot \phi,\dot \phi]=-\gamma''(\psi_{\mathrm{triv}})\dot\phi^2\quad&\mathrm{in~}\mathcal{R},\\
 		\mathcal{A}_{\phi\phi}^0[\dot \phi,\dot \phi]=0 &\mathrm{on~} \Gamma\cup B,
 	\end{cases}
 \end{equation*}
 and
 \begin{equation*}
 	\begin{cases}\Delta\mathcal{A}_{w\phi}^0[\dot w,\dot \phi]=-2\gamma'(\psi_{\mathrm{triv}})\dot\phi\partial_y\zeta_{[\dot w]}\quad&\mathrm{in~}\mathcal{R},\\
 		\mathcal{A}_{w\phi}^0[\dot w,\dot \phi]=0 &\mathrm{on~} \Gamma\cup B,
 	\end{cases}
 \end{equation*}
 respectively.
We therefore obtain
\begin{equation}\label{U200}
\begin{cases}
\Delta\Upsilon_{200}+\gamma'(\psi_{\mathrm{triv}})\Upsilon_{200}
                       =\Theta&\mathrm{in}\ \mathcal R,\\
\partial_y\Upsilon_{200}-\widetilde\alpha_{\mathrm{cr}}\Upsilon_{200}
                       =\dfrac{C_0}{\psi_{\mathrm{triv},y}(1)}
                                       &\mathrm{on}\ \Gamma,\\
\Upsilon_{200}=0&\mathrm{on}\ B,
\end{cases}
\end{equation}
and
\begin{equation}\label{101eq}
\begin{cases}
\Delta\Upsilon_{101}+\gamma'(\psi_{\mathrm{triv}})\Upsilon_{101}=0
                                       &\mathrm{in}\ \mathcal R,\\
\partial_y\Upsilon_{101}-\widetilde\alpha_{\mathrm{cr}}\Upsilon_{101}
                       =-\dfrac1{\psi_{\mathrm{triv},y}^2(1)}
                                       &\mathrm{on}\ \Gamma,\\
\Upsilon_{101}=0&\mathrm{on}\ B.
\end{cases}
\end{equation}
Differentiation in $x$ shows their $x$-derivatives is in the critical weighted
kernel. The coordinate normalization then gives
\[
\Upsilon_{101}=\tfrac12 f_{101}x^2\varphi_0+g(y),\qquad
\Upsilon_{200}=\tfrac12 f_{200}x^2\varphi_0+k(y),
\]
with $g(0)=g(1)=k(0)=k(1)=0$.
Multiplication of their transverse equations by $\varphi_0$ and integration
by parts yield
\[
\begin{aligned}
-f_{101}\int_0^1\varphi_0^2dy
       &=-\frac1{\psi_{\mathrm{triv},y}^2(1)},\\
\int_0^1\Theta\varphi_0dy-f_{200}\int_0^1\varphi_0^2dy
       &=\frac{C_0}{\psi_{\mathrm{triv},y}(1)}.
\end{aligned}
\]
These identities prove the coefficient formulas in \eqref{taylor}.

Here we consider  the auxiliary Dirichlet problem, in fact, it will be used to derive the conjugate flow equation in Section 6,
\begin{equation}\label{conjusec4}
\hat\psi_{\mathrm{triv},yy}=-d^2\gamma(\hat\psi_{\mathrm{triv}}),
\qquad \hat\psi_{\mathrm{triv}}(0;d)=0,\quad
\hat\psi_{\mathrm{triv}}(1;d)=1.
\end{equation}
We shall obtain the laminar family
$\hat\psi_{\mathrm{triv}}(y;d)$ and total head $\hat Q(d)$,
only for $d\approx1$ and with $\alpha=\alpha_{\mathrm{cr}}$ fixed.
Indeed, the Dirichlet linearization is invertible,  for a nonzero test
function $v$ vanishing at both endpoints,
\[
\int_0^1\big(v_y^2-\gamma'(\psi_{\mathrm{triv}})v^2\big)dy
 =\int_0^1\psi_{\mathrm{triv},y}^2
         \left|\partial_y\frac{v}{\psi_{\mathrm{triv},y}}\right|^2dy>0.
\]
The implicit function theorem and \eqref{nocritical} therefore give a
smooth, strictly increasing family near $d=1$. We claim the following identities hold
\begin{subequations}\label{laminarvariations}
\begin{align}
\left.\partial_d\hat\psi_{\mathrm{triv}}(y;d)\right|_{d=1}
 &=y\psi_{\mathrm{triv},y}(y)
           -\psi_{\mathrm{triv},y}(1)\varphi_0(y),\label{laminarvariations1}\\
\big(\partial_y^2+\gamma'(\psi_{\mathrm{triv}})\big)
 \left.\partial_d^2\hat\psi_{\mathrm{triv}}(y;d)\right|_{d=1}
 &=2\psi_{\mathrm{triv},y}^2(1)\Theta(y).\label{laminarvariations2}
\end{align}
\end{subequations}
We denote $\dot\psi=\partial_d\hat\psi_{\mathrm{triv}}(y,d)|_{d=1}$ and  $\ddot\psi=\partial_d^2\hat\psi_{\mathrm{triv}}(y,d)|_{d=1}$. Differentiating \eqref{conjusec4} with respect to $d$, and let $d=1$, $\dot \psi$ satisfies the following Dirichlet problem
\begin{equation}\label{dotpsi}
	\dot{\psi}_{yy} + \gamma'(\psi_{\text{triv}})\dot{\psi} = -2\gamma(\psi_{\text{triv}})\quad \dot{\psi}(0)=\dot{\psi}(1)=0.	
\end{equation}
It is easy to verify that $\dot\psi=y \psi_{\mathrm{triv},y}(y) - \psi_{\mathrm{triv},y}(1)\varphi_0(y)$ satisfies  \eqref{dotpsi}.
Meanwhile, $\ddot{\psi}$ satisfies  
\begin{equation}
	(\partial_y^2 + \gamma'(\psi_{\mathrm{triv}}))\ddot{\psi}
	= -\gamma''(\psi_{\mathrm{triv}})\dot{\psi}^2
	- 4\gamma'(\psi_{\mathrm{triv}})\dot{\psi}
	- 2\gamma(\psi_{\mathrm{triv}}).
\end{equation}
Moreover, we have 
\begin{equation}
	\begin{aligned}
	\Theta(y)& = -\frac{1}{2}\gamma''(\psi_{\text{triv}}) \left(-\frac{\dot{\psi}}{\psi_{\text{triv},y}(1)}\right)^2 
	+ \frac{2\gamma'(\psi_{\text{triv}})}{\psi_{\text{triv},y}(1)} \left(-\frac{\dot{\psi}}{\psi_{\text{triv},y}(1)}\right) 
	- \frac{\gamma(\psi_{\text{triv}})}{\psi_{\text{triv},y}^2(1)}\\
	&=\frac{1}{2\psi_{\mathrm{triv},y}^2(1)} \left[ -\gamma''(\psi_{\mathrm{triv}})\dot{\psi}^2 - 4\gamma'(\psi_{\mathrm{triv}})\dot{\psi} - 2\gamma(\psi_{\mathrm{triv}}) \right],
	\end{aligned}
\end{equation}
that is \eqref{laminarvariations2}.
The second identity vanishes at $y=0,1$. Divergence theorem thus implies
\[
\left.\partial_d^2\hat\psi_{\mathrm{triv},y}(1;d)\right|_{d=1}
       =2\psi_{\mathrm{triv},y}^2(1)\int_0^1\Theta\varphi_0dy.
\]
For
\[
\hat Q(d)=\left(\frac{\hat\psi_{\mathrm{triv},y}(1;d)}d\right)^2
                       +2\alpha_{\mathrm{cr}}(d-1),
\]
equation \eqref{laminarvariations} and the critical Robin condition give
\begin{equation}\label{M0head}
\begin{aligned}
\hat Q''(1)
 &=4\alpha_{\mathrm{cr}}
   +\frac{2\alpha_{\mathrm{cr}}^2}{\psi_{\mathrm{triv},y}^2(1)}
   +4\psi_{\mathrm{triv},y}^3(1)\int_0^1\Theta\varphi_0dy\\
 &=-4\psi_{\mathrm{triv},y}^3(1)M_0.
\end{aligned}
\end{equation}
On the other hand, the first integral of the same laminar ODE gives the
scalar quadrature
\[
d=\int_0^1
 \frac{dt}{\sqrt{\hat Q(d)-2\alpha_{\mathrm{cr}}(d-1)
                         +2\int_t^1\gamma(r)dr}}.
\]
Differentiating once and twice at $d=1$, and applying the change of variable
$t=\psi_{\mathrm{triv}}(y)$ in the  integral gives
\[
\hat Q'(1)=0,\qquad
\hat Q''(1)=6\alpha_{\mathrm{cr}}^3
                         \int_0^1\frac{dy}{\psi_{\mathrm{triv},y}^4(y)}>0.
\]
Thus, we can obtain \eqref{M0sign}, thus we have $f_{101}>0$ and $f_{200}>0$.
\end{proof}}

\begin{theorem}\label{smallthm}
Assume $\gamma\in C_{\mathrm{loc}}^{7+\beta}(\mathbb R)$ and
\eqref{nocritical}. Then there is a continuous curve of nontrivial
symmetric small solitary waves
$\mathscr C_{\mathrm{loc}}\subset\tilde{\mathcal X}\times\mathbb R$,
with $\alpha^\varepsilon=\alpha_{\mathrm{cr}}-\varepsilon$,
$0<\varepsilon\ll1$, and
\[
\theta^\varepsilon(x,1)
 =-\frac{3\varepsilon f_{101}}{2f_{200}}
       \sech^2\left(\frac{\sqrt{\varepsilon f_{101}}}{2}x\right)
       +O(\varepsilon^2)
\]
in $C_b^{3+\beta}(\Gamma)$. Equivalently,
\begin{equation}\label{wsmall}
\begin{aligned}
w^\varepsilon(x)
 &=\frac{3\varepsilon f_{101}}
          {2\psi_{\mathrm{triv},y}(1)f_{200}}
       \sech^2\left(\frac{\sqrt{\varepsilon f_{101}}}{2}x\right)
       +O(\varepsilon^2)\\
 &=\frac{\varepsilon}
     {\alpha_{\mathrm{cr}}^3\displaystyle\int_0^1
                            \psi_{\mathrm{triv},y}^{-4}(y)dy}
       \sech^2\left(\frac{\sqrt{\varepsilon f_{101}}}{2}x\right)
       +O(\varepsilon^2).
\end{aligned}
\end{equation}
Moreover:

\noindent $(i)$ These waves are of elevation and satisfy
\begin{equation}\label{nodalproperty}
\eta_x<0\ \mathrm{in}\ (\mathcal R\cup\Gamma)\cap\{x>0\},\qquad
\eta_x>0\ \mathrm{in}\ (\mathcal R\cup\Gamma)\cap\{x<0\}.
\end{equation}

\noindent$(ii)$ there exists $\delta>0$ such that, if $(\phi,w)\in \mathcal{X}$, $\|(\phi,w)\|_{\mathcal X}<\delta$, and $\varepsilon>0$ is sufficiently small and $w$ satisfies    \eqref{nodalproperty}, then $\mathcal{F}(\phi, w,\alpha^\varepsilon)=0$ implies that $w=w^\varepsilon$.

\noindent $(iii)$ for all  $0<\varepsilon\ll 1$, the linearized operator $\mathcal{F}_{(\phi,w)}(\phi^\varepsilon,w^\varepsilon,\alpha^\varepsilon)$ is invertible.
\end{theorem}
\begin{proof}
  Introducing the scaled variables
\begin{equation*}
	x=|\varepsilon|^{-1/2}X,\quad q(x)=\varepsilon Q(X),\quad q_x(x)=\varepsilon|\varepsilon|^{1/2}P(X),\quad P=Q_X,
\end{equation*}
the expansion \eqref{taylor} yields
\begin{equation}
	\label{reducedq}
	Q_{XX}=f_{101}Q+f_{200}Q^2+\mathcal O(\varepsilon(|Q|+|P|)),
\end{equation}
when $\varepsilon\to 0^+$, we have an explicit homoclinic orbit
\begin{equation*}
	Q(X)=-\frac{3f_{101}}{2f_{200}}\sech^2\bigg(\frac{\sqrt{f_{101}}}{2}X\bigg),
\end{equation*}
joining the point $(0,0)$ to itself and intersecting the $Q-$axis at the point $(Q_0,0):=(-\frac{3f_{101}}{2f_{200}},0)$. See Figure \ref{fig:enter-label}.
Let $Q_*$ denote this explicit orbit and set
$G_\varepsilon(Q,P)=\varepsilon^{-2}
f(\varepsilon Q,\varepsilon^{3/2}P,\varepsilon)$ and
$G_0(Q,P)=f_{101}Q+f_{200}Q^2$.
Since $f\in C^3$, $f(0,0,\varepsilon)=0$, and $f$ is even in its
second argument, Taylor's formula and its first two differentiated
versions give $\|G_\varepsilon-G_0\|_{C^2(K)}\le C_K\varepsilon$
on every fixed bounded ball $K$ in the $(Q,P)$-plane.
The stable manifold theorem and the transverse intersection with $P=0$
therefore give a reversible homoclinic $Q_\varepsilon$, centered by
$Q_\varepsilon'(0)=0$ (see for example \cite{kirch}), with
$\|Q_\varepsilon-Q_*\|_{C_b^1(\mathbb R)}\le C\varepsilon$.
This estimate follows on a fixed finite interval from continuous
dependence and transversality, and on the two tails from Lipschitz
dependence of the local stable and unstable orbits in an exponential
weight with any fixed rate $0<\kappa<\sqrt{f_{101}}$.
Differentiating the ODE twice and using the $C^2$ estimate gives the
same bound in $C_b^4$, hence in $C_b^{3+\beta}$. Thus rescaling yields
\[
\big\|q^\varepsilon-
 \varepsilon Q_*(\sqrt{\varepsilon}\,\cdot)\big\|_{C_b^{3+\beta}}
 \le C\varepsilon
       \|Q_\varepsilon-Q_*\|_{C_b^{3+\beta}}
 \le C\varepsilon^2.
\]
Since $w^\varepsilon=-q^\varepsilon/\psi_{\mathrm{triv},y}(1)$,
this also proves \eqref{wsmall}.
\begin{figure}
	\centering
	\includegraphics[width=0.5\linewidth]{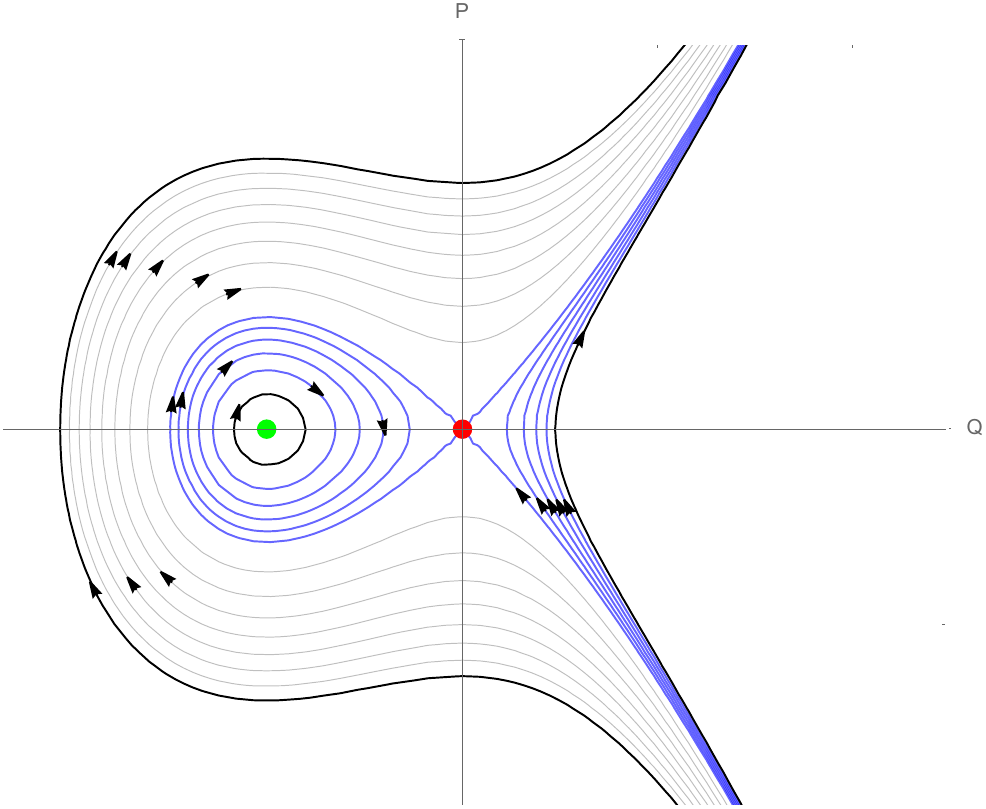}
	\caption{The phase portrait of equation \eqref{reducedq} when $\varepsilon\to0^+$ and $M_0<0$.}
	\label{fig:enter-label}
\end{figure}
The centered homoclinic orbit satisfies $q<0$ and $q_x>0$ for
$x>0$. Since $q=\theta|_\Gamma=-\psi_{\mathrm{triv},y}(1)w$,
we obtain $\eta_x=w_x<0$ on $\Gamma^+$. The function $\eta_x$
is harmonic, vanishes on the bed and the symmetry boundary, and tends
to zero at infinity. The strong maximum principle therefore gives
$\eta_x<0$ in $\mathcal R^+$. Reflection proves \eqref{nodalproperty}.
Now we prove $(ii)$. We assume that $\mathcal F(\phi,w,\alpha^\varepsilon)=0$ and $\|(\phi,w)\|_{\mathcal X}<\delta$, with $\theta=\mathcal T^{-1}(\phi,w)$. By the properties of the center manifold, $\theta$ is determined by a homoclinic orbit of the reduced ODE. In particular, for this solution to be a wave of elevation satisfying \eqref{nodalproperty}, the homoclinic orbit needs to lie in the left half-plane $\{Q<0\}$. We must therefore necessarily have $\theta=\theta^\varepsilon$, up to translation. Indeed, any solution $\theta$ which is not a translation of $\theta^\varepsilon$ must lie in the right half-plane $\{Q>0\}$ for large $|x|$, and can therefore not be a wave of elevation. By the isomorphism property of $\mathcal{T}$ ,  we can deduce that the property $(ii)$ holds.

Finally, we show that $(iii)$ holds. It follows from \cite[Theorem 1.9]{chennonlinearity} that $\dot{\theta}\in\mathrm{Ker~}\mathcal{L}(\phi,w,\alpha^\varepsilon)$ only if $\dot{q}=\dot{\theta}(\cdot,1)$ solves the linearized reduced ODE
\begin{equation*}
	\dot{q}''=\nabla_{(q,q')}f(q,q',\varepsilon)\cdot(\dot{q},\dot{q}').
\end{equation*}
The corresponding rescaled quantities $(\dot{Q},\dot{P})$ solve the planar system
\begin{equation}\label{ds}
	\begin{pmatrix}
		\dot{Q} \\ \dot{P}
	\end{pmatrix}_X
	=\mathcal{M}(X)
	\begin{pmatrix}
		\dot{Q} \\ \dot{P}
	\end{pmatrix}
\end{equation}
with
\begin{equation*}
	\lim_{X\to\pm\infty}\mathcal{M}(X)=\begin{pmatrix}
		0 &1 \\f_{101}+\mathcal{O}(\varepsilon^{1/2}) &\mathcal{O}(\varepsilon^{1/2})
	\end{pmatrix}.
\end{equation*}
As $X\to\pm\infty$, the eigenvalues of $\mathcal M$ approach
$\pm\sqrt{f_{101}}+O(\varepsilon^{1/2})$. Hence the space of
solutions of \eqref{ds} bounded as $X\to+\infty$ is one dimensional.
It contains the nonzero translation mode $q_x^\varepsilon$, so every
globally bounded variational solution is a multiple of this mode.
Since $q_x^\varepsilon$ is odd, the only bounded even solution is zero.
The injective linearized correspondence in Lemma \ref{centerhypotheses}
therefore gives a trivial kernel in the even space. The homotopy in
Lemma \ref{fredprop} commutes with reflection, so the same index-zero
argument applies to its even restrictions. Thus
$\mathcal F_{(\phi,w)}(\phi^\varepsilon,w^\varepsilon,\alpha^\varepsilon)
:\mathcal X_{\mathrm e}\to\mathcal Y_{\mathrm e}$ is invertible.
\end{proof}

\begin{example}
For constant vorticity $\gamma(t)=-\gamma$, assume $-2<\gamma<2$.
Then $\psi_{\mathrm{triv},y}(1)=1+\gamma/2$,
$\alpha_{\mathrm{cr}}=1-\gamma^2/4$, $\varphi_0(y)=y$, and
\[
\begin{aligned}
\Theta=\frac{\gamma}{\psi_{\mathrm{triv},y}^2(1)},\qquad
&M_0=-\frac{\gamma^2+12}{8\psi_{\mathrm{triv},y}^3(1)},\\
f_{101}=\frac3{\psi_{\mathrm{triv},y}^2(1)},\qquad
&f_{200}=\frac{3(\gamma^2+12)}{8\psi_{\mathrm{triv},y}^3(1)}.
\end{aligned}
\]
Consequently,
\[
w^\varepsilon(x)=\frac{12\varepsilon}{\gamma^2+12}
 \sech^2\left(\frac{\sqrt{3\varepsilon}}{2+\gamma}x\right)
 +O(\varepsilon^2).
\]
In particular, for irrotational flow $f_{101}=3$, $f_{200}=9/2$, and
\[
w^\varepsilon(x)=\varepsilon
       \sech^2\left(\frac{\sqrt{3\varepsilon}}2x\right)
       +O(\varepsilon^2).
\]

\end{example}

\begin{lemma}\label{fredprop}
	Let $\delta_1,\delta_2>0$ be sufficiently small. Then for any solution $(\phi, w,\alpha)\in \mathcal{U}$ to ${\mathcal{F}}(\phi,w,\alpha)=0$ with $\|(\phi,w)\|_{\mathcal{X}}<\delta_1$ and $0<\alpha_{\mathrm{cr}}-\alpha<\delta_2$, the linearized operator $\partial_{(\phi,w)}\mathcal{F}(\phi,w,\alpha)$ is Fredholm with index zero both as a map $\mathcal{X}_b\to\mathcal{Y}_b$ and as a map $\mathcal{X}\to\mathcal{Y}$.
\end{lemma}
\begin{proof}
	The limiting operator $\mathcal{L}(\phi,w,\alpha)$ is $\mathcal{L}^0$ which is invertible by Lemma \ref{invertiblelemma}. Moreover, from Lemma \ref{localproper}, we know that $\mathcal{L}(\phi,w,\alpha)$ is semi-Fredholm $\mathcal{X}_b\to\mathcal{Y}_b$ and $\mathcal{X}\to\mathcal{Y}$.
	We define the family of the operators
	\begin{equation*}
		\mathcal{L}^{(t)}=(1-t)\mathcal{L}(0,0,\alpha)+t\mathcal{L}(\phi,w,\alpha)\quad\mathrm{with~}t\in[0,1].
	\end{equation*}
	Applying the Schauder estimates \eqref{schauder1}and arguing as in Lemma \ref{localproper}, $\mathcal{L}^{(t)}$ is also locally proper both $\mathcal{X}_b\to\mathcal{Y}_b$ and $\mathcal{X}\to\mathcal{Y}$. By continuity of the index, we can deduce that $\mathcal{L}(\phi,w,\alpha)$ is a Fredholm operator with index zero both  $\mathcal{X}_b\to\mathcal{Y}_b$ and $\mathcal{X}\to\mathcal{Y}$. By the isomorphism property of the operator $\mathcal{T}$, we can deduce that the linearized operator $\partial_{(\phi,w)}\mathcal{F}(\phi,w,\alpha)$ of $\mathcal{F}$ is Fredholm with index zero both $\mathcal{X}_b\to\mathcal{Y}_b$ and $\mathcal{X}\to\mathcal{Y}$.
\end{proof}

\section{Velocity Fields and Nodal Properties}
Let us denote the velocity components $U$ and $V$ as  functions of the conformal variables $x$ and $y$:
\begin{equation*}
    u(x,y):=U(\xi(x,y),\eta(x,y))\quad\mathrm{and}\quad v(x,y):=V(\xi(x,y),\eta(x,y)).
\end{equation*}
By the chain principle, we have
\begin{equation}
    \label{conformalvelocity}
    (u,v)=\bigg(\frac{\psi_x\eta_x+\psi_y\eta_y}{\eta_x^2+\eta_y^2},\frac{
    \psi_y\eta_x-\psi_x\eta_y}{\eta_x^2+\eta_y^2}\bigg).
\end{equation}
It is easy to see that $u$ and $v$ solves
\begin{subequations}\label{conformalvelequa}
    \begin{align}
u_x+v_y&={-}\gamma(\psi)\eta_x\quad&&\mathrm{in~}\mathcal{R},\label{conformalvelequa1}\\
        u_y-v_x&={-}\gamma(\psi)\eta_y\quad&&\mathrm{in~}\mathcal{R},\label{conformalvelequa2}\\
        u\eta_x-v\eta_y&=0\quad&&\mathrm{on~}\Gamma,\label{conformalvelequa3}\\
        u^2+v^2&=\mu-2\alpha(\eta-1) &&\mathrm{on~}\Gamma,\label{conformalvelequa4}\\
        v&=0&&\mathrm{on~}B\label{conformalvelequa5},
    \end{align}
\end{subequations}
and the following asymptotic conditions hold
\begin{equation}\label{velasy}
    \lim_{x\to\pm\infty}u(x,y)=\psi_{\mathrm{triv},y}(y),\quad
    \lim_{x\to\pm\infty}v(x,y)=0.
\end{equation}
Note that
\begin{equation}\label{psixpsiy}
    \begin{aligned}
        \psi_x=u\eta_x-v\eta_y,\quad \psi_y=u\eta_y+v\eta_x.
    \end{aligned}
\end{equation}
Moreover, it is easy to see that $v$ satisfies the equation
\begin{equation*}
\Delta v=-\gamma'(\psi)v|\nabla\eta|^2\quad\mathrm{in~}\mathcal{R}.
\end{equation*}

Now we prove that solutions to \eqref{stripstreameq} possess a monotonicity known as the nodal property. For elevation solitary waves, we show that the vertical component $\eta$ of the wave profile $\mathcal{S}$ (see \eqref{freesurface}) is strictly decreasing on each side of the wave crest, that is \x{nodalproperty} holds.
It suffices to work in the positive half of the domain $\mathcal{R}$, we denote
\begin{equation*}
\mathcal{R}^+:=\{(x,y)\in\mathcal{R}:x>0\}\quad
\mathrm{and}\quad\Gamma^+:=\{(x,y)\in\Gamma:x>0\}
\end{equation*}
{and let $L^+:=\{(0,y):0<y<1\}$ denote the vertical symmetry boundary.}
The monotone property in terms of $v$  is
\begin{equation}\label{nodalv}
    v< 0\quad\mathrm{in~}\Gamma^+\cup\mathcal{R}^+.
\end{equation}

{
By \eqref{M0sign}, the local branch constructed here consists of elevation
waves. We therefore study the nodal pattern \eqref{nodalv}.
}
\begin{lemma}\label{lemma5.1}
Let $(\eta,\psi,\alpha)$ solve \eqref{stripstreameq} and $v$ be  defined by \eqref{conformalvelocity}. Assume that \eqref{nodalv} holds, then $\eta_x<0$ in $\Gamma^+\cup\mathcal{R}^+$.
\end{lemma}
\begin{proof}
	The proof  is similar to that of \cite[Lemma 4.1]{susannaarma}, we give a sketch. By \eqref{conformalvelequa3}, the vector fields $(u,v)$ and $(\eta_y,\eta_x)$ are parallel and non-vanishing on $\Gamma$.  By \eqref{velasy} and \eqref{trivial}, we can deduce that $u\to\psi_{\mathrm{triv},y}(1)>0$ and $\eta_y\to1$ as $x\to\infty$, which implies $v\cdot\eta_x>0$. The remaining inequality in $\mathcal{R}^+$ can be obtained by applying Theorem \ref{strong} to the harmonic function $\eta_x$.
\end{proof}

\begin{lemma}
Let $(\eta,\psi,\alpha)$ solve \eqref{stripstreameq} and $(u,v)$ be defined by \eqref{conformalvelocity}. Then $v$ solves the following problem of linear elliptic equations
    \begin{subequations}\label{vpde}
\begin{align}
&\Delta v+\gamma'(\psi)|\nabla\eta|^2 v=0\quad&&\mathrm{in~}\mathcal{R},\label{vpde1}\\
&{u(uv_y-vv_x)-v\eta_y(\alpha-\gamma(1)u)}=0&&\mathrm{on~}\Gamma,\label{vpde2}\\
&v=0&&\mathrm{on~}B.
\end{align}
\end{subequations}
\end{lemma}
\begin{proof}
  Taking the curl of \eqref{conformalvelequa1}-\eqref{conformalvelequa2}, we can easily obtain \eqref{vpde1}. Differentiating \eqref{conformalvelequa4} with respect to $x$ and multiplying it by $-u/2$, we obtain that
  \begin{equation}\label{uvupbc}
    -u(vv_x+uu_x+\alpha\eta_x)=0.
  \end{equation}
  Inserting \eqref{conformalvelequa1} into \eqref{uvupbc} yields
  \begin{equation*}
    {u(uv_y-vv_x)-u\eta_x(\alpha-u\gamma(1))}=0.
  \end{equation*}
  Finally, \eqref{conformalvelequa3} yields  \eqref{vpde2}.
\end{proof}

\begin{prop}\label{closedcondition}
	{\rm(}Closed condition{\rm)} Let $(\eta,\psi,\alpha)$ solve \eqref{stripstreameq} and $v$ be defined by \eqref{conformalvelocity}. If $v\leq 0$ on $\Gamma^+$ and $\gamma'(\psi)\leq 0$, then the strict inequality \eqref{nodalv} holds, or else $v\equiv0$.
\end{prop}
\begin{proof}
	Assume that $v\not\equiv0$.  Since $v$ vanishes along $B$ and as $x\to\infty$, the definition \eqref{conformalvelocity} together with the evenness of $\eta$ and $\psi$ implies that $v=0$ on $L$. Since $v\leq 0$ on $\Gamma^+$ and $\gamma'(\psi)\leq 0$, the strong maximum principle in Theorem \ref{strong} implies that $v<0$ in $\mathcal{R}^+$. Now we claim that $v<0$ on $\Gamma^+.$ If not, $v$ achieves its maximum of $0$ at some point $(x_0,1)\in\Gamma^+$. Then  at this point \eqref{vpde2} simplifies to  $u^2v_y=0$. The dynamic boundary condition \eqref{stripstreameq4} and \eqref{conformalvelequa4} implies that $u\not=0$ at $(x_0,1)$, so we must have $v_y=0$, which contradicts Theorem \ref{hopf}.
\end{proof}

Now we study the behavior of $v$ as $x\to\infty$. To this end, for $M>0$, we split $\mathcal{R}^+$ into a bounded rectangle of length $2M$ and an overlapping semi-infinite strip
\begin{equation*}
    \mathcal{R}_M^+:=\{(x,y)\in\mathcal{R}:x>M\},
\end{equation*}
with the corresponding analogues for boundary components
\begin{align*}
\Gamma_{M}^{+}:=\{(x,y)\in\Gamma:x>M\},\quad B_{M}^{+}:=\{(x,y)\in B:x>M\},\quad L_{M}^{+}:=\{(x,y)\in\mathcal{R}:x=M\}.
\end{align*}

\begin{lemma}\label{closelemma1}
Let $(\eta^*,\psi^*,\alpha^*)$ be a solution of problem \eqref{stripstreameq} with  the associated velocity field $(u^*,v^*)$ form \eqref{conformalvelocity}, $\alpha^*<\alpha_{\mathrm{cr}}$ and $v^*$ satisfies \eqref{nodalv}.  Then for any $M>0$, there exists $\varepsilon
    _M>0$ such that for all solutions $(\eta,\psi,\alpha)$ to \eqref{stripstreameq} with $\|\eta-\eta^{*}\|_{C^{3}(\mathcal{R})}+\|\psi-\psi^{*}\|_{C^{3}(\mathcal{R})}+|\alpha-\alpha^{*}|<\varepsilon_{M}$, the corresponding  vertical velocity $v$ satisfies
    \begin{equation}\label{vopen}
    v<0\quad\mathrm{in~}(\mathcal{R}\cup\Gamma)\cap\{0<x\leq 2M\}.
    \end{equation}
\end{lemma}
\begin{proof}
We begin by  considering the inequalities
\begin{subequations}\label{vstar}
    \begin{align}
        &v_x^*<0\quad\mathrm{on~}L^+\cup{\{(0,1)\}},\label{vstar1}\\
        &v_y^*<0 \quad\mathrm{on~}B^+\cap\{0<x\leq 2M\},\label{vstar2}\\
        &v_{xy}^*<0 \quad\mathrm{at~}(0,0).\label{vstar3}
    \end{align}
\end{subequations}
To prove \eqref{vstar1} and \eqref{vstar2}, we differentiate \eqref{vpde2} with respect to $x$ and get
\begin{equation}\label{vstarbc}
{(u^*)^2v^*_{xy}-u^*(v^*_x)^2
-\eta^*_y\big(\alpha^*-\gamma(1)u^*\big)v^*_x}=0
\quad\mathrm{at~}(0,1).
\end{equation}
Moreover, since $v$ satisfies \eqref{vpde1} and is odd in $x$, at $(0,1)$ we must  have  $v^*=v^*_{xx}=v^*_{yy}=v^*_y=0$. By \eqref{vstarbc}, $v^*_x=0$ at this point shall imply $v^*_{xy}=0$, which contradicts Theorem \ref{serrin}. Moreover, since $v^*$ vanishes along $L^+$, Theorem \ref{hopf} ensures that $v^*_x<0$ on  $L^+$. Similarly, since $v^*$ vanishes along $B^+$, Theorem \ref{hopf} gives $v^*_y<0$ on $B^+\cap\{0<x\leq 2M\}$.

{It remains to prove \eqref{vstar3} at $(0,0)$. Oddness gives
$v^*=v_y^*=0$ on $L^+$, while the equation gives the corresponding pure
second derivatives at the corner. Since $v_y^*(x,0)<0$ for $x>0$, one has
$v_{xy}^*(0,0)\le0$. If the equality holds, the first and mixed second-order
boundary derivatives required by Theorem \ref{serrin} would vanish, and
that lemma would contradict $v^*<0$ in the quadrant. Hence
$v_{xy}^*(0,0)<0$. }
Finally, by arguing in \cite{constantinstrauss07}, it can be shown that \eqref{vopen} with \eqref{vstar} defines an open set in $C^2(\mathcal{R}^+\cap\{0\leq x\leq 2M\})$. Thus, $v$ also satisfies \eqref{vopen} and \eqref{vstar} when $\varepsilon_M$ is sufficiently small.
\end{proof}

{
\begin{lemma}\label{closelemma2}
Let $I=[\alpha_-,\alpha_+]\subset\subset(0,\alpha_{\mathrm{cr}})$ and assume
$\gamma'(\psi)\le0$. There exists $\delta=\delta(I)>0$ such that, for every
solution with $\alpha\in I$, if for some $M>0$
\begin{equation*}
\|(\eta-y,\psi-\psi_{\mathrm{triv}})\|_{C^1(\mathcal R_M^+)}<\delta,
\qquad v\le0\quad\mathrm{on~}L_M^+,
\end{equation*}
then $v<0$ in $\mathcal R_M^+\cup\Gamma_M^+$, unless $v\equiv0$.
\end{lemma}
\begin{proof}
For the limiting laminar operator, let $\nu_0(\widetilde\alpha)>0$ be the principal
transverse eigenvalue in Lemma \ref{trivialkernel}. Compactness of $I$
gives $\inf_{\alpha\in I}\nu_0(\widetilde\alpha)>0$. Modify the  coefficient
$\widetilde\alpha$ in \eqref{ev2} by a fixed small number $\kappa>0$. The Rayleigh
quotient \eqref{minivar} and continuity show that the perturbed principal
eigenvalue is still uniformly positive. Let $h_\alpha>0$ be the
corresponding eigenfunction, normalized by $h_\alpha(1)=1$; thus $h_\alpha$ solves
\begin{equation}\label{tailbarrier}
h_\alpha''+\gamma'(\psi_{\mathrm{triv}})h_\alpha
=-\nu_0(\widetilde\alpha+\kappa)h_\alpha,\qquad
h_\alpha(0)=0,\qquad
h_\alpha'(1)=(\widetilde\alpha+\kappa)h_\alpha(1),
\end{equation}
where $\nu_0(\widetilde\alpha+\kappa)$ is uniformly positive for
$\alpha\in I$.

Choose a fixed $\varepsilon>0$ sufficiently small and set
$H_\alpha=h_\alpha+\varepsilon$. After reducing $\delta$ uniformly for
$\alpha\in I$, the $C^1$ tail closeness, $\gamma'\le0$, and
\eqref{tailbarrier} give
\begin{equation}\label{barrierineq}
H_\alpha''+\gamma'(\psi)|\nabla\eta|^2H_\alpha\le0
\quad\mathrm{in~}\mathcal R_M^+,
\end{equation}
and
\begin{equation}\label{barrierboundary}
u^2H_\alpha'(1)
-\eta_y\big(\alpha-\gamma(1)u\big)H_\alpha(1)>0
\quad\mathrm{on~}\Gamma_M^+.
\end{equation}
Indeed, at the laminar state the left-hand side of
\eqref{barrierboundary} equals
$\psi_{\mathrm{triv},y}^2(1)
(\kappa-\widetilde\alpha\varepsilon)$, which is uniformly positive if
$\varepsilon$ is sufficiently small. We may also decrease $\delta$ so that
$u>0$ on $\Gamma_M^+$.

Set $f=v/H_\alpha$. By \eqref{barrierineq}, it satisfies
\begin{equation*}
\Delta f+2\frac{H_\alpha'}{H_\alpha}f_y
+\frac{H_\alpha''+\gamma'(\psi)|\nabla\eta|^2H_\alpha}{H_\alpha}f=0
\quad\mathrm{in~}\mathcal R_M^+,
\end{equation*}
whose zeroth-order coefficient is nonpositive. On $\Gamma_M^+$,
\eqref{vpde2} becomes
\begin{equation*}
u^2H_\alpha f_y-uH_\alpha^2ff_x
+\big[u^2H_\alpha'-\eta_y(\alpha-\gamma(1)u)H_\alpha\big]f=0.
\end{equation*}
The function $f$ is nonpositive on $L_M^+$, vanishes on the bed, and tends
to zero at infinity. If it had a positive maximum, Theorem \ref{strong}
would impose it on $\Gamma_M^+$. There $f_x=0$ and Theorem \ref{hopf} gives
$f_y>0$, contradicting \eqref{barrierboundary}. Hence $v\le0$; Theorem \ref{strong} implies the strict sign unless $v\equiv0$.
\end{proof}
}

\begin{prop}\label{opencondition}
	{\rm(}Open condition{\rm)} Assume that $\gamma'(\psi)\leq 0$ and $(\eta^*,\psi^*,\alpha^*)$ is a solution of \eqref{stripstreameq} with $v^*$ from \eqref{conformalvelocity} satisfying \eqref{nodalv} and $\alpha^*<\alpha
	_{\mathrm{cr}}$. Then there exists $\varepsilon>0$ such that for all solution $(\eta,\psi,\alpha)$ to \eqref{stripstreameq} with $\|\eta^*-\eta\|_{C^3(\mathcal{R})}+\|\psi^*-\psi\|_{C^3(\mathcal{R})}+|\alpha
	^*-\alpha|\leq\varepsilon$, the corresponding vertical velocity $v$ also satisfies \eqref{nodalv}.
\end{prop}

\begin{proof}
	Fix $\alpha_0\in(\alpha^*,\alpha_{\mathrm{
			cr}})$ and choose $M>0$ such that $\|(\eta^*-y,\psi^*-\psi_{\mathrm{triv}})\|_{C^1(\mathcal{
			R}_M^+)}<\frac{\delta}{2}$, where $\delta=\delta
	(\alpha)$ is as in Lemma \ref{closelemma2}. Let $\varepsilon_M>0$ be chosen such that Lemma \ref{closelemma1} holds for $(\eta^*-y,\psi^*-\psi_{\mathrm{triv}})$. Let $\varepsilon=\min(\varepsilon
	_M,\frac{1}{2}\delta
	,|\alpha_0-\alpha^*|)$ be such that \eqref{vopen} holds. In particular, $v\leq0$ on $L_M^+$. Since $\|(\eta-y,\psi-\psi_{\mathrm{triv}})\|_{C^1(\mathcal{
			R}_M^+)}<\delta$, Lemma \ref{closelemma2} implies that $v<0$ in $\mathcal{
		R}_M^+\cup\Gamma_M^+$.
\end{proof}

\section{Flow Forces and Compactness}
In this section, we study an important invariant which is the so-called flow force in \cite{basu, bk, cwwz}.
In our case, the flow force is
given as
\begin{equation}
    \label{flowforcedef}
    S=\int_{x=\mathrm{const}}(P-P_{\mathrm{atm}}+U^2)dY-UVdX,
    \end{equation}
where the integral is taken over a vertical cross section of the fluid and the second term comes from the divergence form of the horizontal component \eqref{euler2} of the momentum equation,
\begin{equation*}
	(U^2+P-P_{\mathrm{atm}})_X+(UV)_Y=0.
\end{equation*}

\begin{lemma}\lb{sinvar}
    Assume that $(\psi,\eta,\alpha)$ solves \eqref{stripstreameq}. Then the flow force \eqref{flowforcedef} can be rewritten as
    \begin{equation}\lb{sxpea}
          S(x;\psi,\eta,\alpha)=\int_0^1\bigg(\frac{1}{2}(u^2-v^2)-\alpha(\eta-1)+(G(\psi)-G(1))+\frac{\mu}{2}\bigg)\eta_y+uv\eta_xdy,
    \end{equation}
    which is independent of $x$.
 \end{lemma}
{
\begin{proof}
In the nondimensional physical variables, Bernoulli's law reads
\[
P+\frac12(U^2+V^2)+\alpha(Y-1)-G(\Psi)
   =\frac{\mu}{2}+P_{\mathrm{atm}}-G(1),
\qquad
G(s)=-\int_0^s\gamma(t)\,dt.
\]
For fixed $x$, parameterize the corresponding cross section by
$y\mapsto H(x,y)$.  Since $\xi_y=-\eta_x$, direct substitution into
\eqref{flowforcedef} gives precisely \eqref{sxpea}.

We next differentiate \eqref{sxpea}.  The conformal map is harmonic, so
$\eta_{xx}=-\eta_{yy}$.  Using this identity and integrating the terms
containing $\eta_{xy}$ and $\eta_{yy}$ by parts in $y$, we obtain
\begin{equation}\label{Sdiff}
\begin{aligned}
\frac{dS}{dx}
={}&\left[
\left(\frac12(u^2-v^2)-\alpha(\eta-1)
       +G(\psi)-G(1)+\frac{\mu}{2}\right)\eta_x
-uv\eta_y\right]_{y=0}^{y=1}\\
&+\int_0^1\Big[
\eta_y\big\{u(u_x+v_y)+v(u_y-v_x)\big\}
+\eta_x\big\{v(u_x+v_y)-u(u_y-v_x)\big\}\\
&-\gamma(\psi)\psi_x\eta_y
+\gamma(\psi)\psi_y\eta_x
\Big]\,dy .
\end{aligned}
\end{equation}
The boundary term in \eqref{Sdiff} vanishes.  Indeed, on $B$ we have
$\eta_x=v=0$.  On $\Gamma$, where $\psi=1$, the dynamic boundary
condition gives
\[
\frac12(u^2-v^2)-\alpha(\eta-1)+\frac{\mu}{2}=u^2.
\]
Hence the upper boundary  is
$u^2\eta_x-uv\eta_y=u(u\eta_x-v\eta_y)=0$ by the kinematic boundary
condition.

It remains to simplify the integral.  The velocity equations  satisfies \eqref{conformalvelequa1}-\eqref{conformalvelequa2}.
Therefore the first two lines of the integrand in \eqref{Sdiff} equal
\[
-\gamma(\psi)v(\eta_x^2+\eta_y^2).
\]
On the other hand, using \eqref{psixpsiy}, the remaining two terms equal
\[
-\gamma(\psi)\psi_x\eta_y+\gamma(\psi)\psi_y\eta_x
=\gamma(\psi)v(\eta_x^2+\eta_y^2).
\]
Thus we can deduce
$dS/dx=0$.
\end{proof}
}

{On the unbounded strip, bounded sets in the H\"older spaces are not necessarily
compact. For monotone elevation waves, loss of compactness can produce a
translated limit joining two laminar flows. We use the two invariants to rule
out such a limit.}

A bore is a solution $(\psi,\eta,\alpha)$ of \eqref{stripstreameq} and satisfies \[\lim_{x\to\pm\infty}(\eta,\psi)(x,y)=(\eta_{\pm}(y),\psi_{\pm}(y)),\]
 with distinct limits $(\eta_-,\psi_-)$ and $(\eta_+,\psi_+)$. {For the monotone translated limits considered below, the two end states solve \eqref{stripstreameq} and are laminar (as verified in the proof of Lemma \ref{cptlemma}).} Then they are of the form
 \begin{equation*}
     \eta_{\pm}(y)=\hat{\eta}_{\mathrm{triv}}(y,d_{\pm})\quad\mathrm{and}\quad\psi_{\pm}(y)=\hat{\psi}_{\mathrm{triv}}(y,d_{\pm}),
 \end{equation*}
{For a nontrivial bore in the class studied here, the two depths
have to be distinct.} It is easy to compute that \begin{equation}\label{hateta}\hat{\eta}_{\mathrm{triv}}(y;d)=dy,\end{equation} and $\hat{\psi}_{\mathrm{triv}}(y;d)$ satisfies the following problem
\begin{subequations}\label{hatpsi}
	\begin{align}
		&\hat{\psi}_{\mathrm{triv},yy}=-\gamma(\hat\psi_{\mathrm{triv}})d^2\quad&&\mathrm{in~}(0,1),\label{hatpsi1}\\
		&\hat\psi_{\mathrm{triv}}=1 \quad&&\mathrm{on~}y=1,\label{hatpsi2}\\
		& \hat\psi_{\mathrm{triv}}=0 \quad&&\mathrm{on~}y=0,\label{hatpsi3}\\
		&\hat\psi_{\mathrm{triv},y}^2(1)={(\mu-2\alpha(d-1))d^2} \quad&&\mathrm{on~}y=1.\label{hatpsi4}
	\end{align}
\end{subequations}
For convenience, in the following analysis, we write $\hat{\psi}_{\mathrm{triv}}(y;d)$ and $\hat\eta_{\mathrm{triv}}(y;d)$ {as} $\hat\psi(y)$ and $\hat\eta(y)$. Since $\hat\psi$ and $\hat\eta$ also solve the dynamic boundary condition, we obtain
\begin{equation}\label{conju1}
	\hat Q(d_-)=\hat Q (d_+)=\hat Q(1),
\end{equation}
with $$\hat Q(d)=\frac{\hat\psi_y^2(1)}{d^2}+2\alpha(d-1){,\qquad\hat Q(1)=\mu}.$$
For bores obtained as translated limits of the solitary waves considered here, the invariance of the flow force in Lemma \ref{sinvar} implies that
\begin{equation}\label{conju2}
	\hat S(d_-)=\hat S(d_+)=\hat S(1),
\end{equation}
with
\begin{equation*}
	\hat S(d):=S(x;\hat\eta,\hat\psi).
\end{equation*}
{Equations \eqref{conju1} and \eqref{conju2} are well known as the conjugate-flow
equations.  We impose the following condition on the prescribed
vorticity and its laminar flows:}
{
\begin{equation}\label{conjugatehyp}
\tag{$H_{\mathrm{cf}}$}
\begin{gathered}
\text{For every }{0<\alpha\leq\alpha_{\mathrm{cr}}},\text{ every laminar
solution of \eqref{hatpsi}}\\
\text{with }\hat S(d)=\hat S(1)\text{ satisfies }
\hat\psi_y(y)>0\text{ for }0\le y\le1.
\end{gathered}
\end{equation}
When $d=1$, the assumption \eqref{conjugatehyp} is \eqref{nocritical}.

\begin{example}
The affine vorticity $\gamma(\psi)=\lambda\psi$ with $\lambda<0$,
already considered in Section 3, satisfies $(H_{\mathrm{cf}})$. For every
$d>0$, the Dirichlet problem \eqref{hatpsi1}--\eqref{hatpsi3} has the
unique solution
\begin{equation*}
\hat\psi(y;d)=
\frac{\sinh(d\sqrt{-\lambda}\,y)}{\sinh(d\sqrt{-\lambda})},
\qquad
\hat\psi_y(y;d)=
\frac{d\sqrt{-\lambda}\,\cosh(d\sqrt{-\lambda}\,y)}
{\sinh(d\sqrt{-\lambda})}>0.
\end{equation*}
Thus every possible conjugate flow, {including one at the critical
parameter}, lies in the same unidirectional
Dirichlet family as the base state $d=1$.
\end{example}
}

{Multiplying \eqref{hatpsi1} by $\hat\psi_y$ and integrating by parts gives}
\begin{equation}\label{inteid0}
{\frac{\hat\psi_y^2(y)}{d^2}-2G(\hat\psi(y))
=\frac{\hat\psi_y^2(1)}{d^2}-2G(1).}
\end{equation}
{Under the assumption \eqref{conjugatehyp}, \eqref{inteid0} is equivalent to}
\begin{equation}\label{inteid1}
{\hat\psi_y(y)=d\big[\hat Q(d)-2\alpha(d-1)
-2G(1)+2G(\hat\psi(y))\big]^{1/2}>0.}
\end{equation}
{Applying the change variable $s=\hat\psi(y)$, we deduce that $d$ satisfies the following relation}
\begin{equation*}
{d=\int_0^1
\frac{ds}{\sqrt{\hat Q(d)-2\alpha(d-1)-2G(1)+2G(s)}}.}
\end{equation*}

\begin{lemma}\label{convexlemma}
Assume \eqref{conjugatehyp} and $0<\alpha\leq\alpha_{\mathrm{cr}}$.
Define the interval of depths of strictly increasing laminar flows by
\[
\mathcal D:=\left\{\int_0^1\frac{ds}{\sqrt{K+2G(s)}}:
K>-2\min_{[0,1]}G\right\}.
\]
For $d\in\mathcal D$, let $\hat\psi(\cdot;d)$ be the strictly increasing
solution of the Dirichlet problem \eqref{hatpsi1}--\eqref{hatpsi3},
and define $\hat Q(d)$ and $\hat S(d)$ from this family as above.
The dynamic condition \eqref{hatpsi4} holds precisely when
$\hat Q(d)=\mu$. Then $1\in\mathcal D$, $\hat Q$ is strictly convex
on $\mathcal D$, and
\[
\hat Q'(1)=2(\alpha-\alpha_{\mathrm{cr}}).
\]
If $\alpha<\alpha_{\mathrm{cr}}$, any other root
$d_*\in\mathcal D$ of $\hat Q(d_*)=\hat Q(1)$ is unique, satisfies
$d_*>1$, and has $\hat Q(d)<\mu$ for $1<d<d_*$.
If $\alpha=\alpha_{\mathrm{cr}}$, there is no other root in $\mathcal D$.
\end{lemma}
\begin{proof}
By the first integral \eqref{inteid0}, the strictly increasing
Dirichlet family is parameterized by
\[
d=\int_0^1\frac{ds}{\sqrt{K(d)+2G(s)}},\qquad
\hat Q(d)=K(d)+2G(1)+2\alpha(d-1).
\]
The first integral decreases strictly with $K$, so its range
$\mathcal D$ is an interval and $K(d)$ is well defined and smooth.
Assumption \eqref{nocritical} gives $1\in\mathcal D$ and
$K(1)=\mu-2G(1)$. Differentiating twice yields
\[
K'(d)=-\frac{2}{\displaystyle\int_0^1(K(d)+2G(s))^{-3/2}\,ds},
\qquad
K''(d)=\frac{3(K'(d))^2\displaystyle\int_0^1(K(d)+2G(s))^{-5/2}\,ds}
{2\displaystyle\int_0^1(K(d)+2G(s))^{-3/2}\,ds}>0.
\]
Changing variables $s=\psi_{\mathrm{triv}}(y)$ at $d=1$, we obtain
\[
\int_0^1(K(1)+2G(s))^{-3/2}\,ds
=\int_0^1\frac{dy}{\psi_{\mathrm{triv},y}^2(y)}
=\frac1{\alpha_{\mathrm{cr}}}.
\]
Hence $\hat Q''=K''>0$ on $\mathcal D$ and
$\hat Q'(1)=2(\alpha-\alpha_{\mathrm{cr}})$.
The assertions about possible other roots follow from strict convexity.
\end{proof}
\begin{lemma}\label{ffd1}
Assume \eqref{conjugatehyp} and $0<\alpha\leq\alpha_{\mathrm{cr}}$. For every $d\in\mathcal D$, we have
	\[\hat{S}'(d)= \frac{1}{2}\left( \mu - \hat{Q}(d) \right).\]
	In particular, if $\alpha<\alpha_{\mathrm{cr}}$ and a second depth $d_*\in\mathcal D$ in Lemma \ref{convexlemma} exists, then
	$\hat{S}(d_*)>\hat{S}(1)$.
\end{lemma}
\begin{proof}
	Substituting $\hat\eta=dy$ and $u=\hat\psi_y/d$ from
	\eqref{conformalvelocity} and \eqref{hateta} into \eqref{sxpea} gives
	\begin{equation}\label{shat}
		\hat{S}(d) = \int_{0}^{1} d\left( \frac{\hat{\psi}_y^2}{2d^2} - \alpha(dy - 1) + G(\hat{\psi}) - G(1) + \frac{\mu}{2} \right)dy.
	\end{equation}
	By \eqref{inteid0}, we have
	\begin{equation*}
		\begin{aligned}
			\hat{S}(d)& = d \int_{0}^{1}\frac{\hat\psi_y^2(1)}{2d^2}+2G(\hat\psi)-2G(1)-\alpha(dy-1)+\frac{\mu}{2}dy\\
			&=d\int_0^1\frac{1}{2}\hat{Q}(d)+2\alpha-\alpha d-\alpha dy+\frac{\mu}{2}+2(G(\hat\psi)-G(1))dy\\
			&=\frac{1}{2}d\hat{Q}(d)-\frac{3}{2}\alpha d^2+(2\alpha+\frac{\mu}{2})d+2d\int_0^1 G(\hat\psi)-G(1)dy.
		\end{aligned}
	\end{equation*}
	To compute $\hat{S}'(d)$, we define $L(y,\hat\psi,\hat\psi_y,d)$ as
	\begin{equation*}
		L(y,\hat\psi,\hat\psi_y,d)=d\left( \frac{\hat{\psi}_y^2}{2d^2} - \alpha(dy - 1) + G(\hat{\psi}) - G(1) + \frac{\mu}{2} \right),
	\end{equation*}
	and differentiate  \eqref{shat} with respect to $d$, we have
	\begin{equation*}
		\hat{S}'(d)=\int_0^1\frac{\partial L}{\partial\hat\psi}\frac{\partial\hat\psi}{\partial d}+\frac{\partial L}{\partial\hat\psi_y}\frac{\partial\hat\psi_y}{\partial d}+\frac{\partial L}{\partial d}dy,
	\end{equation*}
	Notice that \[\frac{\partial\hat\psi_y}{\partial d}=\frac{\partial}{\partial y}\bigg(\frac{\partial\hat\psi}{\partial d}\bigg)\]
	and integrate by parts, we have
	\begin{equation}\label{hats1}
		\hat{S}'(d)=	\int_0^1\left[\frac{\partial L}{\partial \hat\psi}-\frac{\partial}{\partial y}\left(\frac{\partial L}{\partial\hat\psi_y}\right)\right]\frac{\partial\hat\psi}{\partial d}dy+\left[\frac{\partial L}{\partial\hat\psi_y}\frac{\partial \hat\psi}{\partial d}\right]_{y=0}^{y=1}+\int_0^1\frac{\partial L}{\partial d}dy.
	\end{equation}
	We show that the first and second term of \eqref{hats1} vanish. It is easy to verify that
	\begin{equation*}
		\frac{\partial L}{\partial\hat\psi}=d G'(\hat\psi)
	\end{equation*}
	and
	\begin{equation*}
		\frac{\partial L}{\partial \hat\psi_y}=\frac{\hat\psi_y}{d},\quad \frac{\partial}{\partial y}\left(\frac{\partial L}{\partial\hat \psi_y}\right)=\frac{\hat\psi_{yy}}{d},
	\end{equation*}
	together with \eqref{hatpsi1} we have
	\begin{equation*}
		\frac{\partial L}{\partial \hat{\psi}} - \frac{d}{dy} \left( \frac{\partial L}{\partial \hat{\psi}_y} \right) = -\frac{1}{d} \left( \hat{\psi}_{yy} - d^2 G'(\hat{\psi}) \right)=0.
	\end{equation*}
	Moreover, together with \eqref{hatpsi2} and \eqref{hatpsi3}, it is easy to see that the second term of \eqref{hats1} vanishes.
	Then we have
	\begin{equation*}
		\hat{S}'(d)=\int_0^1\frac{\partial L}{\partial d}dy=\int_0^1\left(-\frac{\hat{\psi}_y^2}{2d^2} + G(\hat{\psi}) - 2\alpha dy + \alpha - G(1) + \frac{\mu}{2}\right)dy
	\end{equation*}
	By \eqref{inteid0}, we have
	\begin{equation*}
		\begin{aligned}
			\hat{S}'(d)&=\int_0^1\left(-\frac{\hat{\psi}_y(1)^2}{2d^2} + G(1) - 2\alpha dy + \alpha - G(1) + \frac{\mu}{2}\right)dy\\& =\int_0^1 \left(-\frac{\hat{\psi}_y(1)^2}{2d^2} - 2\alpha dy + \alpha + \frac{\mu}{2}\right)dy\\
			&=-\frac{\hat{\psi}_y(1)^2}{2d^2} - \alpha (d-1) + \frac{\mu}{2}\\
			&=\frac{1}{2} \left( \mu - \hat{Q}(d) \right).
		\end{aligned}
	\end{equation*}	
	Finally, if $d_*$ exists, Lemma \ref{convexlemma} gives
	\begin{equation*}
		\hat{S}(d_*)-\hat S(1)= \int_1^{d_*}\frac{1}{2}(\mu-\hat Q(d))\,dd>0.
	\end{equation*}
	
\end{proof}

Based on the Lemma \ref{convexlemma}-\ref{ffd1}, we can obtain the following result on non-existence of bores.
\begin{lemma}\label{borenonexistence}
Assume \eqref{conjugatehyp} and $0<\alpha\leq\alpha_{\mathrm{cr}}$.
The conjugate-flow equations \eqref{conju1} and \eqref{conju2}
have no solutions other than $d=1$. Consequently, there are no
nontrivial bores with these invariants, including those obtained
as translated limits of the solitary waves considered here.
\end{lemma}
\begin{proof}
By \eqref{conjugatehyp}, every candidate conjugate flow is strictly
increasing and hence belongs to the Dirichlet family with
$d\in\mathcal D$ in Lemma \ref{convexlemma}.
If $\alpha=\alpha_{\mathrm{cr}}$, that lemma already gives $d=1$.
If $\alpha<\alpha_{\mathrm{cr}}$ and $d\ne1$, the same lemma gives
$d>1$ and $\hat Q(t)<\mu$ for $1<t<d$. Lemma \ref{ffd1} then yields
\[
\hat S(d)-\hat S(1)=\frac12\int_1^d(\mu-\hat Q(t))\,dt>0,
\]
contrary to \eqref{conju2}. Thus both end states have depth $1$
and coincide with the base laminar flow, ruling out a nontrivial bore.
\end{proof}
We are now ready to recover compactness.
\begin{lemma}\label{cptlemma}
Assume \eqref{conjugatehyp} and $\gamma'\leq0$ on the common attained
range. Let $(\psi_n,\eta_n,\alpha_n)$ be solutions of
\eqref{stripstreameq}, with
$\alpha_n\in[\alpha_-,\alpha_{\mathrm{cr}}]$ for some $\alpha_->0$.
Suppose that
\begin{equation}\label{seqbound}
\sup_n\|(\psi_n,\eta_n)\|_{C^{3+\beta}(\mathcal R)}<\infty,
\qquad
\inf_n\inf_{\mathcal R}
(\mu-2\alpha_n(\eta_n-1))|\nabla\eta_n|^2>0,
\end{equation}
and
\begin{equation}\label{monotoneass}
\partial_x\eta_n\leq0\qquad\mathrm{for}\ x\geq0.
\end{equation}
Then a subsequence of $(\psi_n,\eta_n)$ converges in
$C_b^{3+\beta}(\overline{\mathcal R})$.
\end{lemma}
\begin{proof}
After passing to a subsequence, let
$\alpha_n\to\alpha\in[\alpha_-,\alpha_{\mathrm{cr}}]$.
We first prove the uniform tail condition
\begin{equation}\label{seqasy}
\lim_{x\to\pm\infty}\sup_n\sup_{0\leq y\leq1}
|(\psi_n,\eta_n)(x,y)-(\psi_{\mathrm{triv}}(y),y)|=0.
\end{equation}
Monotonicity, symmetry and the far-field limit imply
$\eta_n(x,1)\geq1$. If the surface traces do not tend to $1$
uniformly in $n$, there are $x_n\to\infty$ and $\delta>0$ such that
$\eta_n(x_n,1)\geq1+\delta$. The bounds \eqref{seqbound} give a
subsequence of the translates converging in $C^3_{\mathrm{loc}}$ to
$(\psi_*,\eta_*)$. This limit satisfies the equations and boundary
conditions in \eqref{stripstreameq}, and $\partial_x\eta_*\leq0$
on the whole strip. Hence its surface trace has limits
$d_-\geq d_+\geq1$, with $d_-\geq1+\delta$.

Any sequence of further translates towards either end has a
subsequence converging locally to a solution with
$\eta=d_\pm y$: this follows from the surface limits and uniqueness
of the bounded harmonic Dirichlet problem on the strip.
For fixed $d$, the bounded semilinear Dirichlet problem
\[
\Delta p=-d^2\gamma(p),\qquad p|_B=0,\quad p|_\Gamma=1
\]
has at most one solution taking values in the common attained range. Indeed, the difference of two solutions
satisfies $\Delta z+c(x,y)z=0$, with $c\leq0$ and zero boundary
traces, so the bounded maximum principle on the strip gives $z=0$.
It follows by horizontal translation that this solution is laminar.
Thus the end states above are laminar flows of depths $d_\pm$.
The dynamic boundary condition and Lemma \ref{sinvar}, passed to
the limit, give
\[
\hat Q(d_\pm)=\mu,\qquad \hat S(d_\pm)=\hat S(1).
\]
Here the second equality holds because each original solitary wave
has flow force $\hat S(1)$ at its parameter $\alpha_n$.
Lemma \ref{borenonexistence}, including its critical endpoint,
therefore gives $d_-=d_+=1$, a contradiction.

The surface tails are consequently uniform. Another translation
argument and bounded harmonic uniqueness give uniform convergence
of $\eta_n-y$ throughout the strip. If the tails of
$\psi_n-\psi_{\mathrm{triv}}$ were not uniform, translates at points
violating uniform decay would converge locally to a bounded solution
of the preceding Dirichlet problem with $d=1$. Its uniqueness forces
this limit to equal $\psi_{\mathrm{triv}}$, again a contradiction.
Symmetry treats the negative tail and proves \eqref{seqasy}.

Arzel\`a--Ascoli and \eqref{seqasy} now give, along a subsequence,
$(\psi_n,\eta_n)\to(\psi,\eta)$ locally in $C^3$ and globally in
$C^0$. The limit inherits the $C_b^{3+\beta}$ bounds and solves the
equations at $\alpha$. Write
$v_1=\psi_n-\psi$, $v_2=\eta_n-\eta$ and
$(\psi_t,\eta_t)=(\psi+t v_1,\eta+t v_2)$.
The mean value formula gives
\[
\begin{cases}
\Delta v_1+a_n v_1+b_n\cdot\nabla v_2=0,
\quad \Delta v_2=0&\mathrm{in}\ \mathcal R,\\
A_n\partial_yv_1+B_n\cdot\nabla v_2+C_n v_2=f_n,
\quad v_1=0&\mathrm{on}\ \Gamma,\\
v_1=v_2=0&\mathrm{on}\ B,
\end{cases}
\]
where
\[
\begin{aligned}
a_n&=\int_0^1\gamma'(\psi_t)|\nabla\eta_t|^2\,dt,
&b_n&=2\int_0^1\gamma(\psi_t)\nabla\eta_t\,dt,\\
A_n&=2\int_0^1\partial_y\psi_t\,dt,
&B_n&=-2\int_0^1(\mu-2\alpha_n(\eta_t-1))\nabla\eta_t\,dt,\\
C_n&=2\alpha_n\int_0^1|\nabla\eta_t|^2\,dt,
&f_n&=-2(\alpha_n-\alpha)(\eta-1)|\nabla\eta|^2.
\end{aligned}
\]
All coefficients have the uniform bounds required by
Lemma \ref{schauderlemma2}. Interpolation gives global $C^1$
convergence, and hence
\[
B_n\longrightarrow-2(\mu-2\alpha(\eta-1))\nabla\eta
\qquad\mathrm{uniformly\ on}\ \Gamma.
\]
The lower bound in \eqref{seqbound} and the uniform gradient bound
imply $\inf_\Gamma|B_n|>c>0$ for all sufficiently large $n$.
Since the zero-order boundary condition is $v_1=0$, its complementing
determinant in \eqref{lopatinskii} is precisely $|B_n|^2$.
Lemma \ref{schauderlemma2} therefore yields
\[
\|(v_1,v_2)\|_{C_b^{3+\beta}}
\leq C\bigl(\|f_n\|_{C_b^{2+\beta}(\Gamma)}
             +\|(v_1,v_2)\|_{C_b^0}\bigr)\longrightarrow0.
\]
This proves the assertion.
\end{proof}
\subsection{Bounds of the Froude numbers}

{In this subsection we derive a conformal integral identity
to rule out nontrivial critical and subcritical waves under the hypotheses of Proposition \ref{Dcriticalnonnegative}.}

Recall \eqref{inteid1}, the first integral of the laminar equation is given by
\begin{equation*}
	\psi_{\mathrm{triv},y}(y)
	=\sqrt{\mu-2G(1)+2G(\psi_{\mathrm{triv}}(y))},
\end{equation*}
with 
\begin{equation*}
	\int_0^1\frac{dt}{\sqrt{\mu-2G(1)+2G(t)}}=1,\qquad
	\int_0^1\frac{dt}{[\mu-2G(1)+2G(t)]^{3/2}}
	=\frac1{\alpha_{\mathrm{cr}}}.
\end{equation*}

Inspired by \cite{wjfm}, define the laminar speed in stream-function
coordinates by
\[
 a(p):=\psi_{\mathrm{triv},y}(\psi_{\mathrm{triv}}^{-1}(p))
       =\sqrt{\mu-2G(1)+2G(p)},\qquad 0\le p\le1.
\]
On the portion of the strip where $0\leq\psi\leq1$, define
\begin{equation*}
	h_1(\psi)=\int_0^\psi\frac{1}{a(t)}dt,\qquad h_2(\psi)=\int_0^\psi\frac{1}{a(t)^3}dt,
\end{equation*}
and
\begin{equation*}
	\begin{aligned}
		&H_1(x,y)=-uv\eta_y h_2(\psi)
		+ \left[
		\frac{u^2 - v^2 - a(\psi)^2}{2}h_2(\psi)
		+ \eta -h_1(\psi)
		\right] \eta_x,\\
		&{H_2(x,y)=uv\eta_x h_2(\psi)}
		+ \left[
		\frac{u^2 - v^2 - a(\psi)^2}{2}h_2(\psi)
		+ \eta -h_1(\psi)
		\right] \eta_y.
	\end{aligned}
\end{equation*} 
Together with $a'(p)=-\gamma(p)/a(p)$ and \eqref{conformalvelequa1}-\eqref{conformalvelequa2}, we have
\begin{equation}\label{diverh}
	\nabla\cdot(H_1,H_2)=\left[ \frac{u + 2a(\psi)}{2a(\psi)^3} (u - a(\psi))^2 + \frac{u v^2}{2a(\psi)^3} \right] |\nabla \eta|^2.
\end{equation}

% On the full strip, harmonicity also gives
% int w eta_y = int w + int |grad(eta-y)|^2 (up to vanishing side terms).
% This removes eta_y from the left side of (6.21), but its volume
% density still has no definite sign where u<0. The one-sided
% identity below keeps the essential lateral boundary term instead.

\begin{prop}\label{Dcriticalnonnegative}
{Assume $\gamma'\leq0$ and let $(\eta,\psi,\alpha)$ be a symmetric solitary
solution of \eqref{stripstreameq} with $\alpha>0$, satisfying
the nodal elevation property \eqref{nodalproperty}. Then either the wave is trivial
or $\alpha<\alpha_{\mathrm{cr}}$.}
\end{prop}
\begin{proof}
{

Suppose that the wave is nontrivial. The dynamic condition and the
shear flow limit give $\psi_y>0$ on $\Gamma$. Since $\psi_x=0$ on
$\Gamma$, \eqref{conformalvelocity} and the nodal property imply
$v<0$ there. Proposition \ref{closedcondition} then gives
$v<0$ throughout $\mathcal R^+$. Moreover, the monotonicity of
$\eta(x,1)$ and its limit at infinity give $w>0$ on $\Gamma^+$.

By \eqref{velasy} and the shear flow limit of the conformal map,
choose $M$ sufficiently large that $u>0$ and $\eta_y>0$ on
$\{\xi\ge M\}$. For $N>M$, set
\[
\mathcal D_M^N:=\{(x,y)\in\mathcal R:M<\xi(x,y)<N\}.
\]
The Cauchy--Riemann equations give
\[
\nabla\xi=(\eta_y,-\eta_x),\qquad
\nabla\xi\cdot\nabla\eta=0,
\qquad |\nabla\xi|=|\nabla\eta|.
\]
Each level set of $\xi$ in this region joins $B$ to $\Gamma$, and
its upward tangent has the direction $\nabla\eta$. By
\eqref{psixpsiy},
\[
\nabla\psi\cdot\nabla\eta=u|\nabla\eta|^2>0.
\]
Consequently, $\psi$ increases from $0$ to $1$ along these level
sets. Thus $0\leq\psi\leq1$ in $\mathcal D_M^N$ and
$h_2(\psi)>0$ away from the bed.

Taking the scalar product of the definition of $(H_1,H_2)$ with
$\nabla\xi$,
we obtain
\begin{equation}\label{criticalnormal}
(H_1,H_2)\cdot\nabla\xi
=-uvh_2(\psi)|\nabla\eta|^2>0
\qquad\mathrm{in}\ \mathcal D_M^N.
\end{equation}

We now apply the divergence theorem on $\mathcal D_M^N$.
On $B$, the definition of $H_2$ and
$\psi=\eta=0$ give $H_2=0$, and on $\Gamma$, the kinematic and
dynamic conditions imply
\[
uv\eta_x+\frac12(u^2-v^2-\mu)\eta_y
=-\alpha w\eta_y.
\]
Since $h_1(1)=1$, $h_2(1)=1/\alpha_{\mathrm{cr}}$, and
$\eta=1+w$ on $\Gamma$, it follows that
\[
H_2=\left(1-\frac{\alpha}{\alpha_{\mathrm{cr}}}\right)
w\eta_y.
\]
Furthermore, $d\xi=\xi_x\,dx=\eta_y\,dx$ along the upper
boundary. The outward unit normals on $\{\xi=N\}$ and
$\{\xi=M\}$ are
$\nabla\xi/|\nabla\xi|$ and
$-\nabla\xi/|\nabla\xi|$ respectively. Notice \eqref{criticalnormal}, together with the divergence theorem, we have the integral identity
\begin{equation}\label{criticaldomainidentity}
\begin{aligned}
\left(1-\frac{\alpha}{\alpha_{\mathrm{cr}}}\right)
\int_{\Gamma\cap\mathcal D_M^N}w\,d\xi
={}&\iint_{\mathcal D_M^N}
\nabla\cdot(H_1,H_2)\,dx\,dy\\
&+\int_{\{\xi=M\}}
   -uvh_2(\psi)|\nabla\eta|\,dS
-\int_{\{\xi=N\}}
   -uvh_2(\psi)|\nabla\eta|\,dS.
\end{aligned}
\end{equation}

It is easy to verify that the integral over $\{\xi=M\}$ is strictly
positive, whereas the one over $\{\xi=N\}$ tends to zero as
$N\to\infty$ by the solitary-wave decay. If
$\alpha\geq\alpha_{\mathrm{cr}}$, the left-hand side of
\eqref{criticaldomainidentity} is nonpositive because $w>0$
and $d\xi>0$ on $\Gamma^+$. Its right-hand side is strictly
positive for all sufficiently large $N$, which is a contradiction. Thus we can deduce that $\alpha<\alpha_{\mathrm{cr}}$.

}
\end{proof}

\section{Global Continuum and Main Result}
First we state the following modified version of the global bifurcation theorem, which is our main tool to prove the global continuum.

\begin{theorem}\label{globalbifurthm}\cite[Appendix B, Theorem B.1]{susannaarma}
    Let $\mathscr{X}$ and $\mathscr{Y}$ be Banach spaces, $\mathscr{U}$ be an open subset of $\mathscr{X}\times\mathbb{R}$ with $(0,0)\in\partial\mathscr{U}$. Consider a real-analytic mapping $\mathcal{F}\colon\mathscr{U}\to\mathscr{Y}$. Suppose that
\begin{enumerate}
\item \label{semifredholm} for any $(\mu,x)\in\mathscr{U}$ with $\mathcal{F}(\mu,x)=0$ the Fr\'echet derivative $\mathcal{F}_x(\mu,x)\colon\mathscr{X}\to\mathscr{Y}$ is locally proper;
		\item there exists a continuous curve $\mathscr{C}_\textup{loc}$ of solutions to $\mathcal{F}(\mu,x)=0,$ parameterized as
		\begin{equation*}
		\mathscr{C}_{\textup{loc}}:=\{(\mu,\tilde{x}(\mu)):0<\mu<\mu_* \}\subset\mathcal{F}^{-1}(0),
		\end{equation*}
		for some $\mu_*>0$ and continuous $\tilde{x}$ with values in $\mathscr{X}$ and $\lim_{\mu\searrow0}\tilde{x}(\mu)=0$;
		\item \label{invertible} the linearized operator $\mathcal{F}_x(\mu,\tilde{x}(\mu))\colon\mathscr{X}\to\mathscr{Y}$ is invertible for all $\mu$.
\end{enumerate}
Then $\mathscr{C}_\text{loc}$ is contained in a curve of solutions $\mathscr{C}$,  parameterized as
	\begin{equation*}
		\mathscr{C}:=\{(\mu(s),x(s)):0<s<\infty \}\subset\mathcal{F}^{-1}(0)
	\end{equation*}
	for some continuous $(0,\infty)\ni s\mapsto(x(s),\mu(s) )\in\mathscr{U},$ with the following properties
	\begin{enumerate}
	\item [(i)]One of the following alternatives holds:
		\begin{enumerate}
			\item \label{blow-up}  \textup{(Blow-up)}
	\begin{equation*}
		N(s):=\|x(s)\|_\mathscr{X}+\frac{1}{\operatorname{dist}((\mu(s),x(s)),\partial\mathscr{U})}+|\mu(s)|\to\infty,\quad s\to\infty.
	\end{equation*}
			\item \label{loss of compactness} \textup{(Loss of compactness)} there exists a sequence $s_n\to\infty$ such that $\operatorname{sup}_n N(s_n)<\infty$ but $\{x(s_n)\}$ has no subsequences converging in $\mathscr{X}$.
		\end{enumerate}
	\item [(ii)] Near each point $(\mu(s_0),x(s_0))\in\mathscr{C}$, we can reparameterize $\mathscr{C}$ so that $s\mapsto(\mu(s),x(s))$ is real analytic.
	
	\item[(iii)] $(\mu(s),x(s))\not\in\mathscr{C}_{\text{loc}}$ for $s$ sufficiently large.
	\end{enumerate}
\end{theorem}

{In this application the Banach spaces are
$\mathscr X=\mathcal X_{\mathrm e}$, $\mathscr Y=\mathcal Y_{\mathrm e}$,
and the open set is $\mathcal U_{\mathrm e}$.}
Now we apply Theorem \ref{globalbifurthm} to extend the local curve of the solution $\mathscr{C}_{\mathrm{loc}}$ obtained in  Theorem \ref{smallthm} to a global curve of solution $\mathscr{C}$.
\begin{theorem}\label{globalthm}
Assume that $\gamma$ is real-analytic and satisfies all conditions of Theorem \ref{smallthm}. Then the local curve $\mathscr{C}_\textup{loc}$ is contained in a continuous curve of solutions parameterized as
	\begin{equation*}
		\mathscr{C}=\{(\phi(s),w(s),\alpha(s)):0<s<\infty \}\subset\mathcal{U},
	\end{equation*}
with the following properties:
    \begin{enumerate}
        \item [$(a)$] {One of the following alternatives holds:}
        \begin{enumerate}
            \item [$(i)$] as $s\to\infty$,
            \begin{equation*}
                N(s):={\|(\phi(s),w(s))\|_{\mathcal{X}}}+\frac{1}{\sigma(w(s),\alpha(s))}+\frac{1}{\alpha(s)}+\frac{1}{\alpha_{\mathrm{cr}}-\alpha(s)}\to\infty,
            \end{equation*}
 \item [$(ii)$] there exists a sequence $s_n\to\infty$ such that $\sup_{n\in\mathbb{N}}N(s_n)<\infty$ but $\{(\phi(s_n),w(s_n))\}$ has no subsequences converging in $\mathcal{X}$.
        \end{enumerate}
    \item [$(b)$] {Near each point}  $(\phi(s),w(s),\alpha(s))\in\mathscr{C}$, we can re-parameterize $\mathscr{C}$ so that the mapping $s\to(\phi(s),w(s),\alpha(s))$ is real analytic.
     \item [$(c)$] $(\phi(s),w(s),\alpha(s))\not\in\mathscr{C}_{\mathrm{loc}}$ for $s$ sufficiently large.
\end{enumerate}
\end{theorem}
\begin{proof}
The operator $\mathcal{F}:\mathcal{U}\to\mathcal{Y}$ is real analytic. Theorem \ref{smallthm} ensures the existence of the local curve $\mathscr{C}_{\mathrm{loc}}=\{(\phi^\varepsilon,w^\varepsilon,\alpha^\varepsilon):0<\varepsilon{\ll}1\}\subset\mathcal{U}$.  Lemma \ref{localproper} shows that the linearized operator $\mathcal{F}_{(\phi,w)}(\phi,w,\alpha
)$ is locally proper for $(\phi,w,\alpha)\in\mathcal{U}$. Finally, it follows from the property $(iii)$ of Theorem \ref{smallthm} that $\mathcal{F}_{(\phi,w)}(\phi^\varepsilon,w^\varepsilon,\alpha^\varepsilon)$ is invertible along $\mathscr{C}_{\mathrm{loc}}$. Now all hypotheses of Theorem \ref{globalbifurthm} are satisfied.
\end{proof}

Next we will eliminate or simplify some of the alternatives in Theorem \ref{globalthm} 
{
provided that the vorticity function additionally satisfies $\gamma'(\psi)\leq0$.
}
 We begin with ruling out the loss of compactness alternative (ii) in Theorem \ref{globalthm}.

\begin{lemma}\label{nodalglobal}
{Assume $\gamma'\le0$ on the attained range and
\eqref{nocritical}. Every nontrivial wave on $\mathscr C$ is a
wave of elevation and satisfies \eqref{nodalproperty}.}
\end{lemma}
\begin{proof}
{By Theorem \ref{smallthm}, the local
curve consists of elevation waves satisfying \eqref{nodalproperty}.
First exclude trivial solutions from $\mathscr C$. The set
$\mathscr T:=\{(0,0,\alpha)\in\mathscr C\}$ is relatively
closed. For $0<\alpha<\alpha_{\mathrm{cr}}$, the trivial
linearization is invertible by Lemma \ref{invertiblelemma};
the implicit function theorem makes $\mathscr T$ relatively
open as well. Because $\mathscr C$ is connected and contains
the nontrivial local curve, $\mathscr T$ is empty.

Let $\mathscr N$ be the subset satisfying
\eqref{nodalproperty}. Proposition \ref{closedcondition}
and Proposition \ref{opencondition} make $\mathscr N$ relatively
closed and open among nontrivial solutions on $\mathscr C$.
It contains $\mathscr C_{\mathrm{loc}}$; connectedness thus
gives $\mathscr N=\mathscr C$. In particular,
$\eta_x(x,1)<0$ for $x>0$. Since $\eta(x,1)\to1$ as
$x\to+\infty$ and the solutions are even, $\eta(x,1)>1$
at every finite $x$. Hence all waves on the curve are of
elevation.}
\end{proof}

{By Lemma \ref{nodalglobal}, the monotonicity assumption
\eqref{monotoneass} holds along $\mathscr C$.}
\begin{lemma}\label{psiygamma1}
{Assume $(H_{\mathrm{cf}})$ and the hypotheses of Lemma
\ref{nodalglobal}. Then alternative $(ii)$ in Theorem
\ref{globalthm} cannot occur. Moreover, there is no sequence
$s_n\to\infty$ on $\mathscr C$ for which
\[
\sup_n\|(\phi(s_n),w(s_n))\|_{\mathcal X}<\infty,
\qquad \inf_n\sigma(w(s_n),\alpha(s_n))>0,
\qquad \alpha(s_n)\to\alpha_{\mathrm{cr}}.
\]}
\end{lemma}
\begin{proof}
{
Lemma \ref{nodalglobal} supplies
\eqref{monotoneass}; Lemma \ref{cptlemma}, which uses
$(H_{\mathrm{cf}})$ to exclude monotone bores, then gives a
convergent subsequence. Therefore the alternative $(ii)$ can never occur.

For a sequence in the second assertion, Lemma \ref{nodalglobal}
again supplies \eqref{monotoneass}. The endpoint case of Lemma
\ref{cptlemma} gives a subsequence converging in $\mathcal X_{\mathrm e}$
to a regular wave at $\alpha=\alpha_{\mathrm{cr}}$. It is an
elevation wave, since $w\ge0$ passes to the limit, and its
surface satisfies $\eta_x\le0$. If the limit is nontrivial,
Proposition \ref{closedcondition} and Lemma \ref{lemma5.1} upgrade
this weak inequality to \eqref{nodalproperty}: the case $v\equiv0$
would force $\eta=y$ and then $\psi=\psi_{\mathrm{triv}}$ by the
bounded Dirichlet uniqueness used in Lemma \ref{cptlemma}.
Proposition \ref{Dcriticalnonnegative} therefore contradicts the critical value.
If it is trivial, the $\mathcal X_{\mathrm e}$ norm of
$(\phi(s_n),w(s_n))$ tends to zero. For
$\varepsilon_n=\alpha_{\mathrm{cr}}-\alpha(s_n)>0$ sufficiently
small, the uniqueness statement in Theorem \ref{smallthm}(ii)
places each nontrivial symmetric wave on $\mathscr C_{\mathrm{loc}}$.
This contradicts Theorem \ref{globalthm}(c).
}
\end{proof}

% This is the critical-parameter exclusion needed for the global theorem.
% It uses compactness at alpha_cr and local uniqueness; Proposition 6.7
% alone only rules out a nontrivial critical limit.

\begin{lemma}\label{psiygamma}
    Assume that $(\psi,\eta,\alpha)$ solves \eqref{stripstreameq} with $0\leq\alpha<\alpha_{\mathrm{cr}}$. Then we have
\begin{enumerate}
    \item [$(i)$] if $\gamma\geq 0$, then  $\psi_y\leq1$ on $\Gamma$,
\item [$(ii)$] {if $\gamma\le0$, then $\psi_y\ge1$
on $\Gamma$. In particular, the strict inequality holds whenever $\gamma<0$.}
\end{enumerate}
\end{lemma}
\begin{proof}
If $\gamma\ge0$, we consider the auxiliary function $\rho(x,y)=\psi(x,y)-y$. Recalling that $\psi = \phi + \psi_{triv}$, we can see that $\rho$ solves the following problem
\begin{equation*}
	\begin{cases}
		\Delta\rho=-\gamma(\psi)|\nabla\eta|^{2}\le0 & \text{in}~\mathcal{R},\\
		\rho=0 & \text{on}~\Gamma\cup B.
	\end{cases}
\end{equation*}
By the strong maximum principle in Theorem B.1 and the Hopf boundary lemma in Theorem B.2, we can deduce that $\rho\ge0$ in $\mathcal{R}$ and its outward normal derivative satisfies $\partial_{y}\rho\le0$ on $\Gamma$. Consequently, we obtain
\begin{equation*}
	\psi_{y}-1 = \phi_{y}+\psi_{\mathrm{triv},y}-1\le0 \quad \text{on}~\Gamma.
\end{equation*}

{If $\gamma\leq0$, we employ the same auxiliary function $\rho(x,y)=\psi(x,y)-y$. In this case, $\rho$ satisfies}
\begin{equation*}
	\begin{cases}
		\Delta\rho=-\gamma(\psi)|\nabla\eta|^{2}\geq0 & \text{in}~\mathcal{R},\\
		\rho=0 & \text{on}~\Gamma\cup B.
	\end{cases}
\end{equation*}
By the bounded maximum principle on the strip, $\rho\leq0$ in $\mathcal R$. Since $\rho=0$ on $\Gamma$, its outward normal derivative satisfies $\partial_y\rho\geq0$ there, giving $\psi_y\geq1$ for every $\gamma\leq0$. If $\gamma<0$, then $\Delta\rho>0$, and Theorem B.2 gives the strict inequality
\begin{equation*}
	\psi_{y}-1 > 0 \quad \text{on}~\Gamma.
	\end{equation*}
{When $\gamma=0$, the bounded harmonic function $\psi-y$
vanishes on both $\Gamma$ and $B$, and therefore vanishes throughout
the strip. Thus $\psi_y=1$ on $\Gamma$, completing (ii).}
\end{proof}

Based on Lemma \ref{psiygamma1} and Lemma \ref{psiygamma}, we can establish the following simplified version of Theorem \ref{globalthm}.
\begin{theorem}\label{finalthmglobal}
{Assume $(H_{\mathrm{cf}})$. If the vorticity function
$\gamma$ is real-analytic, satisfies $\gamma'\le0$ on the
attained range, and obeys \eqref{nocritical}, then the local curve
$\mathscr{C}_{\mathrm{loc}}$ is contained in a continuous curve of
solutions parameterized as}
    \[\mathscr{C}=\{(\phi(s),w(s),\alpha(s)):0<s<\infty\}\subset\mathcal{U},\]
with the following properties:
\begin{enumerate}
    \item [$(a)$] As $s\to\infty$
    \begin{enumerate}
        \item [$(i)$] for $\gamma\geq0$,
        \begin{equation*}
            {\frac{1}{\inf_{\Gamma}(\mu-2\alpha(s)w(s))}
            +\frac{1}{\inf_{\Gamma}((\partial_x\zeta_{[w(s)]})^2
                               +(1+\partial_y\zeta_{[w(s)]})^2)}
            +\frac{1}{\alpha(s)}\to\infty,}
        \end{equation*}
        \item [$(ii)$] {for $\gamma\le0$,}
        \begin{equation*}
            \sup_{\Gamma}|\nabla\eta_{[w(s)]}|
            +\frac{1}{\inf_{\Gamma}\big((\partial_x\zeta_{[w(s)]})^2
+(1+\partial_y\zeta_{[w(s)]})^2\big)}
          +\frac{1}{\alpha(s)}\to\infty.
        \end{equation*}
    \end{enumerate}
     \item [$(b)$] Near each point  $(\phi(s),w(s),\alpha(s))\in\mathscr{C}$, we can re-parameterize $\mathscr{C}$ so that the mapping $s\to(\phi(s),w(s),\alpha(s))$ is real analytic.
     \item [$(c)$] $(\phi(s),w(s),\alpha(s))\not\in\mathscr{C}_{\mathrm{loc}}$ for $s$ sufficiently large.
 \item [$(d)$]Every solution on $\mathscr C$ is a wave of
 elevation and satisfies the nodal property \eqref{nodalproperty}.
\end{enumerate}

\end{theorem}
\begin{proof}
The alternative $(i)$ in Theorem \ref{globalthm}
 is equivalent to
 \begin{equation}\label{alter1eq}
 \|(\phi(s),w(s))\|_{\mathcal{X }}+\frac{1}{\sigma(w(s),\alpha(s))}+\frac{1}{\alpha(s)}+\frac{1}{\alpha_{\mathrm{cr}}-\alpha(s)}\to\infty.
 \end{equation}
 By Theorem \ref{strong} and maximum modulus principle to the factors of $\sigma(w,\alpha)$, we see that the term $1/\sigma(w,\alpha)$ can be controlled by
 \begin{equation*}
     \frac{1}{\inf_{\Gamma}(\mu-2\alpha(s)w(s))}+\frac{1}{\inf_{\Gamma}((\partial_x\zeta_{[w(s)]})^2+(1+\partial_y\zeta_{[w(s)]})^2)}.
 \end{equation*}
By the uniform regularity result in Appendix and Theorem
\ref{strong}, the first term $\|(\phi(s),w(s))\|_{\mathcal X}$
is bounded whenever the following three surface quantities are bounded:
\begin{equation*}
   \frac{1}{\inf_{\Gamma}(\mu-2\alpha(s)w(s))}+\sup_{\Gamma}|\nabla\zeta_{[w(s)]}|+\frac{1}{\inf_{\Gamma}((\partial_x\zeta_{[w(s)]})^2+(1+\partial_y\zeta_{[w(s)]})^2)}.
\end{equation*}
{We rule out the alternative $\alpha\to\alpha_{\mathrm{cr}}$ in
\eqref{alter1eq}. If the other terms stayed bounded along some
$s_n\to\infty$, the bounds just described would give a uniform
$\mathcal X$ norm and a positive lower bound for $\sigma$.
Theorem \ref{globalthm}(a), with alternative $(ii)$ excluded by
Lemma \ref{psiygamma1}, would then force
$\alpha(s_n)\to\alpha_{\mathrm{cr}}$ along a subsequence.
The second assertion of that lemma excludes this sequence. Thus,
as $s\to\infty$,}
 
{
\begin{equation}\label{alter1eq2}
\frac{1}{\inf_\Gamma(\mu-2\alpha(s)w(s))}
 +\sup_\Gamma|\nabla\eta_{[w(s)]}|
+\frac{1}{\inf_\Gamma\big((\partial_x\zeta_{[w(s)]})^2
                  +(1+\partial_y\zeta_{[w(s)]})^2\big)}
 +\frac1{\alpha(s)}\to\infty.
\end{equation}
}
 Now we shall apply Lemma \ref{psiygamma} to continue simplifying \eqref{alter1eq2}.

For $\gamma\geq0$, Lemma \ref{psiygamma} gives $\psi_y\leq1$
on $\Gamma$. The dynamic condition prevents $\psi_y$ from vanishing
there, and its far-field limit is $\psi_{\mathrm{triv},y}(1)=\sqrt\mu>0$.
Consequently,
\[
 (\mu-2\alpha w)|\nabla\eta|^2
 =\big(\phi_y+\psi_{\mathrm{triv},y}(1)\big)^2\leq1
 \qquad\mathrm{on}\ \Gamma.
\]
Thus the second term of \eqref{alter1eq2} is controlled by the first.
As for the case $\gamma\leq0$, by Lemma \ref{psiygamma}, we obtain the lower bound
\begin{equation*}
	\mu-2\alpha w\geq \frac{1}{(\partial_x\zeta_{[w(s)]})^2+(1+\partial_y\zeta_{[w(s)]})^2)}\quad\mathrm{on~}\Gamma.
\end{equation*}
Then the first term of  \eqref{alter1eq2} can be controlled by second term, which proves alternative $(ii)$.
\end{proof}

Now we are in a position to state and prove the main result of this paper.
\begin{theorem}
{Fix $g>0$ and the asymptotic depth $d>0$. Assume that
$\gamma$ is real-analytic, satisfies 
$(H_{\mathrm{cf}})$,  $\gamma\le0$ and $\gamma'\le0$ on the attained range. Then there is a global continuous curve
$\mathscr C$ of solutions to \eqref{euler}, parameterized by
$s\in(0,\infty)$. Its waves are symmetric, monotone waves of
elevation, and as $s\to\infty$ we have}
\begin{equation}\label{mainglobal}
{\min\bigg\{\frac{1}{\sup_{\Gamma}|\nabla\eta(s)|},
\inf_{\Gamma}|\nabla\eta(s)|^2,
\frac{\mu}{F(s)}\bigg\}\to0.}
\end{equation}
{Moreover, $\eta$ is even in $x$
and $\eta_x(x,1)<0$ for $x>0$.}

\end{theorem}
\begin{proof}
We first show that the solutions $(\phi,w,\alpha)\in\mathscr C$
in Theorem \ref{finalthmglobal} correspond to solutions of
\eqref{streamori}, and hence of \eqref{euler}. Since $\eta$ is
harmonic, it has a harmonic conjugate $\xi$ satisfying
$\xi_x=\eta_y$ and $\xi_y=-\eta_x$. Normalize it by
$\xi(0,0)=0$. Evenness of $\eta$ then implies that $\xi$ is odd
in $x$ and $\xi(0,y)=0$. Moreover, $\xi(x,y)\to\pm\infty$
uniformly in $y$ as $x\to\pm\infty$, because $\xi_x\to1$
uniformly. The wave is of elevation, so the maximum principle and
Hopf lemma give $\eta>0$ in $\mathcal R$ and
$\xi_x=\eta_y>0$ on $B$. Thus $\xi(x,0)>0$ for $x>0$.
If $\xi$ had a nonpositive value at a point of
$\mathcal R^+\cup\Gamma^+$, its nonpositive minimum away from
$\{x=0\}$ would be attained on $\Gamma^+$, by the maximum
principle and the preceding boundary and far-field properties.
The Hopf lemma would give $\xi_y<0$ there, whereas
$\xi_y=-\eta_x>0$ by \eqref{nodalproperty}. Hence $\xi>0$
on $\mathcal R^+\cup\Gamma^+$.

The strict monotonicity of $\eta(x,1)$ on each half-line,
together with the oddness and the sign of $\xi$, now shows that
$H=\xi+i\eta$ is injective on $\Gamma$. It is also injective
on $B$, and these two images are disjoint. To justify injectivity
in the whole strip, choose $R$ sufficiently large that
$\eta_y(\pm R,y)>0$ for $0\le y\le1$. Each vertical side of
$(-R,R)\times(0,1)$ then has a simple image with height strictly
between $0$ and $\eta(R,1)$. The top image has height at least
$\eta(R,1)$, and the two side images lie in opposite half-planes
by the sign of $\xi$. Thus $H$ is injective on the boundary of
this rectangle. The Darboux--Picard theorem
\cite[Corollary 9.16]{burckel} gives injectivity inside it.
Since $R$ can be arbitrarily large, $H$ is injective on
$\overline{\mathcal R}$; moreover, $|H'|=|\nabla\eta|>0$ by
\eqref{stripk}. This establishes the required conformal map.
 Finally, from \eqref{alphasca} combined with a reversal of the non-dimensional scaling, we find that
  \begin{equation*}
    \begin{aligned}
    &\mu-2\alpha w=\mu-\frac{2\mu}{F^2}(\eta-1)=\mu-\frac{2\mu}{F^2}\frac{\eta^*-d}{d},\\
    &(\partial_x\zeta_{[w(s)]})^2+(1+\partial_y\zeta_{[w(s)]})^2=|\nabla\eta(x,y)|^2={|\nabla\eta^*(x^*,y^*)|^2},
    \end{aligned}
  \end{equation*}
  where ${\eta^*=d\eta}$ and $(x^*,y^*)=d(x,y)$.
The negative-vorticity alternative in Theorem
 \ref{finalthmglobal}(a)(ii), together with
 $\alpha=\mu/F^2$,
 gives \eqref{mainglobal}.
\end{proof}

At the end of this paper, we shall give a concrete example for the main theorem.
{
\begin{example}\label{cubicvorticity}
Let $\gamma(s)=-s-s^3$ for $s\in\mathbb R$. This vorticity function is
real-analytic and satisfies $\gamma'(s)=-1-3s^2<0$ on $\mathbb R$.
For every $d>0$, the laminar Dirichlet problem
\eqref{hatpsi1}--\eqref{hatpsi3} becomes
\[
\hat\psi_{yy}=d^2(\hat\psi+\hat\psi^3),\qquad
\hat\psi(0)=0,\qquad \hat\psi(1)=1.
\]
The constants $0$ and $1$ are lower and upper solutions, so a solution
exists. The strict monotonicity of $-\gamma(s)=s+s^3$ and the maximum principle give
uniqueness and $0<\hat\psi<1$ in $(0,1)$. Applying Theorem \ref{hopf} at
$y=0$ gives $\hat\psi_y(0)>0$. Since $\hat\psi_{yy}>0$ in $(0,1)$,
\[
\hat\psi_y(y)\geq\hat\psi_y(0)>0\qquad(0\leq y\leq1).
\]
Thus every laminar solution is unidirectional, and in particular
$(H_{\mathrm{cf}})$ holds.

For a solitary-wave solution, the stream-function equation is
\[
\Delta\psi-(1+\psi^2)|\nabla\eta|^2\psi=0
\qquad\mathrm{in}\ \mathcal R.
\]
The boundary values $\psi|_B=0$, $\psi|_\Gamma=1$, and the shear flow
limit $\psi\to\psi_{\mathrm{triv}}$ will allow the maximum principle to give
$0\leq\psi\leq1$ throughout $\overline{\mathcal R}$. Consequently,
\[
\gamma(\psi)=-\psi-\psi^3\leq0,\qquad
\gamma'(\psi)=-1-3\psi^2<0.
\]
Thus the  vorticity function satisfies the assumptions of the
main theorem.
\end{example}
}

\appendix

\section{Uniform Regularity}\label{uniformregu}
In this Appendix, we prove that the $C^{3+\beta}$ norm of $\psi$ and $\eta$ which satisfy \eqref{stripstreameq}
can be controlled by the constant $\delta>0$ in the inequalities
\begin{equation}\label{C1}
    \delta\leq|\nabla\eta|\leq 1/\delta\quad\mathrm{and}\quad \mu-2\alpha(\eta-1)\geq\delta\quad\mathrm{in~}\mathcal{R}.
\end{equation}
Note that the left side of \eqref{stripk}, and the infimum of $\sigma$ can be controlled only in terms of $\delta$.
\begin{theorem}\label{regularthm}
{Suppose that $(\psi,\eta,\alpha)$ solves
\eqref{stripstreameq}, $0\leq\alpha\leq\alpha_{\mathrm{cr}}$,
and \eqref{C1} holds for some $\delta>0$. Assume in addition that
$\gamma'\leq0$ and that either $\gamma\geq0$ or $\gamma\leq0$
on the attained range of $\psi$. Then there exists a constant
$C=C(\delta)>0$ such that}
\[
{\|\eta\|_{C_b^{3+\beta}(\overline{\mathcal R})}
      +\|\psi\|_{C_b^{3+\beta}(\overline{\mathcal R})}\leq C.}
\]
{Here and below, constants may also depend on the fixed data
$\mu,\alpha_{\mathrm{cr}},\gamma$ and $\beta$.}
\end{theorem}
\begin{proof}
{
We give the bootstrap directly. First, \eqref{C1} and
$\eta=0$ on $B$ give a uniform $C_b^1$ bound for $\eta$.
The sign and monotonicity assumptions on $\gamma$ also give a
uniform bound for $\psi$. Indeed, if $\gamma\geq0$, then the
maximum principle gives $\psi\geq0$, while
$0\leq\gamma(\psi)\leq\gamma(0)$. Comparing with a suitable
multiple of $y(1-y)$ bounds $\psi$ from above. If $\gamma\leq0$,
then $\psi\leq1$ and
$\gamma(1)\leq\gamma(\psi)\leq0$; the same comparison bounds
$\psi$ from below. Consequently,
\[
\|\psi\|_{C_b^0(\overline{\mathcal R})}
+\|\gamma(\psi)|\nabla\eta|^2\|_{C_b^0(\overline{\mathcal R})}
\leq C.
\]

Applying the boundary and interior estimates in
\cite[Theorem 8.33]{Ts} on translated unit half-balls and balls
gives, after possibly decreasing the exponent,
\begin{equation}\label{initialregularity}
\|\psi\|_{C_b^{1+\kappa}(\overline{\mathcal R})}\leq C
\end{equation}
for some $\kappa\in(0,\beta]$. On $\Gamma$, the dynamic condition
can be written as
\[
|D\eta|^2-
\frac{\psi_y^2(x,1)}{\mu-2\alpha(\eta-1)}=0.
\]
The derivative of its left-hand side with respect to $D\eta$ is
$2D\eta$, whose length is bounded below by $2\delta$. Thus the
two-dimensional, possibly non-oblique boundary estimate
\cite[Theorem 1]{libtams}, together with
\cite[Theorem 8.29]{Ts} at the bottom, gives
\[
\|\eta\|_{C_b^{1+\kappa}(\overline{\mathcal R})}\leq C.
\]
As in \cite[Lemma 2]{libtams}, the boundary operator may first be
extended away from the range specified by \eqref{C1}; this does
not change the uniform structural constants.

It follows that the right-hand side of \eqref{stripstreameq1} is
bounded in $C_b^\kappa$. The standard interior and boundary
Schauder estimates \cite[Theorems 6.2 and 6.6]{Ts} therefore yield
\[
\|\psi\|_{C_b^{2+\kappa}(\overline{\mathcal R})}\leq C.
\]
The nonlinear boundary operator for $\eta$ is now uniformly
$C^{1+\kappa}$. Hence \cite[Theorem 3]{libtams}, again combined
with the bottom Dirichlet estimate, gives
\begin{equation}\label{secondregularity}
\|\eta\|_{C_b^{2+\kappa}(\overline{\mathcal R})}\leq C.
\end{equation}
All these estimates are uniform on the infinite strip because the
local neighborhoods have a fixed size and are related by horizontal
translations.

It remains to reach the exponent and order required in the theorem.
Differentiate \eqref{stripstreameq1}, \eqref{harmoniceta}, and the
dynamic boundary condition with respect to $x$. We obtain
\begin{equation}\label{regularitysystem}
\begin{cases}
\Delta\psi_x+\gamma'(\psi)|\nabla\eta|^2\psi_x
+2\gamma(\psi)\nabla\eta\cdot\nabla\eta_x=0,
                                      &\mathrm{in}\ \mathcal R,\\
\Delta\eta_x=0,                     &\mathrm{in}\ \mathcal R,\\
\psi_x=0,                           &\mathrm{on}\ \Gamma\cup B,\\
\eta_x=0,                           &\mathrm{on}\ B,\\
\psi_y(\psi_x)_y
-(\mu-2\alpha(\eta-1))\nabla\eta\cdot\nabla\eta_x
+\alpha|\nabla\eta|^2\eta_x=0,
                                      &\mathrm{on}\ \Gamma.
\end{cases}
\end{equation}
For this system the complementing quantity in
\eqref{lopatinskii} is
\[
(\mu-2\alpha(\eta-1))^2|\nabla\eta|^2\geq\delta^4.
\]
The coefficient bounds supplied by \eqref{initialregularity} and
\eqref{secondregularity} allow us to apply Lemma
\ref{schauderlemma2}, first with exponent $\kappa$. We obtain
\[
\|\psi_x\|_{C_b^{2+\kappa}}
+\|\eta_x\|_{C_b^{2+\kappa}}\leq C.
\]
The interior equations then recover the remaining normal
derivatives and give a uniform $C_b^{3+\kappa}$ bound for
$(\psi,\eta)$. At this stage all coefficients in
\eqref{regularitysystem} have the regularity required by Lemma
\ref{schauderlemma2} with exponent $\beta$. Applying that estimate
once more and again using the interior equations to recover the
normal derivatives proves
\[
\|\psi\|_{C_b^{3+\beta}(\overline{\mathcal R})}
+\|\eta\|_{C_b^{3+\beta}(\overline{\mathcal R})}\leq C.
\]
}
\end{proof}
\section{Maximum Principles}
We refer to \cite{fraenkel, gidas, serrin} for the following
maximum principles. Let $\Omega\subset\mathbb R^2$ be a
connected open set, possibly unbounded, and let
$L=\Delta+c(x)$, where $c\in C^0(\overline\Omega)$ is bounded.
Let $u\in C^2(\Omega)\cap C^0(\overline\Omega)$ satisfy
$Lu\geq0$ in $\Omega$.

\begin{theorem}\label{strong}{\rm(}Strong maximum principle{\rm)}
Suppose that $u$ attains its maximum over $\overline\Omega$
at an interior point. If this maximum is nonnegative and
$c\le0$, or if this maximum is zero, then $u$ is constant.
\end{theorem}

\begin{theorem}\label{hopf}{\rm(}Hopf boundary lemma{\rm)}
Suppose that $u$ attains its maximum over $\overline\Omega$
at $x^*\in\partial\Omega$, that $u$ is $C^1$ up to the
boundary near $x^*$, and that there is an open ball
$B\subset\Omega$ with
$\overline B\cap\partial\Omega=\{x^*\}$.
Assume that this maximum is nonnegative and $c\le0$, or that
this maximum is zero. Then $u$ is constant or
\[
\nu\cdot\nabla u(x^*)>0,
\]
where $\nu$ is the outward unit normal of the interior tangent
ball at $x^*$, and hence the outward normal of $\Omega$ when
the latter is defined.
\end{theorem}

\begin{theorem}\label{serrin}{\rm(}Serrin corner lemma{\rm)}
Suppose that, after a rigid motion taking $x^*$ to the origin,
$\Omega$ agrees near $x^*$ with the quadrant
$\{x_1>0,\ x_2>0\}$. Assume that $u$ is $C^2$ up to these
two boundary segments, $Lu=0$ in $\Omega$, $u>0$ in
$\Omega$, and $u(x^*)=0$. Then, for every unit vector $s$
pointing strictly into this quadrant,
\[
\partial_su(x^*)>0
\qquad\text{or}\qquad
\partial_s^2u(x^*)>0.
\]
\end{theorem}
\section{Schauder Estimates for Elliptic Problems in Infinite Strips}
Consider an infinite strip domain in $\R^2$
\begin{equation*}
	\mathcal{R}:=\{(x,y)\in\R^2:0<y<1\}
\end{equation*}
with the top boundary
\begin{equation*}
	\Gamma:=\{(x,y)\in\R^2:y=1\},
\end{equation*}
and the bottom $y=0$.

Consider the linear equation and use the Einstein summation convention
\begin{equation}\label{linearlemma}
	\begin{aligned}
		&Au:=a^{ij}\partial_{ij}u+b^i \partial_iu+cu=f\quad&&\mathrm{in~}\mathcal{R},\\
		&Bu:=\sigma^i \partial_i u+\mu u=g &&\mathrm{on~}\Gamma,\\
		&u=0&&\mathrm{on~}y=0.
	\end{aligned}
\end{equation}
Here we assume that $A$ is uniformly elliptic and  $B$ is uniformly oblique such that $a^{ij}=a^{ji}$, $a^{ij}\xi_i\xi_j\geq c|\xi|^2$, and ${|\sigma\cdot e_y|\geq c}$, for some $c>0$. Moreover, we assume  that the coefficients have
the regularity $a^{ij},b^i,c\in C_b^{k-2+\beta}$ and $\sigma^i,\mu\in C_b^{k-1+\beta}(\Gamma)$ for some integer $k\geq2$ and $\beta\in(0,1)$.

\begin{lemma}\cite{agmon}\label{schauderlemma}
	If $u\in C^{k+\beta}_b(\bar{\mathcal{R}})$ solves \eqref{linearlemma}, then
	\begin{equation*}
\|u\|_{C^{k+\beta}(\bar{\mathcal{R}})}\leq C(\|f\|_{C^{k-2+\beta}(\bar{\mathcal{R}})}+\|g\|_{C^{k-1+\beta}(\Gamma)}+\|u\|_{C^0(\bar{\mathcal{R})}}),
	\end{equation*}
	where the constant $C$ depends merely on the obliqueness constants and the norms of the coefficients.

\end{lemma}

{
We also need the following inhomogeneous system, with both unknowns vanishing
on $B$:
\begin{equation}\label{ellipticui}
\begin{cases}
\Delta u_1+d_{ij}\partial_j u_i+e_i u_i=F_1&\mathrm{in}\ \mathcal R,\\
\Delta u_2=F_2&\mathrm{in}\ \mathcal R,\\
a_{ij}\partial_i u_j+b_i u_i=f&\mathrm{on}\ \Gamma,\\
c_i u_i=g&\mathrm{on}\ \Gamma.
\end{cases}
\end{equation}
Here $\partial_1=\partial_x$ and $\partial_2=\partial_y$; in $a_{ij}$ the
first index denotes the derivative and the second the unknown.
\begin{lemma}\label{schauderlemma2}
Let $k\geq2$ and $0<\beta<1$. Assume that
$a_{ij},b_i\in C_b^{k-1+\beta}(\Gamma)$,
$c_i\in C_b^{k+\beta}(\Gamma)$ and
$d_{ij},e_i\in C_b^{k-2+\beta}(\overline{\mathcal R})$.
Suppose
\begin{equation}\label{lopatinskii}
(c_1a_{12}-c_2a_{11})^2+(c_1a_{22}-c_2a_{21})^2\geq\lambda>0.
\end{equation}
Then every bounded classical solution of \eqref{ellipticui} satisfies
\begin{equation*}
\begin{split}
\|u_1\|_{C_b^{k+\beta}}+\|u_2\|_{C_b^{k+\beta}}
\leq C\big(&\|F_1\|_{C_b^{k-2+\beta}}+\|F_2\|_{C_b^{k-2+\beta}}\\
&+\|f\|_{C_b^{k-1+\beta}(\Gamma)}+\|g\|_{C_b^{k+\beta}(\Gamma)}
 +\|u_1\|_{C_b^0}+\|u_2\|_{C_b^0}\big).
\end{split}
\end{equation*}
The constant depends on the displayed coefficient bounds, $k,\beta$ and
$\lambda$.
\end{lemma}
\begin{proof}
For a nonzero tangential frequency $\xi$, the decaying principal solutions at
the top have amplitudes $(v_1,v_2)$. The zero-order boundary condition leaves
the direction $(v_1,v_2)=t(-c_2,c_1)$. The other boundary symbol in this
direction is
\[
t\big[i\xi(c_1a_{12}-c_2a_{11})
        +|\xi|(c_1a_{22}-c_2a_{21})\big].
\]
Its squared modulus is at least $\lambda\xi^2$. This proves the uniform
complementing condition. The Agmon--Douglis--Nirenberg estimate
\cite[Part II]{agmon}, applied on overlapping strip windows, gives the asserted estimate.
The lower-order interior terms are absorbed by interpolation, retaining both
inhomogeneous interior data and both boundary data.
\end{proof}
}

\bibliographystyle{ams}

\begin{thebibliography}{99}


\bibitem{av} A. Aasen, K. Varholm, Traveling gravity water waves with critical layers,
\emph{J. Math. Fluid Mech.} \textbf{20} (2018), 161-187.

\bibitem{agmon} S. Agmon, A. Douglis, L. Nirenberg, Estimates near the boundary for solutions of elliptic partial
differential equations satisfying general boundary conditions. {I}, \emph{Comm. Pure Appl. Math.} \textbf{12} (1959), 623-727; {II}, \emph{Comm. Pure Appl. Math.} \textbf{17} (1964), 35-92.

\bibitem{alinhac} S. Alinhac, Existence d'ondes de rar\'efaction pour des syst\`emes
quasi-lin\'eaires hyperboliques multidimensionnels, \emph{Comm. Partial Differential Equations}, \textbf{14} (1989), 173-230.

\bibitem{amickacta}  C. J.  Amick,  L. E.  Fraenkel, J. F. Toland, On the {S}tokes conjecture for
the wave of extreme form, \emph{Acta Math.} \textbf{148} (1982), 193-214.

\bibitem{amickarma}C. J.  Amick, J. F. Toland,  On solitary water-waves of finite amplitude,
\emph{Arch. Rational Mech. Anal.} \textbf{76} (1981), 9-95.



\bibitem{bbmm} T. Barbieri, M. Berti, A. Maspero, M. Mazzucchelli, Bifurcation of gravity-capillary Stokes waves
with constant vorticity, \emph{J. Differential Equations}, \textbf{451} (2026), Paper No. 113753, 36 pp.

\bibitem{basu} B. Basu, A flow force reformulation for steady irrotational water waves,
\emph{J. Differential Equations}, \textbf{268} (2020), 7417-7452.

\bibitem{bk} B. Basu, F. Kogelbauer, Global bifurcation of irrotational water waves using a flow force formulation,
\emph{J. Differential Equations}, \textbf{301} (2021), 73-96.



\bibitem{burckel} R. B. Burckel, An introduction to classic complex analysis, Vol. 1, Pure and
Applied Mathematics, vol. 82. Academic Press, Inc, New York, 1979.

\bibitem{cvw} R. M. Chen, K. Varholm, S. Walsh, M. H. Wheeler, Vortex-carrying solitary gravity waves of large amplitude,
\emph{Comm. Math. Phys.} \textbf{406} (2025), Paper No. 149, 46 pp.

\bibitem{chenpoincare} R. M. Chen, S. Walsh, M. H. Wheeler, Existence and qualitative theory for stratified solitary water
waves, \emph{Ann. Inst. H. Poincar\'e{} C Anal. Non Lin\'eaire}, \textbf{35} (2018), 517-576.

\bibitem{chennonlinearity} R. M. Chen, S. Walsh, M. H. Wheeler, Center manifolds without a phase space for quasilinear
problems in elasticity, biology, and hydrodynamics, \emph{Nonlinearity}, \textbf{35} (2022), 1927–1985.

\bibitem{cwwz} J. Chu, X. Wang, L. Wang, Z. Zhang, A flow force reformulation of steady periodic fixed-depth
irrotational equatorial flows, \emph{J. Differential Equations}, \textbf{292} (2021), 220-246.

 \bibitem{cr} M. G. Crandall, P. H. Rabinowitz, Bifurcation from simple eigenvalues,
 \emph{J. Funct. Anal.} \textbf{8} (1971), 321-340.

 \bibitem{c-book} A. Constantin, Nonlinear Water Waves with Applications to Wave-Current
 Interactions and Tsunamis, SIAM, Philadelphia, 2011.

 \bibitem{constantincpam04} A. Constantin, W. Strauss, Exact steady periodic water waves
 with vorticity, \emph{Comm. Pure Appl. Math.} \textbf{57} (2004), 481-527.

 \bibitem{constantinstrauss07}A. Constantin, W. Strauss, Rotational steady water waves near stagnation, \emph{Philos. Trans. R. Soc. Lond. Ser. A Math. Phys. Eng. Sci.} \textbf{365} (2007), 2227-2239.

 \bibitem{constantinacta} A. Constantin, W. Strauss, E. V\u arv\u aruc\u a, Global bifurcation of steady
 gravity water waves with critical layers, \emph{Acta. Math.} \textbf{217} (2016), 195-262.

 \bibitem{csv} A. Constantin, W. Strauss, E. V\u arv\u aruc\u a, Large-amplitude steady
 downstream water waves, \emph{Comm. Math. Phys.} \textbf{387} (2021), 237-266.

\bibitem{constantinarma11} A. Constantin, E. V\u arv\u aruc\u a, Steady periodic water waves with
constant vorticity: regularity and local bifurcation, \emph{Arch. Rational. Mech. Anal.} \textbf{199} (2011), 33-67.



\bibitem{dz} G. Dai, Y. Zhang, Global bifurcation structure and some properties of steady periodic
water waves with vorticity, \emph{J. Differential Equations}, \textbf{349} (2023), 125-137.

\bibitem{dpmw} J. D\'{a}vila, M. del Pino, M. Musso, M. H. Wheeler, Overhanging solitary vater waves, \emph{Invent. Math.} (2026), https://doi.org/10.1007/s00222-026-01404-w.

\bibitem{dj} M. -L. Duberil-Jacotin, Sur la d\'etermination rigoureuse des ondes permanents p\'eriodiques d'ampleur dinie,
\emph{J. Math. Pures Appl.} \textbf{13} (1934), 217-291.

\bibitem{matssiam} M. Ehrnstr\"{o}m, J. Escher, E. Wahl\'{en}, Steady water waves with
multiple critical layers, \emph{SIAM J. Math. Anal.} \textbf{43} (2011), 1436-1456.

\bibitem{fraenkel}  L. E. Fraenkel,  An introduction to maximum principles and symmetry in elliptic problems,
Cambridge Tracts in Mathematics, \textbf{128}, Cambridge University Press, Cambridge, 2000.

\bibitem{friedrichs} K. O. Friedrichs, D. H. Hyers, The existence of solitary waves,
\emph{Comm. Pure Appl. Math.} \textbf{7} (1954), 517-550.

\bibitem{ger} F. Gerstner, Theorie der Wellen samt einer daraus abgeleiteten
Theorie der Deichprofile, \emph{Ann. Phys.} 2 (1809), 412-445.


\bibitem{gidas} B. Gidas, W. Ni, L. Nirenberg, Symmetry and related properties via the maximum principle,  \emph{Commun. Math. Phys.}\textbf{68} (1979), 209-243.


 \bibitem{Ts} D. Gilbarg, T.S. Trudinger, Elliptic Partial Differential Equations of Second Order,
 Springer-Verlag, New York, 1998.

\bibitem{gw} M. D. Groves, E. Wahl\'{e}n, Small-amplitude Stokes and solitary gravity water waves
with an arbitrary distribution of vorticity, \emph{Phys. D}, \textbf{237} (2008), 1530-1538.

\bibitem{hcmp} S. V. Haziot, Stratified large-amplitude steady periodic water waves with critical layers,
\emph{Comm. Math. Phys.} \textbf{381} (2021), 765-797.

\bibitem{susannaarma} S. V. Haziot, M. H. Wheeler, Large-amplitude steady solitary water waves with
constant vorticity, \emph{Arch. Ration. Mech. Anal.} \textbf{247} (2023),  Paper No. 27, 49 pp.



\bibitem{hcpde} D. Henry, Large amplitude steady periodic waves for fixed-depth rotational flows,
\emph{Comm. Partial Differential Equations}, \textbf{38} (2013), 1015-1037.

\bibitem{hur08} V. M.  Hur, Exact solitary water waves with vorticity,
\emph{Arch. Ration. Mech. Anal.} \textbf{188} (2008), 213-244.

    \bibitem{hurmrl} V. M.  Hur, Symmetry of solitary water waves with vorticity,
    \emph{Math. Res. Lett.} \textbf{15} (2008), 491-509.

\bibitem{hw} V. M. Hur, M. H. Wheeler, Overhanging and touching waves in constant vorticity flows,
\emph{J. Differential Equations}, \textbf{338} (2022), 572-590.

\bibitem{kirch}  K. Kirchg\"assner, Nonlinearly resonant surface waves and homoclinic bifurcation,
Advances in applied mechanics, {V}ol.\ 26, Academic Press, Boston, 1988.

\bibitem{kkl} V. Kozlov, N. Kuznetsov, E. Lokharu, Solitary waves on constant vorticity flows with
an interior stagnation point, \emph{J. Fluid Mech.} \textbf{904} (2020), A4, 18 pp.

\bibitem{kl} V. Kozlov, E. Lokharu, Global bifurcation and highest waves on water of finite depth,
\emph{Arch. Ration. Mech. Anal.} \textbf{247} (2023), Paper No. 98, 30 pp.

\bibitem{klw2021}{V. Kozlov, E. Lokharu, M. H. Wheeler,
Nonexistence of subcritical solitary waves,
\emph{Arch. Ration. Mech. Anal.} \textbf{241} (2021), 535--552.}

\bibitem{lannesjams}  D. Lannes, Well-posedness of the water-waves equations,
\emph{J. Amer. Math. Soc.} \textbf{18} (2005), 605-654.

\bibitem{libtams} G. M. Lieberman, Two-dimensional nonlinear boundary value problem for
elliptic equations, \emph{Trans. Am. Math. Soc} \textbf{300} (1987), 287-295.



\bibitem{Mielke} A. Mielke, Reduction of quasilinear elliptic equations in cylindrical
domains with applications, \emph{Math. Methods Appl. Sci.} \textbf{10} (1988), 51-66.

\bibitem{serrin} J. Serrin, A symmetry problem in potential theory, \emph{Arch. Ration. Mech. Anal.} \textbf{43} (1971), 304–318.

\bibitem{dan} D. Sinambela, Large-amplitude solitary waves in two-layer density stratified            water, \emph{SIAM J. Math. Anal.} \textbf{53} (2021), 4812-4864.

\bibitem{sto} G. Stokes, On the theory of oscillatory waves,
\emph{Trans. Camb. Philos. Soc.} \textbf{8} (1847), 441-473.

\bibitem{ter1961} A. M. Ter-Krikorov, A solitary wave on the surface of a turbulent liquid,
\emph{\v Z. Vy\v cisl. Mat i Mat. Fiz.} \textbf{1} (1961), 1077-1088.

\bibitem{Var} K. Varholm, Global bifurcation of waves with multiple critical layers,
\emph{SIAM J. Math. Anal.} \textbf{52} (2020), 5066-5089.

\bibitem{vol} V. Volpert, Elliptic partial differential equations. {V}olume 1:
{F}redholm theory of elliptic problems in unbounded domains. Monographs in Mathematics, vol. 101, Birkh\"auser/Springer Basel AG, Basel, 2011.

\bibitem{wahlenjde09}E. Wahl\'en, Steady water waves with a critical layer,
\emph{J. Differential Equations}, \textbf{246} (2009), 2468-2483.



\bibitem{ww} E. Wahl\'en, J. Weber,  Global bifurcation of capillary-gravity water waves with
overhanging profiles and arbitrary vorticity. \emph{Int. Math. Res. Not. IMRN} \textbf{20} (2023), 17377-17410.

\bibitem{wahlenduke}E. Wahl\'en, J. Weber, Large-amplitude steady gravity water waves with
general vorticity and critical layers, \emph{Duke Math. J.} \textbf{173} (2024), 2197-2258.

\bibitem{milesjma} M. H. Wheeler, Large-amplitude solitary water waves with vorticity,
\emph{SIAM J. Math. Anal.}, \textbf{45} (2013), 2937-2994.

\bibitem{wjfm} M. H. Wheeler, The {F}roude number for solitary water waves with
vorticity, \emph{J. Fluid Mech.} \textbf{768} (2015), 91-112.

\bibitem{warma} M. H. Wheeler, Solitary water waves of large amplitude generated by surface
pressure, \emph{Arch. Ration. Mech. Anal.} \textbf{218} (2015), 1131-1187.

\end{thebibliography}

\end{document}